\documentclass[11pt]{amsart}

\usepackage[colorlinks=true, pdfstartview=FitV, linkcolor=blue, citecolor=blue, urlcolor=blue]{hyperref}
 
\usepackage[letterpaper,margin=1in]{geometry}
\usepackage{enumerate, amsmath, amsfonts, amssymb, amsthm, mathtools, thmtools, wasysym, graphics,
xcolor, frcursive,xparse,comment,bbm}
\usepackage{graphicx, xspace, ifthen, rotating, pict2e}
\usepackage{multirow,multicol,mathtools}
\usepackage{tikz}
\usepackage{ytableau}
\usepackage[enableskew]{youngtab}
\usepackage[cmtip,all]{xy}
\usepackage{young}
\usepackage{subcaption}

\usepackage[margin=1in]{geometry}
\usepackage{float}
\usepackage{pgf,pgfplots}
\usepackage{mathrsfs}
\usetikzlibrary{cd}
\usetikzlibrary{arrows}

\definecolor{darkblue}{rgb}{0.0,0,0.7} 
\definecolor{darkred}{rgb}{0.7,0,0} 
\definecolor{darkgreen}{rgb}{0, .6, 0} 

\definecolor{dtsfsf}{rgb}{0.8274509803921568,0.1843137254901961,0.1843137254901961}
\definecolor{wrwrwr}{rgb}{0.3803921568627451,0.3803921568627451,0.3803921568627451}
\definecolor{cqcqcq}{rgb}{0.7529411764705882,0.7529411764705882,0.7529411764705882}

\newcommand{\defn}[1]{{\color{darkred}\emph{#1}}} 

\usepackage[colorinlistoftodos]{todonotes}

\newcommand{\wt}{\operatorname{wt}}

\newcommand{\len}{\mathsf{len}}

\newcommand{\HH}{\mathcal{H}}

\newcommand{\ext}{\mathsf{ex}}
\newcommand{\res}{\mathsf{res}}
\newcommand{\uncrowd}{\mathsf{uncrowd}}
\newcommand{\row}{\mathsf{row}}
\newcommand{\SVT}{\mathsf{SVT}}
\newcommand{\SSYT}{\mathsf{SSYT}}
\newcommand{\sort}{\mathsf{sort}}

\newtheorem{theorem}{Theorem}[section]
\newtheorem{proposition}[theorem]{Proposition}
\newtheorem{lemma}[theorem]{Lemma}
\newtheorem{corollary}[theorem]{Corollary}

\theoremstyle{definition}

\newtheorem{definition}[theorem]{Definition}
\newtheorem{remark}[theorem]{Remark}
\newtheorem{example}[theorem]{Example}

\numberwithin{equation}{section}

\title{A crystal on decreasing factorizations in the $0$-Hecke monoid}

\author[J.~Morse]{Jennifer Morse}
\address[J. Morse]{Department of Mathematics, University of Virginia, Charlottesville, VA 22904, U.S.A}
\email{morsej@virginia.edu}

\author[J.~Pan]{Jianping Pan}
\address[J. Pan]{Department of Mathematics, UC Davis, One Shields Ave., Davis, CA 95616-8633, U.S.A.}
\email{jpgpan@ucdavis.edu}

\author[W.~Poh]{Wencin Poh}
\address[W. Poh]{Department of Mathematics, UC Davis, One Shields Ave., Davis, CA 95616-8633, U.S.A.}
\email{wpoh@ucdavis.edu}

\author[A.~Schilling]{Anne Schilling}
\address[A. Schilling]{Department of Mathematics, UC Davis, One Shields Ave., Davis, CA 95616-8633, U.S.A.}
\email{anne@math.ucdavis.edu}

\date{\today}

\date{\today}

\begin{document}

\begin{abstract}
We introduce a type $A$ crystal structure on decreasing factorizations of fully-commu\-tative elements in the 0-Hecke 
monoid which we call $\star$-crystal. This crystal is a $K$-theoretic generalization of the crystal on decreasing factorizations 
in the symmetric group of the first and last author. We prove that under the residue map the $\star$-crystal intertwines 
with the crystal on set-valued tableaux recently introduced by Monical, Pechenik and Scrimshaw. We also define 
a new insertion from decreasing factorization to pairs of semistandard 
Young tableaux and prove several properties, such as its relation to the Hecke insertion
and the uncrowding algorithm. The new insertion also intertwines with the crystal operators.
\end{abstract}

\maketitle 

\section{Introduction}

Grothendieck polynomials were introduced by Lascoux and Sch\"utzenberger~\cite{LascouxSchutzenberger.1982,
LascouxSchutzenberger.1983} as representatives for the Schubert classes in the $K$-theory of the flag 
manifold. Their stabilizations were studied by Fomin and Kirillov~\cite{FominKirillov.1994}.
The {stable Grothendieck polynomials}, labeled by permutations $w\in \mathbb{S}_n$, are defined as
\begin{equation}
\label{eq.G}
	\mathfrak{G}_w(x_1,\dots,x_m;\beta) = \sum_{(\mathbf{k},\mathbf{h})} \beta^{\ell(\mathbf{h})-\ell(w)} x^{\mathbf{k}},
\end{equation}
where the sum is over decreasing factorizations $[\mathbf{k},\mathbf{h}]^t$ of $w$ in the 0-Hecke algebra.
When $\beta=0$, $\mathfrak{G}_w$ specializes to the Stanley symmetric function $F_w$ \cite{Stanley.1984}.

A robust combinatorial picture has been developed for the special case of Grothendieck 
polynomials indexed by Grassmannian permutations.  
Buch~\cite{Buch.2002} showed that the Grassmannian Grothendieck polynomials can be realized as
the generating functions of semistandard set-valued tableaux:
\begin{equation}
\label{equation.straight G}
	\mathfrak{G}_\lambda(x_1,\ldots,x_m; \beta) = \sum_{T \in \SVT^m(\lambda)} \beta^{\ext(T)} x^{\wt(T)},
\end{equation}
where $\SVT^m(\lambda)$ is the set of semistandard set-valued tableaux of shape $\lambda$ in the 
alphabet $[m]:=\{1,2,\ldots,m\}$ and $\ext(T)$ is the excess of $T$. Recently,
Monical, Pechenik and Scrimshaw~\cite{MPS.2018} provided a type $A_{m-1}$-crystal structure on $\SVT^m(\lambda)$
which, in particular, implies that
\[
	\mathfrak{G}_\lambda(x_1,\ldots,x_m;\beta) = \sum_\mu \beta^{|\mu|-|\lambda|}
	M_\lambda^\mu \; s_\mu(x_1,\ldots,x_m),
\]
where $M_\lambda^\mu$ is the number of highest-weight set-valued tableaux of weight $\mu$
in the crystal $\SVT^m(\lambda)$. 
Their approach recovers a Schur expansion formula for Grassmannian Grothendieck polynomials
given by Lenart~\cite[Theorem~2.2]{Lenart.2000} in terms of flagged increasing tableaux.

In this paper, we define a type $A$ crystal structure on decreasing factorizations of $w$ in the 0-Hecke algebra
of~\eqref{eq.G}, when $w$ 
is fully-commutative~\cite{Stembridge.1996} (or equivalently $321$-avoiding). A permutation $w$ is fully-commutative if its 
reduced expressions do not contain any braids. The number of fully-commutative elements of $\mathbb{S}_n$ is the 
$n$-th Catalan number. The residue map (see Section~\ref{section.residue map}) shows that 
fully-commutative permutations correspond to skew shapes. We call our crystal $\star$-crystal. It is local in the sense that 
the crystal operators
$f^{\star}_i$ and $e^{\star}_i$ only act on the $i$-th and $(i+1)$-th factors of the decreasing factorization. It generalizes
the crystal of Morse and Schilling~\cite{MorseSchilling.2016} for Stanley symmetric functions (or equivalently
reduced decreasing factorizations of $w$) in the fully-commutative case. We show that the $\star$-crystal and the crystal 
on set-valued tableaux intertwine under the residue map (see Theorem~\ref{thm: star}). We also show that the residue 
map and the Hecke insertion~\cite{BKSRY.2008} are related (see Theorem~\ref{thm: res}), thereby 
resolving~\cite[Open Problem 5.8]{MPS.2018} in the fully-commutative case.
In addition, we provide a new insertion algorithm, which we call $\star$-insertion, from decreasing factorizations on 
fully-commutative elements in the 0-Hecke monoid to pairs of (transposes of) semistandard Young tableaux of the same 
shape (see Definition~\ref{def: new_insertion} and Theorem~\ref{theorem.star insertion bijection}), which intertwines
with crystal operators (see Theorem~\ref{thm: star-crystal.star-insertion}). This recovers the Schur expansion of 
$\mathfrak{G}_w$ of Fomin and Greene~\cite{FominGreene.1998} when $w$ is fully-commutative, stating that
\[
 	\mathfrak{G}_w = \sum_\mu \beta^{|\mu|-\ell(w)} g_w^{\mu} s_\mu,
\]
where 
\[
	g_w^{\mu} = |\{ T \in \mathsf{SSYT}^n(\mu') \mid w_C(T) \equiv w\} |,
\]
and $w_C(T)$ is the column reading word of $T$ (see Remark~\ref{remark.schur expansion}). We also show that the 
composition of the residue map with the $\star$-insertion is related to the uncrowding algorithm~\cite{Buch.2002} (see 
Theorem~\ref{theorem.uncrowding}). Other insertion algorithms have recently been studied
in~\cite{ChanPflueger.2019}.

The paper is organized as follows. In Section~\ref{section.star crystal}, we introduce the $\star$-crystal
on decreasing factorizations in the 0-Hecke monoid and show that it intertwines with the crystal on semistandard 
set-valued tableaux~\cite{MPS.2018} under the residue map.
In Section~\ref{section.insertion}, we discuss two insertion algorithms for decreasing factorizations. 
The first is the Hecke insertion introduced by Buch et al.~\cite{BKSRY.2008} and the second is the new
$\star$-insertion. In Section~\ref{section.properties}, properties of the $\star$-insertion are discussed.
In particular, we prove that it intertwines with the crystal operators and that it relates to the uncrowding algorithm.
We conclude in Section~\ref{section.outlook} with some discussions about the non-fully-commutative case.

\subsection*{Acknowledgments}
We are grateful to Travis Scrimshaw and Zach Hamaker for discussions.
Jennifer Morse and Anne Schilling would like to thank MRSI for hospitality during their stay in July 2019,
where part of this research was carried out.

This material is based upon work supported by the National Security Agency under Grant No. H98230-19-1-0119, 
The Lyda Hill Foundation, The McGovern Foundation, and Microsoft Research, while JM and AS were in residence at
the Mathematical Sciences Research Institute in Berkeley, California, during the summer of 2019.
The first author was partially supported by NSF grant DMS--1833333.
The last author was partially supported by NSF grants DMS--1760329 and DMS--1764153. 

\section{The $\star$-crystal}
\label{section.star crystal}

In this section, we define the $K$-theoretic generalization of the crystal on decreasing factorizations by
Morse and Schilling~\cite{MorseSchilling.2016} when the associated word is fully-commutative. The underlying
combinatorial objects are decreasing factorizations in the 0-Hecke monoid introduced in 
Section~\ref{section.decreasing factorizations}. The $\star$-crystal on these decreasing factorizations
is defined in Section~\ref{section.crystal}. We review the crystal structure on set-valued tableaux introduced by
Monical, Pechenik and Scrimshaw~\cite{MPS.2018} in Section~\ref{section.SVT}. The residue
map and the proof that it intertwines the $\star$-crystal and the crystal on set-valued tableaux is given
in Section~\ref{section.residue map}.

\subsection{Decreasing factorizations in the 0-Hecke monoid}
\label{section.decreasing factorizations}

The \defn{symmetric group} $\mathbb{S}_n$ for $n\geqslant 1$ is generated by the simple transpositions
$s_1,s_2,\dots,s_{n-1}$ subject to the relations
\[
\begin{aligned}
	s_i s_j &= s_j s_i,&& \text{ if } |i-j|>1,\\
	s_i s_{i+1} s_i &= s_{i+1} s_i s_{i+1},&& \text{ for $1\leqslant i<n-1$,}\\
	s^2_i &= 1,&& \text{ for $1\leqslant i \leqslant n-1$.}
\end{aligned}
\]
A \defn{reduced expression} for an element $w\in \mathbb{S}_n$ is a word $a_1 a_2 \dots a_\ell$ with
$a_i \in [n-1]:=\{1,2,\dots,n-1\}$ such that
\begin{equation}
\label{equation.product}
	w=s_{a_1}\cdots s_{a_\ell}
\end{equation}
and $\ell$ is minimal among all words satisfying~\eqref{equation.product}. In this case, $\ell$ is called the length of $w$.

\begin{definition}
The \defn{$0$-Hecke monoid} $\mathcal{H}_0(n)$, where $n\geqslant 1$ is an integer, is the monoid of finite words generated 
by positive integers in the alphabet $[n-1]$ subject to the relations
\begin{equation}
\label{equation.relations H0}
\begin{aligned}
	pq &= qp&& \text{ if } |p-q|>1,\\
	pqp &= qpq&& \text{ for all } p,q,\\
	pp &= p&& \text{ for all } p.
\end{aligned}
\end{equation}
\end{definition}

We may form an equivalence relation $\equiv_{\mathcal{H}_0}$ on all words in the alphabet $[n-1]$ based on the 
relations~\eqref{equation.relations H0}. The equivalence classes are infinite since the last relation changes the length
of the word. We say that a word $a = a_1 a_2\ldots a_\ell$ is \defn{reduced} if $\ell\geqslant 0$ 
is the smallest among all words in $\mathcal{H}_0(n)$ equivalent to $a$. In this case,
$\ell$ is the length of $a$. Note that $\mathcal{H}_0(n)$ is in bijection with $\mathbb{S}_n$ by identifying the reduced word
$a_1 a_2\ldots a_\ell$ in $\mathcal{H}_0(n)$ with $s_{a_1} s_{a_2}\cdots s_{a_\ell}\in \mathbb{S}_n$. We say
$w\in \mathcal{H}_0(n)$ or $\mathbb{S}_n$ is \defn{fully-commutative} or \defn{$321$-avoiding} if
none of the reduced words equivalent to $w$ contain a consecutive braid subword of the form $i\; i+1\; i$ or $i\; i-1\; i$
for any $i\in [n-1]$.

\begin{remark}
Any (not necessarily reduced) word $w\in \mathcal{H}_0(0)$ containing a consecutive braid subword is not fully-commutative.
\end{remark}

\begin{definition}\label{def: decr_fact}
A \defn{decreasing factorization} of $w \in \mathcal{H}_0(n)$ into $m$ \defn{factors} is a product of the form 
\[
	\mathbf{h}=h^m \dots h^2 h^1,
\]
where the sequence in each factor
\[
	h^i = h^i_1h^i_2 \dots h^i_{\ell_i}
\]
is either empty (meaning $\ell_i=0$) or strictly decreasing (meaning $h^i_1>h^i_2>\cdots>h^i_{\ell_i}$) for each 
$1\leqslant i\leqslant m$ and $\mathbf{h}\equiv_{\mathcal{H}_0} w$ in $\mathcal{H}_0(n)$.

The set of all possible decreasing factorizations into $m$ factors is denoted by $\mathcal{H}^m$ or
$\mathcal{H}^m(n)$ if we want to indicate the value of $n$. We call $\ext(\mathbf{h})=\len(\mathbf{h})-\ell$
the \defn{excess} of $\mathbf{h}$, where $\len(\mathbf{h})$ is the number of letters in $\mathbf{h}$ and $\ell$ is the length 
of $w$. We say $\mathbf{h}$ is fully-commutative (or $321$-avoiding) if $w$ is fully-commutative.
\end{definition}

\subsection{The $\star$-crystal}
\label{section.crystal}

Let $\mathcal{H}^{m,\star}$ be the set of fully-commutative decreasing factorizations in $\mathcal{H}^m$. 
We introduce a type $A_{m-1}$ crystal structure on $\mathcal{H}^{m,\star}$, which we call the \defn{$\star$-crystal}.
This generalizes the crystal for Stanley symmetric functions~\cite{MorseSchilling.2016} (see also~\cite{Lenart.2004}).

\begin{definition}\label{def: star}
For any $\mathbf{h} = h^m\dots h^2h^1 \in \mathcal{H}^{m,\star}$, we define crystal operators $e^\star_i$ and $f^\star_i$ for $i\in [m-1]$ and a weight function 
$\wt(\mathbf{h})$. The \defn{weight} function is determined by the length of the factors 
\[
	\wt(\mathbf{h}) = (\len(h^1),\len(h^2),\dots,\len(h^m)).
\]
To define the \defn{crystal operators} $e_i^\star$ and $f_i^\star$, we first describe a pairing process:
	\begin{itemize}
		\item Start with the largest letter $b$ in $h^{i+1}$, pair it with the smallest $a\geqslant b$ in $h^{i}$. 
		If there is no such $a$, then $b$ is unpaired.
		\item The pairing proceeds in decreasing order on elements of $h^{i+1}$ and with each iteration, 
		previously paired letters of $h^{i}$ are ignored.
	\end{itemize}
If all letters in $h^i$ are paired, then $f^\star_i$ annihilates $\mathbf{h}$. Otherwise, let $x$ be 
the largest unpaired letter in $h^i$. The crystal operator $f^\star_i$ acts on $\mathbf{h}$ in either of the following ways:
	\begin{enumerate}
		\item If $x+1 \in h^i \cap h^{i+1}$, then remove $x+1$ from $h^i$, add $x$ to $h^{i+1}$.
		\item Otherwise, remove $x$ from $h^i$ and add $x$ to $h^{i+1}$.
	\end{enumerate}
If all letters in $h^{i+1}$ are paired, then $e^\star_i$ annihilates $\mathbf{h}$. Let $y$ be the smallest unpaired 
letter in $h^{i+1}$. The crystal operator $e^\star_i$ acts on $\mathbf{h}$ in either of the following ways:
	\begin{enumerate}
		\item If $y-1 \in h^i \cap h^{i+1}$, then remove $y-1$ from $h^{i+1}$, add $y$ to $h^i$.
		\item Otherwise, remove $y$ from $h^{i+1}$ and add $y$ to $h^i$.
	\end{enumerate}
It is not hard to see that $e^\star_i$ and $f^\star_i$ are partial inverses of each other.
\end{definition}

\begin{example}
Let $\mathbf{h} = (7532)(621)(6)$, then 
\begin{align*}
	f^\star_1(\mathbf{h})&=0,  & e^\star_1(\mathbf{h}) &= (7532)(62)(61),\\
	f^\star_2(\mathbf{h})&= (75321)(61)(6), & e^\star_2(\mathbf{h})&=(753)(6321)(6).
\end{align*}
\end{example}

\begin{remark}
\label{remark.local}
\mbox{}
Compared to~\cite{MorseSchilling.2016}, one pairs a letter $b$ in $h^{i+1}$ with the smallest letter $a\geqslant b$
in $h^i$ rather than $a>b$.
\end{remark}

\begin{proposition}
Let $\mathbf{h} =h^m \dots h^1\in \mathcal{H}^{m,\star}$ such that $f_i^\star(\mathbf{h})\neq 0$.
Then $f_i^\star(\mathbf{h}) \in \mathcal{H}^{m,\star}$, $f_i^\star(\mathbf{h}) \equiv_{\mathcal{H}_0} \mathbf{h}$, and
$\ext(f_i^\star(\mathbf{h})) = \ext(\mathbf{h})$. Furthermore, the $j$-th factor in $f_i^\star(\mathbf{h})$ and $\mathbf{h}$
agrees for $j\notin \{i,i+1\}$. Analogous statements hold for $e_i^\star$.
\end{proposition}

\begin{proof}
Suppose $\tilde{\mathbf{h}}:=f^\star_i(\mathbf{h}) \neq 0$.
Then by definition of $f^\star_i$, $\tilde{\mathbf{h}} = h^m \dots h^{i+2} \tilde{h}^{i+1}\tilde{h}^ih^{i-1}\dots h^1$ and $h^j$ 
is unchanged for $j\notin \{i,i+1\}$. In addition, the number of factors does not change.
	
\smallskip
To see $\mathbf{h} \equiv_{\mathcal{H}_0}\tilde{\mathbf{h}}$, it suffices to show that 
$h^{i+1}h^i \equiv_{\mathcal{H}_0} \tilde{h}^{i+1}\tilde{h}^i$. Let $x$ be the largest unpaired letter in $h^i$. By 
the bracketing procedure this implies that $x\notin h^{i+1}$. We can write $h^{i+1}$ as $w_1 w_2$, where $w_1$ is a 
word containing only letters greater than $x$, and $w_2$ is a word containing only letters smaller than $x$. We can 
write $h^{i}$ as $w_3 x w_4$, where $w_3$ contains only letters greater than $x$ and $w_4$ contains only letters 
smaller than $x$. 
	
The pairing process will result in one of the two following cases:
\begin{enumerate}
\item	
If $x+1 \in h^i \cap h^{i+1}$, then obtain $\tilde{h}^i$ by removing $x+1$ from $h^i$, and $\tilde{h}^{i+1}$ by adding 
$x$ to $h^{i+1}$.
\item
Otherwise, obtain $\tilde{h}^i$ by removing $x$ from $h^i$ and obtain $\tilde{h}^{i+1}$ by adding $x$ to $h^{i+1}$.
\end{enumerate}

We first argue that in either case we must have $x-1\notin w_2$. Assume $x-1\in w_2$ and let $k$ be the largest number 
such that the interval $[x-k,x-1]\subseteq w_2$. By assumption $k\geqslant 1$. In order for $x$ to be the largest unpaired 
letter in $h^i$, $[x-k,x-1]$ must be contained in $w_4$. We can write $w_2=(x-1)\dots (x-k) w'_2$ and 
$w_4 = (x-1)\dots (x-k)w'_4$, where all letters in $w'_2$ are smaller than $x-k-1$. When $k=1$, we have the following 
subword
\[
	(x-1) w'_2 w_3 x (x-1) \equiv_{\mathcal{H}_0} w'_2 w_3 (x-1)x(x-1),
\]
which contains a braid $(x-1)x(x-1)$. When $k>1$, we also have the following subword
	\[
	(x-k) w'_2 w_3 x (x-1)\dots (x-k+1)(x-k) \equiv_{\mathcal{H}_0} w'_2 w_3 (x-1) \dots (x-k+2)(x-k)(x-k+1)(x-k),
\]	
which also contains a braid. 
	
\smallskip \noindent
\textbf{Case (1):} Let $k$ be the largest letter such that $[x+1,x+k]\subseteq w_3$. Clearly $k\geqslant 1$. 
Suppose $k>1$, then we can write $w_3=w'_3(x+k)\dots (x+1)$. Since $x$ is the largest unpaired letter in $h^i$, 
everything in $[x+1,x+k]\subseteq w_3$ must be paired. The letter $x+1$ in $w_3$ is paired with $x+1\in w_1$, which 
implies that $x+i$ in $w_3$ is paired with $x+i\in w_1$ for all $1\leqslant i\leqslant k$. This implies that 
$[x+1,x+k]\subseteq w_1$. Then we have the following subword 
\[
	(x+1)w_2 w'_3 (x+k)\dots (x+2)(x+1) \equiv_{\mathcal{H}_0} w_2 w'_3 (x+k)\dots (x+1)(x+2)(x+1)
\]
which contains a braid. Thus, we must have $k=1$, which implies that $x+2 \notin w_3$. Write $w_1=w'_1 (x+1)$.
Then by direct computation
\begin{align*}
	h^{i+1} h^i & \equiv_{\mathcal{H}_0} w'_1 (x+1) w_2 w'_3 (x+1) x w_4 \equiv_{\mathcal{H}_0} w'_1(x+1) (x+1) w_2 w'_3 x w_4  \\
	&\equiv_{\mathcal{H}_0}  w'_1(x+1) w_2 w'_3 x x w_4 
	\equiv_{\mathcal{H}_0} \left(w'_1(x+1) x w_2\right) \left(w'_3 x w_4\right) 
	= \tilde{h}^{i+1} \tilde{h}^i.
\end{align*}
	
\smallskip \noindent
\textbf{Case (2):}
We claim that if $x+1\notin h^{i+1}$, then $x+1\notin h^i$. Otherwise the $x+1\in h^i$ must be paired with some 
$z\in h^{i+1}$, so we have $z\leqslant x+1$. But $x$ is unpaired, which implies $z>x$, that gives us a contradiction.	
Hence $x+1\notin w_3$. Recall that $x-1\notin w_2$. Therefore, by a straightforward computation
\[
	h^{i+1} h^i = w_1 w_2 w_3 x w_4 \equiv_{\mathcal{H}_0} \left(w_1 x w_2\right)\left( w_3 w_4\right) 
	\equiv_{\mathcal{H}_0} \tilde{h}^{i+1}\tilde{h}^i.
\]
The above arguments show that $h^{i+1}h^i \equiv_{\mathcal{H}_0} \tilde{h}^{i+1}\tilde{h}^i$, thus
$\mathbf{h} \equiv_{\mathcal{H}_0}\tilde{\mathbf{h}}$, and the total length of the decreasing factorization are unchanged 
under $f^\star_i$. Furthermore, the excess remains unchanged under $f^\star_i$. 
	
Similar arguments hold for $e^\star_i$.
\end{proof}

\begin{remark}
\label{rm: star}
Here we summarize several results from the proof that will be needed later.
Namely, if $x$ is the largest unpaired letter in $h^i$, then
\begin{itemize}
	\item $x-1\notin h^{i+1}$.
	\item One and only one of the three statements hold: $x+1 \in h^{i+1}\cap h^i$,  $x+1\notin h^{i+1} \cup h^i$, 
	and $x+1\in h^{i+1}, x+1\notin h^{i}$.
\end{itemize}
\end{remark}

It will be shown in Section~\ref{section.residue map} that $\mathcal{H}^{m,\star}$ is indeed a Stembridge
crystal of type $A_{m-1}$ (for an introduction to crystal and terminology, see~\cite{BumpSchilling.2017}).

\subsection{The crystal on set-valued tableaux}
\label{section.SVT}

In this section, we review the type $A$ crystal structure on set-valued tableaux introduced in~\cite{MPS.2018}.
In fact, in~\cite{MPS.2018} the authors only considered the crystal structure on straight-shaped set-valued
tableaux. Here we consider the crystal on skew shapes as well, see Theorem~\ref{theorem.SVT crystal}.

We use French notation for partitions $\lambda=(\lambda_1,\lambda_2,\dots)$ with $\lambda_1\geqslant \lambda_2
\geqslant \cdots \geqslant 0$, that is, in the Ferrers diagram for $\lambda$, the largest part $\lambda_1$ is at the bottom.

\begin{definition}[\cite{Buch.2002}]
A \defn{semistandard set-valued tableau} $T$ is the filling of a skew shape $\lambda/\mu$ with nonempty subsets of positive integers such that:
\begin{itemize}
\item for all adjacent cells $A$, $B$ in the same row with $A$ to the left of $B$, we have $\max(A)\leqslant \min(B)$,
\item for all adjacent cells $A$, $C$ in the same column with $A$ below $C$, we have $\max(A)< \min(C)$.
\end{itemize}
The \defn{weight} of $T$, denoted by $\wt(T)$, is the integer vector whose $i$-th component 
counts the number of $i$'s that occur in $T$. The \defn{excess} of $T$ is defined as $\ext(T) = |\wt(T)|-|\lambda|$.
We denote the set of all semistandard set-valued tableaux of shape $\lambda/\mu$ by \defn{$\SVT(\lambda/\mu)$}. 
Similarly, if the maximum entry is restricted to $m$, the set is denoted by $\SVT^m(\lambda/\mu)$.
\end{definition}

We now review the crystal structure on semistandard set-valued tableaux given in~\cite{MPS.2018}. We state the definition on skew shapes rather than just straight shapes.

\begin{definition} \label{def: sksvt}
Let $T\in \SVT^m(\lambda/\mu)$. We employ the following pairing rule for letters $i$ and $i+1$.
Assign $-$ to every column of $T$ containing an $i$ but not an $i+1$. Similarly, assign $+$ to 
every column of $T$ containing an $i+1$ but not an $i$. Then, successively pair each $+$ that is to the left
of and adjacent to a $-$, removing all paired signs until nothing can be paired.

The \defn{operator $f_i$} changes the $i$ in the rightmost 
column with an unpaired $-$ (if this exists) to $i+1$, except if the 
cell $b$ containing that $i$ has a cell to its right, denoted $b^\rightarrow$, that contains both $i$ and $i+1$. In that 
case, $f_i$ removes $i$ from $b^\rightarrow$ and adds $i+1$ to $b$. Finally, if no unpaired $-$ exists, then $f_i$ 
annihilates $T$.

Similarly, the \defn{operator $e_i$} changes the $i+1$ in the leftmost 
column with an unpaired $+$ (if this exists) to $i$, except if the cell $b$ containing that $i+1$ has a cell to its left, 
denoted $b^\leftarrow$, that contains both $i$ and $i+1$. In that case, $e_i$ removes $i+1$ from $b^\leftarrow$ and 
adds $i$ to $b$. Finally, if no unpaired $+$ exists, then $e_i$ annihilates $T$.

Based on the pairing procedure above, $\varphi_i(T)$ is the number of unpaired $-$ while $\varepsilon_i(T)$ is the 
number of unpaired $+$. 
\end{definition}

One can easily show that the crystal on $\SVT^m(\lambda/\mu)$ of Definition~\ref{def: sksvt} defines a seminormal
crystal (for definitions see~\cite{BumpSchilling.2017}).
It was proved in~\cite[Theorem 3.9]{MPS.2018} that the above described operators $e_i$ and $f_i$ define
a type $A_{m-1}$ Stembridge crystal structure on $\SVT^m(\lambda)$. We claim that their proof goes through also for 
skew shapes.

\begin{theorem}
\label{theorem.SVT crystal}
The crystal $\SVT^m(\lambda /\mu)$ of Definition~\ref{def: sksvt} is a Stembridge crystal of type $A_{m-1}$.
\end{theorem}

\begin{proof}
Since the proof is exactly the same as in~\cite[Theorem 3.9]{MPS.2018}, we just state the outline and give a brief 
description. For details we refer to~\cite{MPS.2018}.

First note that the signature rule given by column-reading is compatible with the signature rule given by row-reading 
(top to bottom, left to right, and arrange the letters in the same cell by descending order) by semistandardness.
Hence we may consider the crystal to live inside the tensor product of its rows. A single-row semistandard 
set-valued tableaux of a fixed shape is isomorphic to a Stembridge crystal, as shown in~\cite[Proposition 3.5]{MPS.2018}:
\[
	\Phi_s \colon \SVT^m(s\Lambda_1)\to \bigoplus_{k=1}^{m} B((s-1)\Lambda_1+\Lambda_k),
\] 
where $\Lambda_k$ are the fundamental weights of type $A_{m-1}$.

Let $\lambda = (\lambda_1,\dots,\lambda_{\ell})$ and $\mu = (\mu_1,\dots,\mu_{\ell})$ (the last couple $\mu_i$ could 
be zero) be two partitions such that $\mu \subseteq \lambda$. Construct the map below, which is a strict crystal embedding:
\[
	\Psi \colon \SVT^m(\lambda/\mu) \to \SVT^m((\lambda_1-\mu_1)\Lambda_1) \otimes 
	\SVT^m((\lambda_2-\mu_2)\Lambda_1)\otimes \dots \otimes \SVT^m((\lambda_{\ell}-\mu_{\ell})\Lambda_1).
\]
Thus, we have a strict crystal embedding:
\[
	(\Phi_{\lambda_1-\mu_1}\oplus \dots \oplus \Phi_{\lambda_{\ell}-\mu_{\ell}})\circ \Psi \colon
	\SVT^m(\lambda/\mu)\to \bigotimes_{j=1}^{\ell} 
	\left(\bigoplus_{k=1}^m B((\lambda_j-\mu_j)\Lambda_1+\Lambda_k)\right).
\]
Since $\SVT^m(\lambda/\mu)$ is a seminormal crystal, we can conclude that it is a Stembridge crystal.
\end{proof}

\subsection{The residue map}
\label{section.residue map}

In this section, we define the residue map from set-valued tableaux of skew shape to fully-commutative decreasing
factorizations in the 0-Hecke monoid.  We then show in Theorem~\ref{thm: star} that the residue map intertwines with the 
crystal operators, proving that $\mathcal{H}^{m,\star}$ is indeed a crystal of type $A_{m-1}$ (see 
Corollary~\ref{corollary.crystal}).

\begin{definition}
Given $T\in \SVT^m(\lambda/\mu)$, we define the \defn{residue map} $\res: \SVT^m(\lambda/\mu)\to\mathcal{H}^{m}$ 
as follows. Associate to each cell $(i,j)$ in $\lambda/\mu$ its content $\ell(\lambda)+j-i$, where $\ell(\lambda)$ is the 
number of parts in $\lambda$.
Produce a decreasing factorization $\mathbf{h}=h^m h^{m-1}\ldots h^2 h^1$ by declaring $h^k$ to be the (possibly empty)
sequence formed by taking the contents of all cells in $T$ containing the entry $k$ and then arranging the contents
in decreasing order. This defines $\res(T):=\mathbf{h}$. 
\end{definition}

\begin{example}
Let $T$ be the set-valued tableau of skew shape $(2,2)/(1)$
\[
\ytableausetup{notabloids}
T = 
\begin{ytableau}
23 & 3\\
\none & 12
\end{ytableau}\,.
\]
The content of each cell in $T$ is denoted by a subscript as follows:
\[
\ytableausetup{notabloids} 
\begin{ytableau}
23_1 & 3_2\\
\none & 12_3
\end{ytableau}\,.
\]
To read off the third factor, we search for all cells with an entry 3; these cells have contents 1 and 2, 
so we have 21 in the third factor. Altogether, we obtain $\res(T)=(21)(31)(3) \in \mathcal{H}^3$.
\end{example}

The image of the residue map $\res$ is $\mathcal{H}^{m,\star}$, the set of fully-commutative decreasing factorizations 
into $m$ factors. In fact, $\res$ is a bijection from semistandard set-valued skew tableaux on the alphabet $[m]$ to 
$\mathcal{H}^{m,\star}$ up to shifts in the skew shape.

For this purpose, let us describe the inverse of the residue map.
Let $\mathbf{h}=h^m h^{m-1}\dots h^2 h^1\in \mathcal{H}^{m,\star}$.
Begin by filling the diagonals of content that appear in $h^m$ by the entry $m$. As the resulting $T$ is supposed to be
of skew shape, the cells containing $m$ along increasing diagonals need to go weakly down from left to right. If these 
diagonals are consecutive, then the cells have to be in the same row of $T$ since $T$ is semistandard. Continue the 
procedure above by putting entry $i$ into the diagonals specified by $h^i$ for all $i=m-1,m-2,\dots,1$, applying the condition
that the resulting filling should be semistandard. 

\begin{proposition} \label{prop.res.image}
If $\mathbf{h}=h^m h^{m-1}\dots h^2 h^1\in\mathcal{H}^{m,\star}$, then the above algorithm is well-defined up to
shifts along diagonals. It produces a skew semistandard set-valued tableau $T$ such that $\res(T)=\mathbf{h}$.
\end{proposition}

\begin{proof}
We shall show more generally that at any given stage in the algorithm for the inverse of the residue map above, 
the tableau $T$ produced is of skew shape if and only if $\mathbf{h}$ is fully-commutative.

Assume that $T$ is not of skew shape. Consider the earliest stage in the algorithm when the
produced tableau is not of skew shape. Then, either one of the following cases must have occurred for the first time.

\smallskip \noindent
\textbf{Case 1:} There are adjacent cells with nonempty sets $A$ and $B$ (where $\max(A)\leqslant \min(B)$) in the 
same row on diagonals $i$ and $i+1$ respectively with no cells appearing directly below these cells, as illustrated on 
the left side of Figure \ref{fig: 121.pattern.row}. Moreover, by minimality, we have an integer $x$ with the following 
properties:
	\begin{enumerate}
		\item $i+1\in h^x$ and $x<\min(A)$,
		\item there does not exist a $y$ with $x \leqslant y < \min(B)$ and $i+2\in h^y$.
	\end{enumerate}
By applying semistandardness, a cell containing $x$ is created directly below the cell containing the set $A$ as in 
the right side of Figure \ref{fig: 121.pattern.row}. Furthermore, by (2), for all $x \leqslant y < \min(B)$, we have that every 
letter in $h^y$ is either at most $i+1$ or at least $i+3$. It follows that, after possibly applying commutativity ($i+1$ with 
letters at most $i-1$ or at least $i+3$) and the idempotent relation, $h^{\min(B)}\dots h^{x+1}h^x$ is
equivalent to one containing the braid subword $i+1\; i\; i+1$. This implies that $\mathbf{h}$ is equivalent to a Hecke 
word containing the same braid subword.

\smallskip \noindent
\textbf{Case 2:} There are adjacent cells with nonempty sets $A$ and $B$ in the same column on diagonals $i+1$ 
and $i$ respectively with no cells appearing directly to the left of these cells, as illustrated on the left side of 
Figure~\ref{fig: 212.pattern.column}. Moreover, by minimality, we have an integer $x$ with the following properties:
	\begin{enumerate}
		\item $i\in h^x$ and $x\leqslant\min(A)$,
		\item there does not exist a $y$ with $x < y \leqslant\min(B)$ and $i-1\in h^y$.
	\end{enumerate}
By applying semistandardness, a cell containing $x$ is created directly to the left of the cell containing the set $A$ as 
in the right side of Figure \ref{fig: 212.pattern.column}. Furthermore, by (2), for all $x<y\leqslant \min(B)$, we have that every 
letter in $h^y$ is either at most $i-2$ or at least $i$. Similar to the argument in Case 1, 
$h^{\min(B)}\dots h^{x+1}h^x$ is equivalent to one containing the braid subword $i\; i+1\; i$. This implies that
$\mathbf{h}$ is equivalent to a word in $\mathcal{H}_0(n)$ containing the same braid subword.

The above arguments imply that the image of $\res$ is contained in $\mathcal{H}^{m,\star}$. Conversely, if
$\mathbf{h}$ is fully-commutative, then the algorithm for $\res^{-1}$ does not produce Case 1 or Case 2 above and
hence the resulting tableau $T$ is of skew shape which in turn implies that the algorithm is well-defined (up to shifts
along the diagonal if a gap of size at least 3 occurs in the labels).
\end{proof}

If the skew shape $\lambda/\mu$ of the tableau $T$ is known, then one may simplify the procedure above noting that 
the filling of $i$ specified by letters in $h^i$ must occur along a horizontal strip for all $i=m,m-1,\ldots,1$. In this case,
the recovered tableau $T$ is unique and there is no shift ambiguity if a gap of size at least 3 occurs in the labels.

\begin{figure}[pbt]
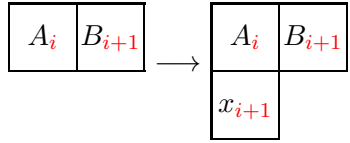

\raisebox{1.2cm}{
\ytableausetup{notabloids, boxsize=2.3em} 
\begin{ytableau}
A_{\color{red}{i}} & B_{\color{red}{i+1}}
\end{ytableau} $\longrightarrow$
\begin{ytableau}
A_{\color{red}{i}} & B_{\color{red}{i+1}}\\
x_{\color{red}{i+1}}
\end{ytableau}
}
\caption{A forbidden case while inverting the residue map.}
\label{fig: 121.pattern.row}
\end{figure}

\begin{figure}[pbt]
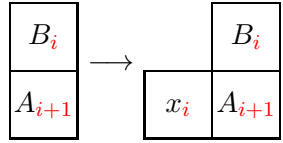

\raisebox{1.2cm}{
\ytableausetup{notabloids, boxsize=2.3em} 
\begin{ytableau}
B_{\color{red}{i}}\\
A_{\color{red}{i+1}}
\end{ytableau} $\longrightarrow$
\begin{ytableau}
\none & B_{\color{red}{i}}\\
x_{\color{red}{i}} & A_{\color{red}{i+1}}
\end{ytableau}
}
\caption{Another forbidden case while inverting the residue map.}
\label{fig: 212.pattern.column}
\end{figure}

\begin{example}\label{eg: inverse_res1}
Let $\mathbf{h}=(61)(752)(75)(762)$ be a decreasing factorization of $w=651762$.

In the algorithm for the inverse of the residue map, the entry 4 is placed on diagonal 1 and 6, respectively.
Due to semistandardness, the entry 3 in diagonal 2 must be placed below the 4 in diagonal 1, 
while the 3's in diagonals 5 and 7 are respectively to the left and below the 4 in diagonal 6.
Continuing with the remaining fillings, we have two possibilities:
\[
T_1=
\raisebox{1.2cm}{
\ytableausetup{notabloids, boxsize=2em} 
\begin{ytableau}
4_{\color{red}{1}}\\
13_{\color{red}{2}}\\
\none & \none & 23_{\color{red}{5}} & 4_{\color{red}{6}}\\
\none & \none & 1_{\color{red}{6}} & 123_{\color{red}{7}}\\
\end{ytableau}}\,,
\]
or
\[
T_2=
\raisebox{1.6cm}{
\ytableausetup{notabloids, boxsize=2em} 
\begin{ytableau}
4_{\color{red}{1}}\\
13_{\color{red}{2}}\\
\none\\
\none & 23_{\color{red}{5}} & 4_{\color{red}{6}}\\
\none & 1_{\color{red}{6}} & 123_{\color{red}{7}}\\
\end{ytableau}}\,,
\]
where $T_1\in\SVT^4((4,4,1,1)/(2,2))$ and $T_2\in\SVT^4((3,3,1,1,1)/(1,1,1))$. Note that they indeed just
differ by a shift along diagonals as stated in Proposition~\ref{prop.res.image}.
\end{example}

\begin{example}\label{eg: inverse_res2}
Let $\mathbf{h}=(8431)(863)(8654)(941)$ be a decreasing factorization of $w=84396541$. 
Suppose that $\mathbf{h}=\res(T)$, where $T\in\SVT^4(\lambda/\mu)$ with $\lambda/\mu=(5,5,4,2,1)/(4,4,1,1)$.

Then, we fill in $4$ along the diagonals with labels 1, 3, 4, 8 respectively, noting that the $4$ in diagonal 4 is to the
right of the $4$ in diagonal $3$ (due to the semistandardness of $T$). Continuing with the remaining fillings, we have 
\[
T=
\raisebox{1.8cm}{
\ytableausetup{notabloids, boxsize=2em} 
\begin{ytableau}
14_{\color{red}{1}}\\
\none & 34_{\color{red}{3}} & 4_{\color{red}{4}}\\
\none & 12_{\color{red}{4}} & 2_{\color{red}{5}} & 23_{\color{red}{6}}\\
\none & \none & \none & \none & 234_{\color{red}{8}}\\
\none & \none & \none & \none & 1_{\color{red}{9}}
\end{ytableau}}\,.
\]
\end{example}

\begin{theorem}
\label{thm: star}
The crystal on set-valued tableaux $\SVT^m(\lambda/\mu)$ and the crystal on decreasing factorizations $\mathcal{H}^{m,\star}$
intertwine under the residue map. That is, the following diagrams commute:
\[
	\begin{tikzcd}
	\SVT^m(\lambda/\mu) \arrow["f_k",d] \arrow[r,"\res"] & \mathcal{H}^{m,\star} \arrow[d,"f^{\star}_k"] \\
	\SVT^m(\lambda/\mu) \arrow[r,"\res"]& \mathcal{H}^{m,\star}
	\end{tikzcd}
	\qquad \qquad \qquad 
	\begin{tikzcd}
	\SVT^m(\lambda/\mu) \arrow["e_k",d] \arrow[r,"\res"] & \mathcal{H}^{m,\star} \arrow[d,"e^{\star}_k"] \\
	\SVT^m(\lambda/\mu) \arrow[r,"\res"]& \mathcal{H}^{m,\star}.
	\end{tikzcd}
\]
\end{theorem}

\begin{proof}
Let $T\in \SVT^m(\lambda/\mu)$, $\mathbf{h}=\res(T)$ and $\ell = \ell(\lambda)$.
We prove the following three statements associated to $f_k(T)$ and $f^{\star}_k(\mathbf{h})$.

\smallskip \noindent
(1) We claim that if there is no unpaired $k$ in $T$, then $f^{\star}_k$ annihilates $\mathbf{h}$.  
Furthermore, if the rightmost unpaired $k$ in cell $b$ of $T$ has content $x$, then $x$ is also the largest unpaired 
letter in $h^k$.

\smallskip

For the proof of (1) it suffices to notice that the signature rule on tableaux is equivalent to the pairing process for 
decreasing factorizations of $\mathcal{H}_0(n)$. We rephrase the pairing procedure for decreasing factorizations on tableaux:
\begin{itemize}
	\item At the beginning, no letter is paired.
	\item Then start with the rightmost column and work westward.
	\item Successively, for each $k+1$, compute its content $a$, then pair it with the $k$ of smallest content weakly greater than $a$ that is yet unpaired. 
\end{itemize}
Next, we argue that the signature rule yields the same result on the rightmost unpaired letter. Assume we are looking 
at cell $b$ containing the current $k+1$ with content $a$.

\smallskip \noindent
\textbf{Case (a):} Suppose there is no unpaired $k$ with content $a$ but at least one unpaired $k$ with  strictly 
greater content(s). Then pair it with the current $k+1$. This is the direct signature rule.

\smallskip \noindent
\textbf{Case (b):} Suppose there is no unpaired $k$ with content weakly greater than $a$, then this $k+1$ is unpaired. 
This is also the direct signature rule.

\smallskip \noindent
\textbf{Case (c):} Suppose there is an unpaired $k$ with content $a$. Then it must be either in the same cell $b$, or one 
row below and one column to the left of $b$ on the diagonal labeled $a$. If they are in the same cell, then the pairing is the
 direct signature rule. 

Otherwise, there must be cells to the left and below $b$ since the shape is skew. Suppose cell $b$ is in row $r$. Consider 
the rightmost entry in cell $(r,j)$ in row $r$ containing a $k+1$, and the leftmost entry in cell $(r-1,q)$ in row $r-1$ 
containing a $k$. Considering this as the first of a consecutive occurrence, cell $b$ is cell $(r,j)$, so we have $\ell+j-r=a$. 
By semistandardness and the condition that the shape is skew, we can partially fill out the involved 
subtableau of $T$ for rows $r-1,r$ from column $q$ to $j$:
\begin{center}
	\begin{tabular}{|l|l|l|l|l|}
		\hline
		$\color{red}{k+1\;}_{\color{red}{\ell+q-r}}$ & $\color{blue}{k+1\;}_{\color{red}{\ell+q+1-r}} $ &
		$\dots$ &
		$\color{purple}{k+1\;}_{\color{red}{\ell+j-1-r}} $ &
		$\color{darkgreen}{k+1\dots}_{\color{red}{\ell+j-r}} $ \\ \hline
		$\dots \textcolor{blue}{k\;}_{\color{red}{\ell+q-r+1}}$ 	& $\color{purple}{k\;}_{\color{red}{\ell+q-r}}$ &
		$\dots$ &
		$\color{darkgreen}{k\;}_{\color{red}{\ell+j-r}}$ &
		$\color{red}{k\;}_{\color{red}{\ell+j-r+1}}$\\ \hline
	\end{tabular}\;.
\end{center}
All the cells $(s,t)$ with $q<t<j$ and $s\in \{r,r-1\}$ and the cells $(r,q)$ and $(r-1,j)$ are single-valued by 
semistandardness as shown in the above figure. 

From the $k+1$ in $(r,j)$, we start the pairing process. First, we claim that the $k$ in cell $(r-1,j)$ must be unpaired at 
this point. Suppose that there is a $k+1$ to the east of cell $(r,j)$ with content smaller or equal to $\ell+j-r+1$, then it must be 
cell $(r,j+1)$, which violates that $(r,j)$ is the rightmost cell in row $r$ containing a $k+1$.
Then the pairing says the $k+1$ in cell $(r,t)$ pairs with the $k$ in cell $(r-1,t-1)$ for $q<t\leqslant j$. Lastly, 
the $k+1$ in cell $(r,q)$ has to pair with the previously unpaired $k$ in cell $(r-1,j)$ since there are no unpaired 
$k$ with label greater or equal to $\ell+q-r$ and smaller than $\ell+j-r+1$.

Although the pairing is different than the usual signature rule pairing, which pairs $k+1,k$ in the same column, the 
$2(j-q+1)$ letters end up being paired. Since it will not influence which one will be the rightmost unpaired letter, it is still 
equivalent to the signature rule.

So in any case, the pairing is equivalent to the signature rule. Thus, the rightmost unpaired $k$ in $T$ corresponds
to the largest unpaired letter in $h^k$.

\smallskip \noindent
(2) We claim that if $f_k$ changes the rightmost unpaired $k$ in $T$ to a $k+1$ (with content $x$) without moving it, 
then $f^{\star}_k$ moves a letter $x$ from $h^k$ to $h^{k+1}$.

Since $f_k$ does not need to move any letter, it means the cell to the right of $b$, denoted by $b^{\rightarrow}$, 
does not contain a $k$. 
It is the only cell with content $x+1$ that could contain a $k$. This implies that $x+1\notin h^k$. 
By Definition~\ref{def: star}, $f^{\star}_k$ moves $x$ from $h^k$ to $h^{k+1}$.

\smallskip \noindent

(3) We claim the following. If $f_k$ changes a $k$ from $b^{\rightarrow}$ into a $k+1$ and moves to cell $b$, 
then $f^{\star}_k$ removes an $x+1$ from $h^k$ and changes it to an $x$ in $h^{k+1}$.

That $f_k$ needs to move a number means that $k$ and $k+1$ are in $b^{\rightarrow}$, which implies that 
$x+1\in h^k \cap h^{k+1}$. By Definition~\ref{def: star}, $f^{\star}_k$ removes the $x+1$ from $h^k$ and adds an $x$ 
to $h^{k+1}$.

We have proved the three statements and they complete the proof that $f_k$ and $f^{\star}_k$ intertwine under the residue
map. The proof is similar for $e_k$ and $e^{\star}_k$.
\end{proof}

\begin{corollary}
\label{corollary.crystal}
The set $\mathcal{H}^{m,\star}$, together with crystal operators $e^{\star}_i$ and $f^{\star}_i$ for $1\leqslant i<m$ 
and weight function $\wt$ defined in Definition~\ref{def: star}, is a Stembridge crystal.
\end{corollary}

\begin{proof}
By Theorem~\ref{thm: star} and the fact that the residue map preserves the weight and is invertible, this follows from the fact
that $\SVT^m(\lambda/\mu)$ is a Stembridge crystal proven in~\cite[Theorem 3.9]{MPS.2018}
(see also Theorem~\ref{theorem.SVT crystal}).
\end{proof}

\begin{example}
Consider the tableau $T$ (with labels in red) given by
\[
\ytableausetup{notabloids,boxsize=2em}
T = 
\begin{ytableau}
3_{\color{red}{1}}\\
1_{\color{red}{2}} & 123_{\color{red}{3}}
\end{ytableau}\,,
\]
with $\res(T)=(31)(3)(32)$.

For the crystal operators on set-valued tableaux we obtain
\[f_1(T) = 
\begin{ytableau}
3_{\color{red}{1}}\\
12_{\color{red}{2}} & 23_{\color{red}{3}}
\end{ytableau}\,,
\]
with $\res\left(f_1(T)\right)=(31)(32)(2)$. Then it can be easily checked that the following diagram commutes:
\[
\ytableausetup{mathmode,boxsize=2em}
\begin{tikzpicture}
\node (T) at (0,0) 
{$T = \begin{ytableau}
3_{\color{red}{1}} \\ 
1_{\color{red}{2}} & 123_{\color{red}{3}}
\end{ytableau}$};
\node (h) at (5,0) {$(31)(3)(32)$};
\node (T') at (0,-3) 
{$f_1(T) = \begin{ytableau}
3_{\color{red}{1}} \\
12_{\color{red}{2}} & 23_{\color{red}{3}}
\end{ytableau}$};
\node (h') at (5,-3) {$(31)(32)(2)$.};
\draw [->] (T) edgenode[above]{$\res$} (h);
\draw [->] (T') edgenode[above]{$\res$} (h');
\draw [->] (T) edgenode[right]{$f_1$} (T');
\draw [->] (h) edgenode[right]{$f_1^\star$} (h');
\end{tikzpicture}
\]
\end{example}

\section{Insertion algorithms}
\label{section.insertion}

In this section, we discuss two insertion algorithms for decreasing factorizations in $\mathcal{H}^m$ (resp. 
$\mathcal{H}^{m,\star}$).
The first is the Hecke insertion introduced by Buch et al.~\cite{BKSRY.2008}, which we review in
Section~\ref{section.hecke insertion}. We prove a relationship between Hecke insertion and the residue map 
(see Theorem~\ref{thm: res}). In particular, this proves~\cite[Open Problem 5.8]{MPS.2018} for fully-commutative
permutations. 
The second insertion is a new insertion, which we call $\star$-insertion, introduced in
Section~\ref{section.star insertion}. It goes from fully-commutative decreasing factorizations in the 0-Hecke monoid to pairs of 
(transposes of) semistandard tableaux of the same shape and is well-behaved with respect to the crystal operators.

\subsection{Hecke insertion}
\label{section.hecke insertion}

\ytableausetup{boxsize=1.5em}
Hecke insertion was first introduced in~\cite{BKSRY.2008} as column insertion. Here we state the row insertion version 
as in \cite{PP.2016}. In this section, we represent a decreasing factorization $\mathbf{h} = h^m h^{m-1} \cdots h^1$,
where $h^i = h^i_1 h^i_2 \ldots h^i_{\ell_i}$, by a \defn{decreasing Hecke biword}
\[
\begin{bmatrix}
\mathbf{k}\\ 
\mathbf{h}
\end{bmatrix}
=
\begin{bmatrix}
m & \ldots & m & \ldots & 1 & \ldots &1 \\
h^m_1 & \ldots & h^m_{\ell_m} & \ldots & h^1_1 & \ldots & h^1_{\ell_1}
\end{bmatrix}.
\]
In addition, we say that $[\mathbf{k},\mathbf{h}]^t$ is \defn{fully-commutative} if $\mathbf{h}$ 
is fully-commutative. 

\begin{example}\label{example.biword}
Consider the decreasing factorization $\mathbf{h}=(1)(2)(31)(\;)(32)$. Then the corresponding biword 
$[\mathbf{k},\mathbf{h}]^t$ is
\[
\begin{bmatrix}
\mathbf{k}\\ 
\mathbf{h}
\end{bmatrix}
=
\begin{bmatrix}
5 & 4 & 3 & 3 & 1 & 1 \\
1 & 2 & 3 & 1 & 3 & 2
\end{bmatrix}.
\]
\end{example}

\begin{definition}
\label{definition.hecke insertion}
Starting with a decreasing Hecke biword $[\mathbf{k},\mathbf{h}]^t$, we define \defn{Hecke row insertion} from the right. 
The insertion sequence is read from right to left. Suppose there are $n$ columns in $[\mathbf{k},\mathbf{h}]^t$.

Start the insertion with $(P_0,Q_0)$ being both empty tableaux. We recursively construct $(P_{i+1},Q_{i+1})$ from 
$(P_i,Q_i)$. Suppose the $(n-i)$-th column in $[\mathbf{k},\mathbf{h}]^t$ is $[y,x]^t$.

We describe how to insert $x$ into $P_i$, denoted  $P_i \leftarrow x$, by describing how to insert $x$ into a row $R$. 
The insertion may modify the row and may produce an output integer, which will be inserted into the next row. First, 
we insert $x$ into the first row $R$ of $P_i$ following the rules below:
\begin{enumerate}
	\item If $x\geqslant z$ for all $z\in R$, the insertion terminates in either of the following ways:
		\begin{enumerate}
			\item If we can append $x$ to the right of $R$ and obtain an increasing tableau, the result $P_{i+1}$ 
			is obtained by doing so; form $Q_{i+1}$ by adding a box with $y$ in the same position where $x$ is 
			added to $P_i$.
			\item Otherwise row $R$ remains unchanged. Form $Q_{i+1}$ by adding $y$ to the existing corner of 
			$Q_i$ whose column contains the rightmost box of row $R$.
		\end{enumerate}
	\item Otherwise, there exists a smallest $z$ in $R$ such that $z>x$.
		\begin{enumerate} 
			\item If replacing $z$ with $x$ results in an increasing tableau, then do so. Let $z$ be the output integer 
			to be inserted into the next row.
			\item Otherwise, row $R$ remains unchanged. Let $z$ be the output integer to be inserted into the next row.
		\end{enumerate}
\end{enumerate}
The entire Hecke insertion terminates  at $(P_n,Q_n)$ after we have inserted every letter from the Hecke biword. 
The resulting insertion tableau $P_n$ is an increasing tableau, meaning that both rows and columns of $P_n$
are strictly increasing.
If $\mathbf{k}=(n,n-1,\dots,1)$, the recording tableau $Q_n$ is a standard set-valued tableau.
\end{definition}

\begin{example}\label{eg: hecke_insertion}
Take $[\mathbf{k},\mathbf{h}]^t$ from Example~\ref{example.biword}. Following the Hecke row insertion, we compute 
its insertion tableau and recording tableau:
\[
\begin{aligned}
\emptyset
& \rightarrow &
\ytableausetup{notabloids}
\raisebox{-0.5cm}{
\begin{ytableau}
2
\end{ytableau}}
& \rightarrow &
\ytableausetup{notabloids}
\raisebox{-0.5cm}{
\begin{ytableau}
2 & 3
\end{ytableau}}
& \rightarrow &
\ytableausetup{notabloids}
\begin{ytableau}
2& \none \\
1 &3
\end{ytableau}
& \rightarrow &
\ytableausetup{notabloids}
\begin{ytableau}
2 & \none \\
1 & 3
\end{ytableau}
& \rightarrow &
\ytableausetup{notabloids}
\begin{ytableau}
2 & 3 \\
1 & 2
\end{ytableau}
& \rightarrow &
\ytableausetup{notabloids}
\begin{ytableau}
3 \\
2 & 3 \\
1 & 2
\end{ytableau} = P,\\
\emptyset
& \rightarrow &
\ytableausetup{notabloids}
\raisebox{-0.5cm}{
\begin{ytableau}
1
\end{ytableau}}
& \rightarrow &
\ytableausetup{notabloids}
\raisebox{-0.5cm}{
\begin{ytableau}
1 & 1
\end{ytableau}}
& \rightarrow &
\ytableausetup{notabloids}
\begin{ytableau}
3 & \none \\
1 & 1
\end{ytableau}
& \rightarrow &
\ytableausetup{notabloids}
\begin{ytableau}
3 & \none \\
1 & 13
\end{ytableau}
& \rightarrow &
\ytableausetup{notabloids}
\begin{ytableau}
3 & 4 \\
1 & 13
\end{ytableau}
& \rightarrow &
\ytableausetup{notabloids}
\begin{ytableau}
5\\
3 & 4 \\
1 & 13
\end{ytableau} = Q.
\end{aligned}
\]
\end{example}

\begin{example}
\label{example.destandardize}
Note that the recording tableau for the Hecke insertion of Definition~\ref{definition.hecke insertion} is not
always a semistandard set-valued tableau. For example, for $\mathbf{h}=(21)(41)$ we have
\[
\begin{bmatrix}
\mathbf{k}\\ 
\mathbf{h}
\end{bmatrix}
=
\begin{bmatrix}
2 & 2 & 1 & 1 \\
2 & 1 & 4 & 1 
\end{bmatrix}
\]
and
\[
	P= \begin{ytableau} 4\\ 1&2 \end{ytableau} \qquad \text{and} \qquad
	Q= \begin{ytableau} 22\\ 1&1 \end{ytableau}.
\]
However, in Theorem~\ref{thm: res} below we will see that in certain cases it is.
\end{example}

\begin{theorem}\label{thm: res}
Let $T \in \SVT(\lambda)$ and $[\mathbf{k},\mathbf{h}]^t= \res(T)$. 
Apply Hecke row insertion from the right on $[\mathbf{k},\mathbf{h}]^t$ to obtain the pair of tableaux $(P,Q)$. 
Then $Q = T$.
\end{theorem}

\begin{remark}
Combining Theorems~\ref{thm: res} and ~\ref{thm: star} shows that Hecke insertion from right to left (as opposed to 
left to right in \cite{PP.2016})
intertwines the crystal on set-valued tableaux and the $\star$-crystal, even though in general it is not always well-defined 
(see Example~\ref{example.destandardize}).
This resolves~\cite[Open Problem 5.8]{MPS.2018} when the decreasing factorizations are fully-commutative.
Even when $\mathbf{h}$ is fully-commutative, but does not correspond to a straight-shaped tableau under $\mathsf{res}^{-1}$
as in Example~\ref{example.destandardize}, one can fill the skew part with small enough numbers and apply the Hecke
insertion on this tableau. In the above example
\[
\begin{bmatrix}
\mathbf{k}\\ 
\mathbf{h}
\end{bmatrix}
=
\begin{bmatrix}
2 & 2 & 1 & 1 & 0 & 0\\
2 & 1 & 4 & 1 & 3 & 2
\end{bmatrix}
\qquad \text{with} \qquad 
	Q = T =  \begin{ytableau} 12 & 2\\ 0&0&1 \end{ytableau}.
\]
Note, however, that unlike in~\cite{MPS.2018} we use row Hecke insertion from right to left rather than column
insertion from left to right (in analogy to~\cite{MorseSchilling.2016} for Edelman--Greene insertion).
\end{remark}

Since $k\in T(i,j)$ if and only if $\ell+j-i \in h^k$ under the residue map, where $\ell=\ell(\lambda)$ and $h^k$ is the 
$k$-th factor of $\mathbf{h}$, the statement of Theorem~\ref{thm: res} is equivalent to applying Hecke insertion on the 
entries of $T$ sorted first by ascending order of entries, followed by ascending diagonal content.

\begin{example}
Let $T$ be the semistandard set-valued tableau
\[
\ytableausetup{notabloids}
T = 
\begin{ytableau}
2_{\color{red}{1}} & 4_{\color{red}{2}}\\
1_{\color{red}{2}} & 23_{\color{red}{3}}
\end{ytableau}\;.
\]	
The insertion sequence by entry is listed in the table below:
\begin{table}[h!]
	\begin{tabular}{lccccc}
		Cell    & (1,1) & (2,1) & (1,2) & (1,2) & (2,2) \\
		Content & 2     & 1     & 3     & 3     & 2  \\
		Entry & 1     & 2     & 2     & 3     & 4      
	\end{tabular}
\end{table}

We will prove Theorem~\ref{thm: res} by induction by considering all subtableaux of $T$, obtained by adding
the entries in $T$ one by one in the order above:
\[\begin{aligned}
\emptyset 
& \rightarrow &
\raisebox{-0.5cm}{
\ytableausetup{notabloids}
\begin{ytableau}
1_{\color{red}{2}}
\end{ytableau}}
& \rightarrow &
\ytableausetup{notabloids}
\begin{ytableau}
2_{\color{red}{1}} \\
1_{\color{red}{2}} 
\end{ytableau}
\ytableausetup{notabloids}
& \rightarrow &
\begin{ytableau}
2_{\color{red}{1}} & \none \\
1_{\color{red}{2}} & 2_{\color{red}{3}}
\end{ytableau}
& \rightarrow &
\ytableausetup{notabloids}
\begin{ytableau}
2_{\color{red}{1}} & \none \\
1_{\color{red}{2}} & 23_{\color{red}{3}}
\end{ytableau}
& \rightarrow &
\ytableausetup{notabloids}
\begin{ytableau}
2_{\color{red}{1}} & 4_{\color{red}{2}} \\
1_{\color{red}{2}} & 23_{\color{red}{3}}
\end{ytableau} = T.
\end{aligned}
\]	
In addition, the corresponding sequence of insertion tableaux and recording tableaux is listed here:
\[
\begin{aligned}
\emptyset
& \rightarrow &
\ytableausetup{notabloids}
\raisebox{-0.5cm}{
\begin{ytableau}
2
\end{ytableau}}
& \rightarrow &
\ytableausetup{notabloids}
\begin{ytableau}
2 \\
1
\end{ytableau}
& \rightarrow &
\ytableausetup{notabloids}
\begin{ytableau}
2 & \none \\
1 & 3
\end{ytableau}
& \rightarrow &
\ytableausetup{notabloids}
\begin{ytableau}
2 & \none \\
1 & 3
\end{ytableau}
& \rightarrow &
\ytableausetup{notabloids}
\begin{ytableau}
2 & 3 \\
1 & 2
\end{ytableau} = P.\\
\emptyset
& \rightarrow &
\ytableausetup{notabloids}
\raisebox{-0.5cm}{
\begin{ytableau}
1
\end{ytableau}}
& \rightarrow &
\ytableausetup{notabloids}
\begin{ytableau}
2 \\
1
\end{ytableau}
& \rightarrow &
\ytableausetup{notabloids}
\begin{ytableau}
2 & \none \\
1 & 2
\end{ytableau}
& \rightarrow &
\ytableausetup{notabloids}
\begin{ytableau}
2 & \none \\
1 & 23
\end{ytableau}
& \rightarrow &
\ytableausetup{notabloids}
\begin{ytableau}
2 & 4 \\
1 & 23
\end{ytableau} = Q.
\end{aligned}
\]	
\end{example}

\begin{proof}[Proof of Theorem \ref{thm: res}]
We prove the theorem by proving the following more specific statement.

For a given step in the insertion process, suppose that the entries of $T$ that are involved so far form a nonempty 
subtableau $T'$ of $T$ with shape $\mu$ containing cell $(1,1)$, and the insertion tableau and recording tableau at the 
corresponding step are $P(T')$ and $Q(T')$. Then, they both have shape $\mu$, and the entry of cell $(i,j)$ of $P(T')$ is $\ell+j-\mu'_j+i-1$, and  
$Q(T') = T'$, where $\mu'$ is the transpose of the partition $\mu$ and $\ell:= \lambda'_1=\ell(\lambda)$.

We prove this by induction on subtableaux of $T$. 

\vspace{2pt}

\noindent
\textbf{Base step:} Suppose $T'$ only contains a single cell $(1,1)$ and $T'(1,1)=S$, where $S$ is a subset of $T(1,1)$ 
with cardinality $d$. Then $P(T')$ is obtained by inserting $d$ times the number $\ell$. So we have 
$P(T') = \begin{ytableau}\ell\end{ytableau}$ and $Q(T') = T'$. Here $\mu = (1)$, so for $(i,j)=(1,1)$, we have 
$\ell +j-\mu'_j+i-1= \ell$. 

\vspace{2pt}

\noindent
\textbf{Inductive step:} Suppose that the statements hold for some subtableau $T'$ of shape $\mu$. Assume the next 
insertion step involves adding the entry $k$ in cell $(p,q)$ of $T$ to $T'$ to obtain $T''$. There are two cases: (1) the 
cell $(p,q)$ is already in $T'$, or (2) the cell $(p,q)$ is not in $T'$.

\vspace{2pt}

\noindent
\textbf{Case (1):} We must have $(p,q)$ to be an inner corner of $T'$ (no cell is to its right or above it), so $p=\mu'_q$ 
and $p> \mu'_{q+1}$. In this case, $k$ is recorded in $Q(T')$. Then by the induction on $T'$, every cell $(i,j)$ of $P(T')$ 
has value $\ell+j-\mu'_j+i-1$. To determine the insertion path of $P(T') \leftarrow \ell+q-p$, we compute the columns $q$ 
and $q+1$ of $P(T')$ as follows:
 
 \smallskip
 
\begin{center}
	\begin{tabular}{ccc}
		row number                          & $q$-th column                        & $(q+1)$-st column                           \\ \cline{2-2} 
		\multicolumn{1}{c|}{p}              & \multicolumn{1}{c|}{$\ell+q-1$}     &                                            \\ \cline{2-2}
		\multicolumn{1}{c|}{}               & \multicolumn{1}{c|}{$\vdots$}            &                                            \\ \cline{2-3}
		\multicolumn{1}{c|}{$\mu'_{q+1}<p$} & \multicolumn{1}{c|}{$\ell+q-p+\mu'_{q+1}-1$} & 
		\multicolumn{1}{c|}{$\ell+q$}   \\ \cline{2-3} 
		\multicolumn{1}{c|}{}               & \multicolumn{1}{c|}{$\vdots$}            & \multicolumn{1}{c|}{$\vdots$}                   \\
		 \cline{2-3} 
		\multicolumn{1}{c|}{2}              & \multicolumn{1}{c|}{$\ell+q-p+1$}   & \multicolumn{1}{c|}{$\ell+q+2-\mu'_{q+1}$} \\
		 \cline{2-3} 
		 \multicolumn{1}{c|}{1}              & \multicolumn{1}{c|}{$\ell+q-p$}     & \multicolumn{1}{c|}{$\ell+q+1-\mu'_{q+1}$} \\ 
		\cline{2-3} 
	\end{tabular}
\end{center}

Following Case 2(b) of Hecke insertion, the insertion path is vertically up column $q+1$. At the top of the column, 
$\ell+q$ is inserted into row $\mu'_{q+1}+1$. Furthermore, $\ell+q$ is greater than $\ell+q-p+\mu'_{q+1}$ in cell 
$(\mu'_{q+1}+1,q)$ because $p > \mu'_{q+1}$. By Hecke insertion Case 1(b), the insertion ends in row $\mu'_{q+1}+1$. 
Also $P(T')$ is unchanged, and $k$ is recorded in cell $(p,q)$ of $Q(T')$ since it is the corner whose column contains the 
rightmost box of row $\mu'_{q+1}+1$. In this case, we get $Q(T'')=T''$. Since the shape $\mu$ is unchanged, we have 
that $P(T'')=P(T')$ also satisfies the statement. 
\vspace{2pt}

\noindent
\textbf{Case (2):} If cell $(p,q)$ is not in $T'$, then it must be an outer corner of $T'$, so $\mu'_q=p-1$ and 
$\mu'_{q-1}>p-1$. Specifically, two cases can happen:
(a) $p=1$ and $(1,q-1)\in T'$,
(b) both $(p-1,q),(p,q-1) \in T'$, or $q = 1$ and $(p-1,1)\in T$.

\smallskip

\noindent
\textbf{Case 2(a):} The first row of $P(T')$ is $\ell+1-\mu'_1, \dots, \ell+j-\mu'_j,\dots,\ell+(q-1)-\mu'_{q-1}$. Since 
$\ell+q-p = \ell+q-1 > \ell+(q-1)-\mu'_{q-1}$, it is appended to the end of the first row which is the cell $(1,q)$. The letter 
$k$ is recorded in the same new cell of $Q(T')$. In this case, the only entry in $P$ that is changed is $(1,q)$, and its 
entry $\ell+q-1$ satisfies the statement. Also $Q(T'')$ equals $T''$. 
\smallskip 

\noindent
\textbf{Case 2(b):} Since entry $(i,q-1)$ of $P(T')$ is $\ell+q-1-\mu'_{q-1}+i-1$ and entry $(i,q)$ of $P(T')$ is 
$\ell+q-\mu'_q+i-1$, the number $q-p+\ell$ is in-between the two when $i=1$. So the insertion starts by bumping 
$(1,q)$. To get the insertion path, we compute columns $q-1$ and $q$ as follows:

\smallskip

\begin{center}
	\begin{tabular}{ccc}
		row number                        & $(q-1)$-st column                               & $q$-th column                      \\ \cline{2-2} 
		\multicolumn{1}{c|}{$\mu'_{q-1}$} & \multicolumn{1}{c|}{$\ell+q-2$}              &                                   \\ \cline{2-2}
		\multicolumn{1}{c|}{}             & \multicolumn{1}{c|}{...}                     &                                   \\ \cline{2-3}
		\multicolumn{1}{c|}{$p-1$}          & \multicolumn{1}{c|}{$\ell+q+p-\mu'_{q-1}-3$} & \multicolumn{1}{c|}{$\ell+q-1$}  
		\\ \cline{2-3} 
		\multicolumn{1}{c|}{}             & \multicolumn{1}{c|}{...}                     & \multicolumn{1}{c|}{...}          \\ \cline{2-3} 
		\multicolumn{1}{c|}{2}            & \multicolumn{1}{c|}{$\ell+q-\mu'_{q-1}$}     & \multicolumn{1}{c|}{$\ell+q-p+2$}   \\ 
		\cline{2-3} 
		\multicolumn{1}{c|}{1}            & \multicolumn{1}{c|}{$\ell+q-1-\mu'_{q-1}$}   & \multicolumn{1}{c|}{$\ell+q-p+1$} \\ 
		\cline{2-3} 
	\end{tabular}
\end{center}

By Hecke insertion Case 2(a), $\ell+q-p$ is placed in cell $(1,q)$ and the original column $q$ is shifted one position 
higher. By Hecke insertion Case 1(a), the insertion terminates at row $p$ and the original entry in cell $(p-1,q)$ is 
appended at the rightmost box of row $p$. Thus, $\mu'_{q}$ increases by 1. The updated entries in column $q$ still
satisfy the statement. Since the entries in other columns of $P(T')$ are unchanged and $\mu'_j$ is unchanged for 
$j\neq q$, they also satisfy the statement. So we have $P(T'')$ satisfies the statement. The letter $k$ is inserted 
into the new cell $(p,q)$ of $Q(T')$, which makes $Q(T'')=T''$.

Thus, the statement holds, proving the theorem.
\end{proof}

\subsection{The $\star$-insertion}
\label{section.star insertion}

We define a new insertion algorithm, which we call $\star$-insertion, from fully-commutative decreasing Hecke biwords
$[\mathbf{k},\mathbf{h}]^t$ to pairs of tableaux $P$ and $Q$, denoted by $\star([\mathbf{k},\mathbf{h}]^t)=(P,Q)$, 
as follows.

\begin{definition}
\label{def: new_insertion}
Fix a fully-commutative decreasing Hecke biword $[\mathbf{k},\mathbf{h}]^t$. The insertion is done by reading the columns of
this biword from right to left.

Begin with $(P_0, Q_0)$ being a pair of empty tableaux.
For every integer $i\geqslant 0$, we recursively construct $(P_{i+1}, Q_{i+1})$ from $(P_i, Q_i)$ as follows.
Let $[q,x]^t$ be the $i$-th column (from the right) of $[\mathbf{k},\mathbf{h}]^t$.
Suppose that we are inserting $x$ into row $R$ of $P_i$.
\begin{description}
	\item [Case 1] If $R$ is empty or $x>\max(R)$, then form $P_{i+1}$ by appending $x$ to row $R$ and form 
	$Q_{i+1}$ by adding $q$ in the corresponding position to $Q_i$. Terminate and return $(P_{i+1},Q_{i+1})$.
	\item [Case 2] Otherwise, if $x\notin R$, locate the smallest $y$ in $R$ with $y>x$. Bump $y$ with $x$ and insert 
	$y$ into the next row of $P_i$.
	\item [Case 3] Otherwise, if $x\in R$, locate the smallest $y$ in $R$ with $y\leqslant x$ and interval $[y,x]$
	contained in $R$. Row $R$ remains unchanged and $y$ is to be inserted into the next row of $P_i$.
\end{description}
Denote $(P,Q)=( P_\ell,Q_\ell)$ if $[\mathbf{k},\mathbf{h}]^t$ has length $\ell$. We define the \defn{$\star$-insertion}
by $\star([\mathbf{k},\mathbf{h}]^t)=(P,Q)$.

Furthermore, denote by $P\leftarrow x$ the tableau obtained by inserting $x$ into $P$.
The collection of all cells in $P\leftarrow x$, where insertion or bumping has occurred is called the \defn{insertion path} 
for $P\leftarrow x$. In particular, in Case 1 the newly added cell is in the insertion path, in Case 2 the cell containing the
bumped letter $y$ is in the insertion path, and in Case 3 the cell containing the same entry as the inserted letter
is in the insertion path.
\end{definition}

\begin{example}
Let
\setcounter{MaxMatrixCols}{15}
\[
\begin{bmatrix}
\mathbf{k}\\
\mathbf{h}
\end{bmatrix} =
\begin{bmatrix}
4 & 4  & 2 & 2 & 1 & 1\\
4 & 2 & 4 & 2 & 3 & 1
\end{bmatrix}.
\]

The corresponding sequence of insertion tableaux and recording tableaux under the $\star$-insertion is listed here:

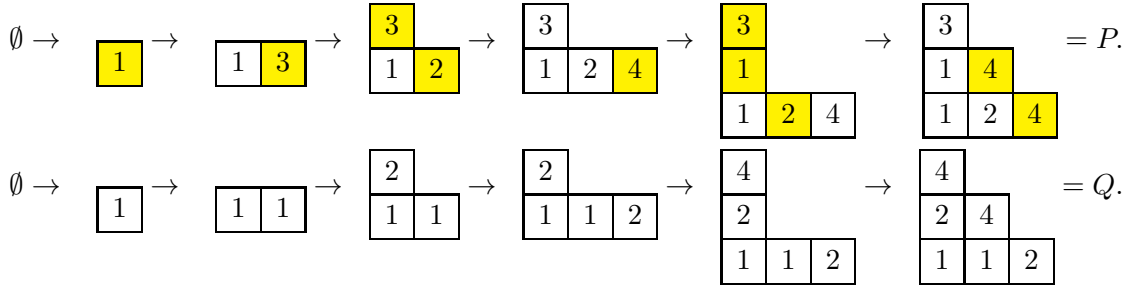
\[
\begin{aligned}
\emptyset
& \rightarrow &
\ytableausetup{notabloids}
\raisebox{-0.5cm}{
	\begin{ytableau}
	*(yellow)1
	\end{ytableau}}
& \rightarrow &
\ytableausetup{notabloids}
\raisebox{-0.5cm}{
\begin{ytableau}
1 & *(yellow)3
\end{ytableau}}
& \rightarrow &
\ytableausetup{notabloids}
\begin{ytableau}
*(yellow)3 & \none \\
1 & *(yellow)2
\end{ytableau}
& \rightarrow &
\ytableausetup{notabloids}
\begin{ytableau}
3 & \none \\
1 & 2 & *(yellow)4
\end{ytableau}
& \rightarrow &
\ytableausetup{notabloids}
\begin{ytableau}
*(yellow)3 \\
*(yellow)1 \\
1 & *(yellow)2 & 4
\end{ytableau} 
& \rightarrow &
\ytableausetup{notabloids}
\begin{ytableau}
3 \\
1 & *(yellow)4\\
1 & 2 & *(yellow)4
\end{ytableau}= P.\\
\emptyset
& \rightarrow &
\ytableausetup{notabloids}
\raisebox{-0.5cm}{
	\begin{ytableau}
	1
	\end{ytableau}}
& \rightarrow &
\ytableausetup{notabloids}
\raisebox{-0.5cm}{
\begin{ytableau}
1 & 1
\end{ytableau}}
& \rightarrow &
\ytableausetup{notabloids}
\begin{ytableau}
2 & \none \\
1 & 1
\end{ytableau}
& \rightarrow &
\ytableausetup{notabloids}
\begin{ytableau}
2 & \none \\
1 & 1 & 2
\end{ytableau}
& \rightarrow &
\ytableausetup{notabloids}
\begin{ytableau}
4\\
2\\
1 & 1 & 2
\end{ytableau}
& \rightarrow &
\ytableausetup{notabloids}
\begin{ytableau}
4\\
2 & 4\\
1 & 1 & 2
\end{ytableau} = Q.
\end{aligned}
\]	
Then we have $\star([\mathbf{k},\mathbf{h}]^t)=(P,Q)$, and the cells in the insertion paths at each step are highlighted in yellow.
\end{example}

\begin{lemma}\label{lem: star.insertion}
Let $[\mathbf{k},\mathbf{h}]^t$ be a fully-commutative decreasing Hecke biword.
Suppose that $\star([\mathbf{k},\mathbf{h}]^t)=(P,Q)$. 
Then, the following statements hold:
\begin{enumerate}
\item $P^t$ is semistandard and $Q$ has the same shape as $P$.
\label{lem: P^t.semistd.P.Q.same.shape}
\item  Let $x$ be an integer such that 
$x\cdot\mathbf{h}$ is fully-commutative. Then the insertion path for $P\leftarrow x$ goes weakly to the left. 
\label{lem: insertion.path}
\end{enumerate}
\end{lemma}

\begin{proof}
We will prove \eqref{lem: P^t.semistd.P.Q.same.shape} by induction on the number of cells of $P$. 
Statement~\eqref{lem: insertion.path} will follow by some results in the proof of 
statement~\eqref{lem: P^t.semistd.P.Q.same.shape}.

Consider the leftmost column $[q,x]^t$ of $[\mathbf{k},\mathbf{h}]^t$ and let $[\mathbf{k'},\mathbf{h'}]^t$ be the Hecke 
biword formed by taking the remaining columns in the same order.
If the $\star$-insertion of $[\mathbf{k'},\mathbf{h'}]^t$ yields $(P',Q')$, note that we have $P=P'\leftarrow x$. 
For all integers $j\geqslant 1$, denote by $R_j$ the (possibly empty) $j$-th row of $P'$.
Denote by $u$ the entry to be inserted into $R_j$ and $B_j$ as the cell in the insertion path at $R_j$, where 
$1\leqslant j\leqslant k$. Additionally, if bumping occurs at $R_j$, denote the entry bumped out as $y$.

\smallskip
\noindent \textbf{(1)}
We will prove that if $(P')^t$ is semistandard, then the transpose of the updated tableau is semistandard. 
\begin{description}
\item[Case (a)] Suppose that the insertion terminates at $R_1$. 
Then Case 1 of the $\star$-insertion has occurred, with a cell containing $x$ appended at the end of the row. 
If $R_1$ is nonempty, then $x>\max(R_1)$. Additionally, as $(P')^t$ is semistandard, integers strictly increase along 
$R_1$ but weakly increase along the column containing $B_1$. Hence, the transpose of the resulting tableau $P$ is 
semistandard.
\item[Case (b)] Suppose that insertion terminates at $R_k$, where $k>1$. 
We will show that for all $1\leqslant j\leqslant k$, the changes introduced at row $R_j$ of $P'$ 
maintain the property that the transpose of the updated tableau is semistandard.
\item[Case (b)(i)] Suppose that $j=k$. 
In this case, a new cell containing $u$ is appended at the end 
of $R_k$ and $u>\max(R_k)$ if the row is nonempty, proving that the integers increase strictly along $R_k$.

If Case 2 occurs at $R_{k-1}$, then $u$ is the entry bumped out of $R_{k-1}$ with the property that when $u'$ 
is inserted into $R_{k-1}$, $u\in R_{k-1}$ is the smallest entry with $u>u'$.
Let $z$ be the entry below cell $B_k$. We claim that $z\leqslant u$. 
If we assume instead that $z>u$, then the cell containing $z$ is strictly to the right of $B_{k-1}$.
However, the cell above $B_{k-1}$ has value greater than $u$ since $(P')^t$ is semistandard and $u\notin R_k$. 
This contradicts the minimality of $u'$, as $u'$ is greater than this value, hence proving the claim.

If Case 3 occurs at $R_{k-1}$, then $u$ is bumped out of $R_{k-1}$ with the property that 
when $u'$ is inserted into $R_{k-1}$, $u\in R_{k-1}$ is the smallest entry with $[u,u']\subseteq R_{k-1}$.
Let $z$ be the entry below cell $B_k$. Then, similar to the argument immediately before, $z\leqslant u'$.
Hence, we have established that the integers weakly increase along the column 
containing $B_k$ after $u$ is appended at the end of $R_k$. 

\item[Case (b)(ii)] Suppose that $1\leqslant j<k$ and Case 2 occurs at $R_j$. 
Then $y$ is the entry bumped out of $R_j$ with the property that when $u$ is inserted into $R_j$, 
$y\in R_j$ is the smallest entry with $y>u$. 
Thus, as $u\notin R_j$, for all entries $z$ and $z'$ respectively to the left and to the right of $B_j$, we have
$z<u<y<z'$.

If Case 2 occurs at $R_{j-1}$, then $u$ is bumped out of $R_{j-1}$ with the property that 
when $u'$ is inserted into $R_{j-1}$, $u\in R_{j-1}$ is the smallest entry with $u>u'$. 
Let $z$ be the entry below cell $B_j$. Then by repeating the same argument as in the first subcase of in Case (b)(i), 
we obtain $z\leqslant u$.

If Case 3 occurs at $R_{j-1}$, then $u$ was bumped out of $R_{j-1}$ with the property that 
when $u'$ is inserted into $R_{j-1}$, $u\in R_{j-1}$ is the smallest entry with $[u,u']\subseteq R_{j-1}$.
Let $z$ be the entry below cell $B_j$. Then by repeating the same argument as in the second subcase of in 
Case (b)(i), we obtain $z\leqslant u'$.

Hence, we have established that integers increase weakly along the column 
containing $B_j$ but increase strictly along $R_j$ after $u$ bumps out $y$. 
\item[Case (b)(iii)] Suppose that $1\leqslant j<k$ and Case 3 occurs at $R_j$. 
In this case, there are no changes to row $R_j$ after inserting $u$ and bumping $y$.
Hence, it is trivial that integers increase weakly along the column 
containing $B_j$ but increase strictly along $R_j$ after $u$ bumps out $y$. 
\end{description}
In all cases, we have shown that if $(P')^t$ is semistandard, then the transpose of the updated tableau remains 
semistandard. Therefore, by induction on the number of added cells, we have proved that the insertion tableau $P$ 
under $\star$-insertion satisfies the property that $P^t$ is semistandard.

Finally, note that the shape of the recording tableau is modified only when Case 1 of the $\star$-insertion has occurred. 
In this case, a cell is added to form $Q$ at the same position as the cell added to form $P$. 
Since we always begin with a pair of empty tableaux, by inducting on the number of added cells, 
the shapes of $P$ and $Q$ are the same. 

\smallskip
\noindent \textbf{(2)}
Suppose that the insertion terminates at $R_k$, where $k\geqslant 1$. 
We shall prove that $B_j$ is weakly to the left of $B_{j-1}$ for all $1<j\leqslant k$ by revisiting the cases explored in 
the proof of part \eqref{lem: P^t.semistd.P.Q.same.shape} (note that $P$ should replace the role of $P'$).

If Case 2 occurs at $R_{j-1}$, then $u$ is the entry bumped out of $R_{j-1}$ with the property that when $u'$ is inserted 
into $R_{j-1}$, $u\in R_{j-1}$ is the smallest entry with $u>u'$.
As in the proof of the first subcase of Case (b)(i) in part \eqref{lem: P^t.semistd.P.Q.same.shape}, we conclude that the 
entry $z$ of the cell below $B_k$ satisfies $z\leqslant u$, showing that $B_j$ is weakly to the left of $B_{j-1}$.

If Case 3 occurs at $R_{j-1}$, then $u$ was bumped out of $R_{j-1}$ with the property that 
when $u'$ is inserted into $R_{j-1}$, $u\in R_{j-1}$ is the smallest entry with $[u,u']\subseteq R_{j-1}$.
As in the proof of the second subcase of Case (b)(i) in part \eqref{lem: P^t.semistd.P.Q.same.shape}, we conclude 
that the entry $z$ of the cell below $B_j$ satisfies $z\leqslant u'$, $B_j$ is weakly to the left of $B_{j-1}$.

This completes the proof.
\end{proof}

For the following results, given a tableau $P$ with positive integer entries, $\row(P)$ denotes its row reading word, 
obtained by reading these entries row-by-row starting from the top row (in French notation), reading from left to right. 
We will consider $\row(P)$ as an element in a fixed $0$-Hecke monoid.

\begin{lemma}\label{lem: row.reading.eq}
Let $P$ be a tableau such that $P^t$ is semistandard and $\row(P)$ is fully-commutative.
Let $x$ be an integer such that $\row(P)\cdot x$ is fully-commutative. Then,
\begin{equation}\label{eq: row.reading.eq}
\row(P\leftarrow x)\equiv_{\mathcal{H}_0} \row(P)\cdot x.
\end{equation}
\end{lemma}

\begin{proof}
To prove \eqref{eq: row.reading.eq}, let us first prove the following statements for all row tableaux $P$:
\begin{itemize}
\item With the assumptions in lemma, if insertion terminates at row $P$ while computing $P\leftarrow x$, then 
\[
	\row(P\leftarrow x)\equiv_{\mathcal{H}_0}\row(P)\cdot x.
\]
\item With the assumptions in lemma, if $y$ is bumped from row $P$ and $P$ changes to $P'$ while computing 
$P\leftarrow x$, then 
\[
	\row(P\leftarrow x)\equiv_{\mathcal{H}_0}y\cdot\row(P').
\]
\end{itemize}
Assume that insertion terminates at row $P$ while computing $P\leftarrow x$. Then, Case 1 must have occurred and 
$P$ changes to $P'$, where $P'$ is $P$ appended by a cell containing $x$. 
Hence, we have
\[
	\row(P\leftarrow x)\equiv_{\mathcal{H}_0}\row(P')\equiv_{\mathcal{H}_0}\row(P)\cdot x.
\]

Assume that $y$ is bumped from row $P$ and $P$ changes to $P'$ while computing $P\leftarrow x$. Then, either 
Case 2 or Case 3 must have occurred.

If Case 2 occurs at $P$, then $x\not\in P$ and there is a $y\in P$ with $y>x$; furthermore, $y$ is the smallest value 
with such property.  Write $P$ as $AyB$, where $A$ and $B$ are the row subtableaux of $P$ formed by entries to the 
left and to the right of $y$, respectively. 
Then, $P\leftarrow x$ is the tableau with row $Axb$ followed by row $y$.
As $x\notin P$, we have $\max(A)<x<y<\min(B)$. Hence by commutativity relations, for all $z\in B$, we have 
$z\cdot x\equiv_{\mathcal{H}_0}x\cdot z$ and for all $z\in A$, we have $z\cdot y\equiv_{\mathcal{H}_0}y\cdot z$, so 
that regarding $A$ and $B$ as words in $\mathcal{H}_0(n)$, we obtain
\[
	A\cdot y\equiv_{\mathcal{H}_0}y\cdot A,\quad B\cdot x\equiv_{\mathcal{H}_0}x\cdot B.
\]

It follows that
\begin{multline*}
\row(P)\cdot x
\equiv_{\mathcal{H}_0} \row(AyB)\cdot x
\equiv_{\mathcal{H}_0} A\cdot y\cdot B\cdot x
\equiv_{\mathcal{H}_0} y\cdot A\cdot x\cdot B
\equiv_{\mathcal{H}_0} y\cdot \row(AxB)
\equiv_{\mathcal{H}_0} \row(P\leftarrow x).
\end{multline*}

If Case 3 occurs at $P$, then $x,y\in P$ with $y$ being the smallest value such that $[y,x]\subseteq P$. Write $P$ 
as $ABC$, where $B=[y,x]$, $A$ and $C$ are respectively the row subtableaux of $P$ formed by entries to the left and 
to the right of $B$. Then, $P\leftarrow x$ is the tableau with row $ABC$ followed by row $y$. As $\row(P)\cdot x$ was assumed 
to be fully-commutative, $x+1\notin P$. Furthermore, by minimality of $y$, $y>\max(A)+1$. Hence, by commutativity relations, 
for all $z\in A$, we have $z\cdot y\equiv_{\mathcal{H}_0} y\cdot z$ and for all $z\in C$, we have 
$x\cdot z \equiv_{\mathcal{H}_0} z\cdot x$, so that
\[
	A\cdot y\equiv_{\mathcal{H}_0}y\cdot A,\quad C\cdot x\equiv_{\mathcal{H}_0}x\cdot C.
\]

Moreover, by using the relations $p-1\; p\; p = p-1\; p-1\; p$, we have $y\cdot B\equiv_{\mathcal{H}_0} B\cdot x$. 
It follows that
\begin{multline*}
\row(P)\cdot x
\equiv_{\mathcal{H}_0} \row(ABC)\cdot x
\equiv_{\mathcal{H}_0} A\cdot B\cdot C\cdot x
\equiv_{\mathcal{H}_0} A\cdot y\cdot B\cdot C
\equiv_{\mathcal{H}_0} y\cdot \row(ABC)
\equiv_{\mathcal{H}_0} \row(P\leftarrow x).
\end{multline*}
Hence, the two statements above hold for all row tableaux $P$.

We are now ready to prove \eqref{eq: row.reading.eq} in full generality. The result follows once we prove by induction 
on the number of rows of $P$, with the given setup above, that the following statements hold:
\begin{itemize}
\item If the insertion terminates within tableau $P$ while computing $P\leftarrow x$, then
\[\row(P\leftarrow x)\equiv_{\mathcal{H}_0}\row(P)\cdot x.\]
\item If $y$ is bumped from tableau $P$ and $P$ changes to $P'$ while computing $P\leftarrow x$, then 
\[\row(P\leftarrow x)\equiv_{\mathcal{H}_0}y\cdot\row(P').\]
\end{itemize}

Indeed, if $P$ is a (possibly empty) row tableau, then we are done by the two previous statements that have been proved. 
Let $k\geqslant 1$ be an arbitrary integer. Assume that both statements mentioned above hold for all such tableaux 
$P$ with $k$ rows.

Let $P$ be a tableau with $k+1$ rows with the setup as above. Then, we may consider the subtableau $P^*$ 
formed from its first $k$ rows and denote the final row as $R$. Note that $\row(P)=\row(R)\cdot\row(P^*)$ and $\row(R)$ 
is fully-commutative.

Assume that the changes from $P$ to $P\leftarrow x$ involve at most the first $k$ rows of $P$. Then $P\leftarrow x$ 
is the same tableau as $P^*\leftarrow x$ with an extra row $R$, so that by the inductive hypothesis,
\[
\row(P\leftarrow x)
\equiv_{\mathcal{H}_0} \row(R)\cdot\row(P^*\leftarrow x)
\equiv_{\mathcal{H}_0} \row(R)\cdot\row(P^*)\cdot x\\
\equiv_{\mathcal{H}_0} \row(P)\cdot x.
\]

Now assume that the changes from $P$ to $P\leftarrow x$ involves all $k+1$ rows of $P$. Let $P'$ be the resulting 
tableau after performing these changes on $P^*$ and let $y$ be the entry bumped from the final row of $P^*$.
Then, $P\leftarrow x$ is the tableau obtained by concatenating tableau $R\leftarrow y$ after $P'$.

If the insertion terminates at row $R$, then by the previous statements for all row tableaux and the inductive hypothesis, we obtain
\begin{multline*}
\row(P\leftarrow x)
\equiv_{\mathcal{H}_0} \row(R\leftarrow y)\cdot\row(P')
\equiv_{\mathcal{H}_0} \row(R)\cdot y\cdot\row(P')\\
\equiv_{\mathcal{H}_0} \row(R)\cdot\row(P^*\leftarrow x)
\equiv_{\mathcal{H}_0} \row(R)\cdot\row(P^*)\cdot x
\equiv_{\mathcal{H}_0} \row(P)\cdot x.
\end{multline*}

Otherwise, if the insertion bumps $z$ from $R$ and $R$ changes to $R'$ while computing $R\leftarrow y$, 
then it holds that the insertion bumps $z$ from $P$ while computing $P\leftarrow x$. In this case, if we denote 
$P''$ as the tableau $P'$ concatenated by row $R'$, then
\[
\row(P)\cdot x
\equiv_{\mathcal{H}_0} \row(R\leftarrow y)\cdot\row(P')
\equiv_{\mathcal{H}_0} z\cdot\row(R')\cdot\row(P')
\equiv_{\mathcal{H}_0} z\cdot\row(P'')
\equiv_{\mathcal{H}_0} \row(P\leftarrow x).
\]
This completes the induction.
\end{proof}

\begin{remark}
\label{remark.321 avoiding}
Observe that the assumption that $\row(P)$ is fully-commutative implies that $\row(R)$ is fully-commutative for each row 
$R$ of $P$. Moreover, in the proof of Lemma \ref{lem: row.reading.eq}, if $x$ is to be inserted into row $R$ of $P$ 
when computing $P\leftarrow y$ and $x\in R$, then the extra assumption that $\row(P)\cdot x$ is fully-commutative implies 
that $R$ does not contain $x+1$.
\end{remark}

\begin{lemma} \label{lem: row.bumping}
Let $P$ be a tableau such that $P^t$ is semistandard and $\row(P)$ is fully-commutative. Let $x,x'$ be integers such 
that $\row(P)\cdot x$ and $\row(P)\cdot xx'$ are fully-commutative.

Denote the insertion paths of $P\leftarrow x$ and $(P\leftarrow x)\leftarrow x'$ as $\pi$ and $\pi'$ respectively. 
Also, suppose that $P\leftarrow x$ and $(P\leftarrow x)\leftarrow x'$ introduce boxes $B$ and $B'$ respectively. 
Then the following statements about $\star$-insertion are true:
\begin{enumerate}
	\item If $x<x'$, then $\pi'$ is strictly to the right of $\pi$. 
	Moreover, $B'$ is strictly to the right of and weakly below $B$. \label{lem: horizontal.strip}
	\item If $x\geqslant x'$, then $\pi'$ is weakly to the left of  $\pi$.  
	Moreover, $B'$ is weakly to the left of and strictly above $B$. \label{lem: vertical.strip}
\end{enumerate}
\end{lemma}

\begin{proof}
Similar to Fulton's proof \cite{Fulton.1996} of the Row Bumping Lemma, we will keep track of the entries as they are 
bumped from a row. Consider a row $R$ of tableau $P$ and suppose that $u$ and $u'$ are to be inserted into $R$ 
when computing $P\leftarrow x$ and $(P\leftarrow x)\leftarrow x'$ respectively, where $u<u'$. 
Denote by $C$ (similarly $C'$) the box in $\pi$ (similarly $\pi'$) that is also in $R$. 

\medskip

\noindent \textbf{Case 1:} $x<x'$. We will prove that the following assertions hold for $R$:
\begin{enumerate}
\item[(a)] If the insertion terminates at $R$ while computing $P\leftarrow x$, then the insertion terminates at $R$ 
while computing $(P\leftarrow x)\leftarrow x'$.
\item[(b)] $C'$ is strictly to the right of $C$.
\end{enumerate}
Note that the insertion terminates at $R$ when computing $P\leftarrow x$ precisely when Case 1 of the $\star$-insertion
occurs at $R$. Box $C$ containing $u$ is appended at the end of $R$. 
As $u'>u$, Case 1 occurs again at $R$ with box $C'$ containing $u'$ appended to the right of $C$, so bumping does 
not occur at $R$ when computing $(P\leftarrow x)\leftarrow x'$. This proves (a) and simultaneously, (b) for this case.

Let us assume that bumping occurs at $R$ with $y$ bumped out when computing $P\leftarrow x$. 
\begin{description}
\item [Case A] If $y$ is bumped from $R$ because Case 2 occurs, the insertion at row $R$ introduced to box $C'$ 
occurs strictly to the right of $C$ (containing $u$) because:
\begin{enumerate}
\item [(i)] If $u'>\max(R)$, then box $C'$ containing 
$u'$ is appended to the end of $R$ by Case 1. 
In particular, $C'$ appears strictly to the right of $C$.
\item [(ii)] Otherwise, since $u'>u$, the letter $u'$ is inserted into a 
box $C'$ strictly to the right of $C$ with $y'$ bumped out. 
If $u'\notin R$, Case 2 occurs and $y'>y$ because 
$C'$ and $C$ originally contained $y'$ and $y$ respectively. 
Else, $u'\in R$ and Case 3 occurs. Suppose that $[y',u']$ is the 
longest interval of consecutive integers contained in $R$. 
Since box $C$ that originally contained $y$ is strictly to the left of $C'$, 
we have $u<y<u'$. Therefore, $[u,u']$ cannot be contained in $R$, so $y<y'$. 
\end{enumerate}
\item [Case B] Otherwise, $y$ is bumped from $R$ because Case 3 occurs when computing $P\leftarrow x$ 
and $[y,u]$ is the longest interval of consecutive integers contained in $R$ by Remark~\ref{remark.321 avoiding}.
The insertion at row $R$ introduced to box $C'$ occurs strictly to the right of $C$ (containing $u$) because:
\begin{enumerate}
\item[(i)] If either $u'>\max(R)$ or $u'\notin R$, then by similar 
arguments as in Case A(i) and Case A(ii), $C'$ appears to the right of $C$. 
Furthermore, in the latter situation, by a similar argument in Case A(ii), we have $y<y'$.
\item[(ii)] Otherwise, $u'\in R$ and Case 3 occurs. 
As $u'>u$, $u'$ is inserted into box $C'$ strictly to the right of $C$ with $y'$ bumped out. 
In addition, $[y',u']$ is the longest interval of consecutive integers contained in $R$. 
As $\row(R)$ is fully-commutative before computing $P\leftarrow x$, $u+1\notin R$. 
Hence $[u,u']$ cannot be contained in $R$. It follows that $y\leqslant u<u+1<y'$.
\end{enumerate}
\end{description}
Note that in the arguments above, we have also shown that if $y$ and $y'$ are bumped from $R$ 
when computing $P\leftarrow x$ and $(P\leftarrow x)\leftarrow x'$ respectively, then $y<y'$. 
It follows that we may apply similar arguments in the rows following $R$. 
Since assertion (b) now holds for all rows, we conclude that $\pi'$ is strictly to the right of $\pi$. 
In addition, $\pi'$ cannot continue after $\pi$ ends because of assertion (a). 
Considering that $\pi'$ goes weakly left by Lemma \ref{lem: star.insertion}, 
we conclude that box $B'$ is strictly to the right of and weakly below $B$.

\medskip

\noindent \textbf{Case 2:} $x\geqslant x'$. We will prove that the following assertions hold for $R$:
\begin{enumerate}
\item If the insertion terminates at $R$ while computing $P \leftarrow x$, then bumping occurs at $R$ while 
computing $(P\leftarrow x)\leftarrow x'$.
\item $C'$ is weakly to the left of $C$.
\end{enumerate}

If the insertion terminates at row $R$ when computing $P\leftarrow x$, 
then Case 1 occurs and box $C$ containing $u$ is appended at the end of $R$. 
If $u'\in R$, Case 3 occurs at $R$ with $y'\leqslant u'\leqslant u$ bumped out. 
Furthermore, box $C'$ containing $u'$ is weakly to the left of $C$.
If $u'\notin R$, Case 2 occurs at $R$ with $y'>u'$ bumped out and $u'<u$. 
We have $y'\leqslant u$ by minimality of $y'$, so that box $C'$ is weakly to the left of $C$.
In either of the subcases, bumping occurs at 
$R$ when computing $(P\leftarrow x)\leftarrow x'$. This proves (a) and simultaneously, (b) for this case.

Let us assume that bumping occurs at $R$ with $y$ bumped out when computing $P\leftarrow x$. 
\begin{description}
\item[Case A] If $y$ is bumped from $R$ because Case 2 occurs when computing $P\leftarrow x$, the insertion at row 
$R$ introduced to box $C'$ occurs weakly to the left of $C$ (containing $u$) because:
\begin{enumerate}
\item [(i)] If $u'\notin R$, then $u'$ is inserted into box $C'$ containing $y'$ by Case 2, while bumping out this $y'$. 
As $u'<u$, we have $y'\leqslant u<y$ and that $C'$ appears weakly to the left of $C$.
\item [(ii)]Otherwise, $u'\in R$ and Case 3 occurs. The letter $u'$ is inserted into box $C'$ weakly to the left of $C$ as 
$u'\leqslant u$. In addition, if $[y',u']$ is the longest interval of consecutive integers in $R$, then $y'$ is bumped out. 
Furthermore, we have $y'<y$ as $C$, which originally contained $y$ before computing $P\leftarrow x$, is to the right 
of the box containing $y'$. 
\end{enumerate}
\item[Case B] Otherwise, $y$ is bumped from $R$ because Case 3 occurs when computing $P\leftarrow x$. Let $[y,u]$ 
be the longest interval of consecutive integers that is contained in $R$. The insertion at row $R$ introduced to box $C'$ 
occurs weakly to the left of $C$ (containing $u$) because:
\begin{enumerate}
\item[(i)] If $u'\notin R$, then $u'<u$, $u'$ is inserted into box $C'$ containing $y'$ and $y'$ is bumped out by Case 2. 
As $\row(P)\cdot x$ is fully-commutative, in particular $\row(R)$ is fully-commutative. Hence $u'<y$, so that $C'$ is weakly to
the left of box containing $y$ (hence also weakly to the left of $C$). Furthermore, we have $y'\leqslant y$ by the minimality 
of $y'$. 
\item[(ii)] If $u'\in R$, then either $u'=u$ or $u<u'$. The former case is easy as Case 3 occurs again with $u'$ inserted 
into $C'=C$ and $y'=y$ is bumped out. If $u<u'$, then as $\row(P)\cdot x$ is fully-commutative, $\row(R)$ is fully-commutative, 
so that $u'<y-1$. It follows that $C'$ is strictly to the left of box containing $y$ (hence also strictly to the left of $C$). 
Furthermore, we have $y'\leqslant u'<y-1<y$. 
\end{enumerate}
\end{description}

Note that in the arguments above, we have also shown that if $y$ and $y'$ are bumped from $R$ when computing 
$P\leftarrow x$ and $(P\leftarrow x)\leftarrow x'$ respectively, then $y\geqslant y'$. It follows that we may apply similar 
arguments in the rows following $R$. 
Since assertion (b) now holds for all rows, we conclude that $\pi'$ is weakly to the left of $\pi$. 
In addition, $\pi'$ must continue after $\pi$ ends because of assertion (a). Considering that $\pi'$ goes weakly left by 
Lemma \ref{lem: star.insertion}, we conclude that box $B'$ is weakly to the left of and strictly above $B$.
\end{proof}

Let $U$ be a tableau such that $U^t$ is semistandard and $\row(U)$ is fully-commutative. We describe 
the \defn{reverse row bumping} for $\star$-insertion of $U$ as follows. 
Locate an inner corner of $U$ and remove entry $y$ from that row. 
Perform the following operations until an entry is bumped out of the bottommost row. Suppose that we are reverse 
bumping $y$ into a row $R$. If $y\notin R$, find the largest $x\in R$ with $x<y$; insert $y$ and bump out $x$. 
Otherwise, $y\in R$, so find the largest $x\in R$ such that $[y,x]$ is the longest interval of consecutive integers. 
In this case, row $R$ remains unchanged but $x$ is bumped out. Then reverse bump $x$ into the next row below
unless there is no further row below. In this case, terminate and return the resulting tableau as $T$ along with the 
bumped entry $x$. It is straightforward to see that reverse row bumping specified above reverses the bumping process 
specified by the $\star$-insertion.

\begin{example}\label{eg: reverse.row.bumping}
Let $U$ be the tableau
\[
U \;=\;
\raisebox{2cm}{
\ytableausetup{notabloids}
\begin{ytableau}
5\\
2\\
2 & 5\\
2 & 3 & 5\\
1 & 2 & 4
\end{ytableau}}
\;.
\]
By performing reverse row bumping on the topmost 5 in $U$, we obtain
\[
T \;=\;
\raisebox{1.5cm}{
\ytableausetup{notabloids}
\begin{ytableau}
5\\
2 & 5\\
2 & 3 & 5\\
1 & 3 & 4
\end{ytableau}}
\]
and entry 2.
It is also straightforward to check that $U=T\leftarrow 2$.
\end{example}

\begin{corollary}\label{cor:horizontal.vertical.strips}
Let $T$ be a tableau of shape $\lambda$ such that $T^t$ is semistandard and $row(T)$ is fully-commutative.
Let $k$ be a positive integer.

Let $x_1<x_2<\dots<x_k$ (similarly $x_k\leqslant \dots \leqslant x_2\leqslant x_1$) be integers such that
$\row(T)\cdot x_1 x_2\dots x_i$ is fully-commutative for all $1\leqslant i\leqslant k$. Then, the collection of boxes added
to $T$ to form the tableau 
\[
	U=((T\leftarrow x_1)\leftarrow x_2)\cdots \leftarrow x_k
\] 
has the property that no two boxes are in the same column (similarly row). 

Conversely, if $U$ is a tableau of shape $\mu$ such that $\lambda\subseteq\mu$ and $\mu/\lambda$ consists of 
$k$ boxes with no two boxes in the same column, i.e, a horizontal strip of size $k$ (similarly row, i.e., a vertical strip of 
size $k$), then there is a unique tableau $T$ of shape $\lambda$ and unique integers $x_1<x_2<\cdots<x_k$ 
(similarly $x_k\leqslant \cdots\leqslant  x_2\leqslant x_1$) such that 
\[
	U=((T\leftarrow x_1)\leftarrow x_2)\cdots \leftarrow x_k.
\] 
In particular, if $(P,Q)=\star([\mathbf{k},\mathbf{h}]^t)$, where $[\mathbf{k},\mathbf{h}]^t$ is a fully-commutative 
decreasing Hecke biword, then $Q$ is semistandard.
\end{corollary}

\begin{proof}
Assume that $x_1<x_2<\cdots<x_k$.
By statement~\eqref{lem: horizontal.strip} of Lemma~\ref{lem: row.bumping}, the sequence of added boxes in
$U=((T\leftarrow x_1)\leftarrow x_2)\cdots \leftarrow x_k$ moves weakly below and strictly to the right when computing 
$U$. In particular, no two of the added boxes can be in the same column.

To recover the required tableau $T$ and integers $x_1<x_2<\cdots<x_k$, perform reverse row bumping on the boxes 
specified by the shape $\mu/\lambda$ within $U$ starting from the rightmost box, working from right to left. The 
tableau $T$ and the integers $x_1,x_2,\dots,x_k$ are uniquely determined by the operations. Moreover, by
Lemma~\ref{lem: row.bumping}, the integers $x_k,x_{k-1},\dots, x_1$ obtained in the given order of operations satisfy
$x_1<x_2<\cdots<x_k$.

Now assume $x_k\leqslant \cdots \leqslant x_2\leqslant x_1$. By statement~\eqref{lem: vertical.strip} of 
Lemma~\ref{lem: row.bumping}, the sequence of added boxes moves strictly above and weakly to the right when 
computing $U$. In particular, no two of the added boxes can be in the same row.

Similarly, one may perform reverse row bumping on the boxes specified by the shape $\mu/\lambda$ within $U$ 
starting from the topmost box, working from top to bottom. Again, the operations uniquely determine the tableau $T$ 
and the integers $x_1,x_2,\dots,x_k$. Moreover, by Lemma \ref{lem: row.bumping}, the integers
$x_k,x_{k-1},\dots, x_1$ obtained in the given order of operations satisfy $x_k\leqslant \cdots\leqslant x_2\leqslant x_1$.

Finally, note that in a decreasing Hecke biword $[\mathbf{k},\mathbf{h}]^t$, where $\mathbf{h}=h^m\dots h^2 h^1$,
entries within a fixed $a^i$ are inserted in increasing order. It follows that the collection of all boxes with label $i$ form 
a horizontal strip within the tableau $Q$. Collecting all these horizontal strips with values $i$ from $m$ to $i$ in order 
by using the converse recovers $Q$, implying that $Q$ is semistandard.
\end{proof}

\begin{theorem}
\label{theorem.star insertion bijection}
The $\star$-insertion is a bijection from the set of all fully-commutative decreasing Hecke biwords to the set of all pairs of 
tableaux $(P,Q)$ of the same shape, where both $P^t$ and $Q$ are semistandard and $\row(P)$ is fully-commutative.
\end{theorem}

\begin{proof}
By successive applications of Lemma \ref{lem: row.reading.eq}, if $(P,Q)=\star([\mathbf{k},\mathbf{h}]^t)$, then as 
$\mathbf{h}$ is fully-commutative, $\row(P)$ is also fully-commutative. 
Hence, using Lemma \ref{lem: star.insertion} and Corollary \ref{cor:horizontal.vertical.strips}, 
$\star$-insertion is a well-defined map from the set of all fully-commutative decreasing Hecke biwords to the set of 
all pairs of tableaux $(P,Q)$ of the same shape with both $P^t,Q$ semistandard and $\row(P)$ being fully-commutative.

It remains to show that the $\star$-insertion is an invertible map. Assume that $P$ and $Q$ are tableaux of the same 
shape with both $P^t,Q$ semistandard and $\row(P)$ being fully-commutative. 
Since $Q$ is semistandard, the collection of boxes with the same entry form a horizontal strip. Starting with the 
largest such entry $m$, perform reverse row bumping with the boxes in the strip from right to left. 
By Lemma \ref{lem: row.bumping}, this recovers the entries in $h^m$ in decreasing order. 
Repeating this procedure in decreasing order of entries recovers $\mathbf{h}=h^m\dots h^2 h^1$, which
automatically yields a decreasing Hecke biword $[\mathbf{k},\mathbf{h}]^t$. 
Furthermore, by repeated applications of Lemma \ref{lem: row.reading.eq}, since $\row(P)$ was fully-commutative, 
then the reverse word of $\mathbf{h}$ is fully-commutative, so that $\mathbf{h}$ is fully-commutative too. 
Finally, by repeated applications of the converse stated in Corollary \ref{cor:horizontal.vertical.strips}, the recovered 
decreasing Hecke biword $[\mathbf{k},\mathbf{h}]^t$ is unique.
\end{proof}

\section{Properties of the $\star$-insertion}
\label{section.properties}

In this section, we show that the $\star$-insertion intertwines with the crystal operators. More precisely, 
the insertion tableau remains invariant on connected crystal components under the $\star$-insertion as shown
in Section~\ref{section.micro moves} by introducing certain micro-moves. In Section~\ref{section.star insertion and crystal},
it is shown that the $\star$-crystal on $\mathcal{H}^{m,\star}$ intertwines with the usual crystal operators on semistandard 
tableaux on the recording tableaux under the $\star$-insertion. In Section~\ref{section.uncrowding}, we relate
the $\star$-insertion to the uncrowding operation.

\subsection{Micro-moves and invariance of the insertion tableaux}
\label{section.micro moves}

In this section, we introduce certain equivalence relations of the $\star$-insertion in order to establish its relation with the 
$\star$-crystal. From now on we are focusing on the sequence in the insertion order. Since each decreasing 
factorization $\mathbf{h}$ is inserted from right to left, we look at $\mathbf{h}$ read from right to left.

\begin{definition}\label{def: equiv}
We define an equivalence relation through \defn{micro-moves} on fully-commutative words in $\mathcal{H}_0(n)$.
\begin{enumerate}
	\item \defn{Knuth moves}, for $x<z<y$:
	\begin{enumerate}
		\item[(I1)] $xyz \sim yxz$
		\item[(I2)] $zxy \sim zyx$
	\end{enumerate}
	\item \defn{Weak Knuth moves}, for $y>x+1$:
	\begin{enumerate}
		\item[(II1)] $xyy \sim yxy$ 
		\item[(II2)] $xxy \sim xyx$
	\end{enumerate}
	\item \defn{Hecke move}, for $y=x+1$:\\
	(III) $xxy \sim xyy$
\end{enumerate}
Note that the micro-moves preserve the relation $\equiv_{\mathcal{H}_0}$.
\end{definition}

Similar relations have appeared in~\cite[Eq.~(1.2)]{FominGreene.1998}.

\begin{example}
\label{example.Hecke words}
The $13242\in\mathcal{H}_0(5)$ is equivalent to $31242$, $13422$, $13224$, $31224$, and itself.
\end{example}

Next, we use the following notation on $\star$-insertion tableaux. For a single-row increasing tableau $R$,
let $R^x$ denote the first row of the tableau $R\leftarrow x$ and let $R(x)$ denote the output of the $\star$-insertion
from the first row. If the $\star$-insertion outputs a letter, then denote it by $R(x)$; if $x$ is appended to the end of the 
row $R$, then the output $R(x)$ is $0$, which can be ignored. We always have $x\cdot 0 \sim x \sim 0 \cdot x$.

\begin{example}
Let $R = 
\raisebox{0cm}{
	\ytableausetup{notabloids} 
	\begin{ytableau}
	1 & 3 & 4 & 6 & 7 & 8
	\end{ytableau}
}$, then the first row of $R\leftarrow 7$ is
\[
	R^7 = \raisebox{0cm}{\ytableausetup{notabloids} \begin{ytableau} 1 & 3 & 4 & 6 & 7 & 8
	\end{ytableau}}
\]
and $R(7) = 6.\,$ Furthermore, the first row of $R^7 \leftarrow 9$ is $R^{7,9} = 
\raisebox{0cm}{
\ytableausetup{notabloids} 
\begin{ytableau}
1 & 3 & 4 & 6 & 7 & 8 & 9
\end{ytableau}
}$ and $R^8(9)=0$.
\end{example}

\begin{lemma} \label{lema: single-row}
Let $R$ be a single-row increasing tableau, and $x,y,z$ be letters such that $\row(R)\cdot x\cdot y \cdot z$ is 
fully-commutative. Let $x',y',z'$ be letters such that $xyz \sim x'y'z'$. Following the above notation, we have 
\[
	R^{xyz} = R^{x'y'z'} \qquad \text{and} \qquad R(x)R^x(y)R^{xy}(z) \sim R(x')R^{x'}(y')R^{x'y'}(z').
\]
\end{lemma}

\begin{proof}
Let $R$ be a single-row increasing tableau and $M$ be the largest letter in $R$. First note that if $a\in R$ and 
$\row(R)\cdot a$ is fully-commutative, then $a+1\notin R$, see also Remark~\ref{remark.321 avoiding}.

There are five types of equivalence triples, so we discuss them in 3 groups. 

\smallskip \noindent 
\textbf{1. Cases (I1) and (II1):} We have $x<z<y$, or $x<z=y$ and $y>x+1$. In both cases $x'=y, y'=x, z'=z$. 

\smallskip \noindent
\textbf{Case (1A):}  $M<x<z\leqslant y$. In this case, the first resulting tableau is $R^{xyz}=\ytableausetup{notabloids} 
\begin{ytableau}
R & x & z
\end{ytableau}$ and the outputs are $R(x)=R^x(y)=0$ and $R^{xy}(z) = y$. The second resulting tableau is 
$R^{yxz}=\ytableausetup{notabloids} 
\begin{ytableau}
R & x & z
\end{ytableau}$ and the outputs are $R(y)=0=R^{yx}(z)$ and $R^{y}(x) = y$. So we have $R^{xyz}=R^{yxz}$ and 
also $0\cdot 0\cdot y \sim 0 \cdot y \cdot 0$.

\smallskip \noindent
\textbf{Case (1B):}  $x\leqslant M <z\leqslant y$. In this case, we have $R^{xy}=R^{yx}$ and $R(x)=R^y(x)$ since $y$ 
is just appended to the end of $R$ and does not influence how $x$ is inserted. This gives $R^{xyz}=R^{yxz}$. The 
related outputs are $R^x(y)=R(y)=0$, $R^{xy}(z)=R^{yx}(z)=y$. Thus, $R(x)\cdot 0 \cdot y \sim 0\cdot R(x) \cdot y$.

\smallskip \noindent
\textbf{Case (1C):}  $x <z\leqslant M<y$. In this case, we also have that $R^{xy}=R^{yx}$ and $R(x)=R^y(x)$, for the 
same reason as case (1B). Thus, we have $R^{xyz}=R^{yxz}$ and $R^{xy}(z)=R^{yx}(z)$. Since we have
$R^x(y)=R(y)=0$, $R(x)\cdot 0 \cdot R^{xy}(z) \sim 0\cdot R^y(x) \cdot R^{yx}(z)$.

\smallskip \noindent
\textbf{Case (1D):}  $x <z\leqslant y\leqslant M$. If $x$ is the maximal letter in $R^x$, then it follows as case (1B). Otherwise, this case needs further separation into subcases.

\smallskip \noindent
\textbf{Case 1D-(i):} $x,y\notin R$. Then $x<R(x), y<R(y)$ and $R(x)\neq y$. 

\noindent
\textbf{(1)} If $R(x)<y$, then $R^x(y)=R(y)$ and $R^y(x)=R(x)$, which implies $R^{xy}=R^{yx}$, thus $R^{xyz}=R^{yxz}$ 
and $R^{xy}(z)=R^{yx}(z)$. Hence $R(x)R^x(y)R^{xy}(z) = R^y(x)R(y)R^{yx}(z)$. Since $R(x)<R^{xy}(z)\leqslant y<R(y)$,
we have $R(y)>R^y(x)+1$ and for the outputs $R(x)R^x(y)R^{xy}(z)=R^y(x)R(y)R^{yx}(z)  \sim R(y)R^y(x)R(y)R^{yx}(z)$ 
by move type (I1) or (II1).

\noindent
\textbf{(2)} If $R(x)>y$, let the letter to the right of $R(x)$ in $R$ be $R(x)^{\rightarrow}$. Then both $R^{xyz}$ 
and $R^{yxz}$ are obtained by replacing $R(x)$ with $x$ and $R(x)^{\rightarrow}$ with $z$. For the output, we have 
$R(x)=R(y)$, $R^x(y)>R(y)$, $R^y(x)=y$, $R^{yx}(z)=R^x(y)$ and $R^{xy}(z)=y$. Since $y< R(x)<R^x(y)$, we have that 
$R^x(y)=R(x)^{\rightarrow}>y+1$. Hence the outputs $R(x)R^x(y)R^{xy}(z) = R(x)R(x)^{\rightarrow}y 
\sim R(x) y R(x)^{\rightarrow}= R(y)R^y(x)R^{yx}(z)$ by move of type (I2).

\smallskip \noindent
\textbf{Case 1D-(ii):} $x\in R,y\notin R$. Then $R(x)\leqslant x, R(y)>y$ and $x+1\notin R$. In this case, we have 
$R^x(y)=R(y)$ and $R^y(x)=R(x)$, thus $R^{xy}=R^{yx}$, $R^{xy}(z)=R^{yx}(z)$ and $R^{xyz}=R^{yxz}$. Since 
$x+1\notin R$, we have $R^{xy}(z)>x+1$. This implies $R(x)\leqslant x < R^{xy}(z)\leqslant y<R(y)$, thus 
$R(x)R^x(y)R^{xy}(z) \sim R(y)R^y(x)R^{yx}(z)$ as it is a type (I1) move.

\smallskip \noindent
\textbf{Case 1D-(iii):} $x\notin R, y\in R$. Then $x<R(x)$, $y\geqslant R(y)$, $y+1\notin R$, $R(x)-1\notin R$, 
$R(x)\leqslant y$, $R(y)\leqslant R^x(y)$ and $R^y=R$. 

\noindent
\textbf{(1)} If $R(x)=y$, denote the box to the right of $y$ in $y$ as $y^{\rightarrow}$. Note that $y^{\rightarrow}>y+1$. 
Then $R^x(y)=y^{\rightarrow}$, $R^{xy}(z)=y$, $R^y(x)=y$ and $R^{yx}(z)=y^{\rightarrow}$. Note $y-1\notin R$, otherwise 
$R(x)\leqslant y-1$. Thus, $R(y)=y$. Both $R^{xyz}$ and $R^{yxz}$ are obtained by replacing $y\in R$ with $x$ and
$y^{\rightarrow}$ with $z$, so $R^{xyz}=R^{yxz}$. The outputs $R(x)R^x(y)R^{xy}(z) = y y^{\rightarrow} y\sim 
y y y^{\rightarrow} = R(y)R^y(x)R^{yx}(z)$ as it is a type (II2) move.

\noindent
\textbf{(2)} Suppose $R(x)<y$ and $R(x)=R(y)$. Then $[R(x),y]\subset R$ and $R^x(y)=R(x)+1$. Since $R^y=R$ and 
$R^{xy}=R^x$, we have that both $R^{xy}$ and $R^{yx}$ equal $R^x$ and furthermore $R^y(x)=R(x)$. Note that $z$ can 
either be equal to $y$ or $z<R^x(y)$, otherwise $z\in R^{xy}$ and $z+1\in R^{xy}$, which will give us a braid from 
$\row(R^{xy})\cdot z$. Thus, we have $R^{xy}(z)=R^{yx}(z) = R(x)+1$. In either case, the outputs are
$R(x)R^x(y)R^{xy}(z) = R(x)(R(x)+1)(R(x)+1) \sim R(x) R(x) (R(x)+1) = R(y)R^y(x)R^{yx}(z)$ as they are 
type (III) moves.

\noindent
\textbf{(3)} Suppose $R(x)<y$ and $R(x)<R(y)$. Then $R(y)>R(x)+1$ and $R^x(y)=R(y)$. Similar to the previous case, both 
$R^{xy}$ and $R^{yx}$ are equal to $R^x$, and $z$ is either $y$ or $z<R(y)$. In either case, $R^{xy}(z)\leqslant R(y)$. 

Then the outputs are $R(x)R^x(y)R^{xy}(z)= R(x)R(y)R^x(z) \sim R(y)R(x)R^x(z) = R(y)R^y(x)R^{yx}(z)$ as they are 
type (I1) or (II1) moves.

\smallskip \noindent
\textbf{Case 1D-(iv):} $x, y\in R$. In this case $x\geqslant R(x)$, $y\geqslant R(y)$, $x+1\notin R$ and $y+1\notin R$.  
Since $x+1\notin R$, $[x,y]$ is not contained in $R$ and hence $R(y)>x+1>x\geqslant R(x)$.

Then $R^x(y)=R(y)$, $R^y(x)=R(x)$ and $R^{xy}=R^{yx}=R$. Since $z>x$ and $x+1\notin R$, 
we have $R(z)>x+1\geqslant R(x)+1$. By similar reasons to the previous two subcases of Case 1D-(iii), $z$ can either 
be $y$ or $z<R(y)$ in 
order to avoid a braid in $\row(R^{xy}) z$. So, we have $R^{xy}(z)\leqslant R(y)$. Then the outputs are 
$R(x)R^x(y)R^{xy}(z)=R(x)R(y)R(z) \sim R(y)R(x)R(z)=R(y)R^y(x)R^{yx}(z)$ as they are type (I1) moves.

\smallskip \noindent 
\textbf{2. Cases (I2) and (II2):} We have $z < x<y$, or $z=x<y$ and $y>z+1$. In both cases $x'=x, y'=z, z'=y$. 
By definition, $x\in R^x$.

\smallskip \noindent
\textbf{Case (2A):}  $M< x < y$, then $R(x)=R^x(y)=0$. $R^{xy}=\ytableausetup{notabloids} 
\begin{ytableau}
R & x & y
\end{ytableau}$ is obtained by appending $x$ and $y$ to the end of $R$. Since $x\in R^x$ and $z\leqslant x<y$, we 
have $R^{xy}(z)=R^x(z)$. Moreover, $R^{xzy}$ is obtained by appending $y$ to the end of $R^{xz}$ and hence 
$R^{xyz}=R^{xzy}$. The outputs are $R(x)R^x(y)R^{xy}(z)=0 0 R^x(z) \sim 0 R^x(z) 0= R(x)R^x(z)R^{xz}(y)$.

\smallskip \noindent
\textbf{Case (2B):}  $z\leqslant x \leqslant M < y$, then $R^x(y)=R^{xz}(y)=0$. Since $R^{xy}=\ytableausetup{notabloids} 
\begin{ytableau}
R^x & y
\end{ytableau}$, $x\in R^x$ and $z\leqslant x$, we have $R^{xy}(z)=R^x(z)$, thus $R^{xyz} = \ytableausetup{notabloids} 
\begin{ytableau}
R^{xz}  & y
\end{ytableau} = R^{xzy}$. The output $R(x)R^x(y)R^{xy}(z) = R(x) 0 R^x(z) \sim R(x) R^{x}(z)0=R(x)R^x(z)R^{xz}(y)$.

\smallskip \noindent
\textbf{Case (2C):}  $z\leqslant x < y\leqslant M$, then we have $R^x(z)\leqslant x$. We discuss the following subcases.

\smallskip \noindent
\textbf{Case 2C-(i):} $x,y\notin R$, then we have $R(x)>x$ and $R^x(y)>y$. Since $y>x$ and $x$ replaces $R(x)$ in 
$R$, we have $R^x(y)>R(x)$ from row strictness. Since $R^x(y)>R(x)$ and $R^x(z)\leqslant x$,  
we have $R^{xy}(z)=R^x(z)$ and $R^{xz}(y)=R^x(y)$. Furthermore, $R^{xyz}=R^{xzy}$. Moreover, we have 
$R^x(z)\leqslant x<R(x)<R^x(y)$, which implies $R^x(y)>R^x(z)+1$. Hence 
$R(x)R^x(y)R^{xy}(z) = R(x)R^x(y)R^x(z) \sim R(x)R^x(z)R^{x}(y) = R(x)R^x(z)R^{xz}(y)$ by type (I2) moves.

\smallskip \noindent
\textbf{Case 2C-(ii):} $x\in R, y\notin R$. Then $R^x=R$, $R(x)\leqslant x, R^x(y)>y$. Since $z\leqslant x$ and 
$[R(x),x]\subset R^x$, we have that $R^x(z)\leqslant R(x)$. Since $R^x(y)>y>x$ and $R^x(z)\leqslant R(x)$, we have 
that $R^{xy}(z)=R^{x}(z)$ and $R^{xz}(y)=R^x(y),$ thus $R^{xyz}=R^{xzy}$. Since $R^x(z)\leqslant R(x)\leqslant 
x<y<R^x(y)$, we have $R^x(y)>R^x(z)+1$. The outputs are $R(x)R^x(y)R^{xy}(z) = R(x)R^x(y)R^x(z) \sim 
R(x)R^x(z)R^x(y)=R(x)R^x(z)R^{xz}(y)$ by type (I2) or (II2) moves.

\smallskip \noindent
\textbf{Case 2C-(iii):} $x\notin R, y\in R$. Then $R^{xy}=R^x$, $R(x) > x$ and $R^x(y)\leqslant y$. Let the letter to the 
right of $R(x)$ in $R$ be $R(x)^{\rightarrow}$. Then $R(x)^{\rightarrow}>R(x)>x$ implies $R(x)^{\rightarrow}>x+1$. 
This also shows that $x+1\notin R^x$ and thus $R^x(y)>R(x)$. Since $R^{xy}=R^x$ and $R^{xzy}=R^{xz}$, we have 
$R^{xyz}=R^{xz}=R^{xzy}$. Since $R^x(z)\leqslant x<R(x)<R^x(y)$, we have $R^x(y)>R^x(z)+1$. Since $z\leqslant x$, 
we also have that $R^{x}(y)=R^{xz}(y)$.  Thus, the outputs are $R(x)R^x(y)R^{xy}(z) = R(x) R^x(y)R^x(z) \sim 
R(x)R^x(z)R^{x}(y)=R(x)R^x(z)R^{xz}(y)$ by a type (I2) move.

\smallskip \noindent
\textbf{Case 2C-(iv):} $x\in R, y\in R$. Then $R^x=R$, $R^{xy}=R$, $R(x)\leqslant x$, $R^x(y)\leqslant y$, $x+1\notin R$ 
and $y+1\notin R$. Thus, $R^x(y)>x+1$.  Since $z\leqslant x$ and $[R(x),x]\subset R^x$, we have that 
$R^x(z)\leqslant R(x)$. Since $R^{xy}=R$, $R^{xyz}=R^z$. Since $R^x(z)\leqslant x$, $R^{xz}(y)=R^z(y)=R(y)$ 
and thus $R^{xzy}=R^{z}$. This implies $R^{xyz}=R^{xzy}$. Now we have $R(z)\leqslant R(x)\leqslant x<x+1<R(y)$.
Therefore, the outputs are $R(x)R^x(y)R^{xy}(z) = R(x)R(y)R(z) \sim R(x)R(z)R(y)=R(x)R^x(z)R^{xz}(y)$ by type (I2) 
or (II2) moves.

\smallskip \noindent 
\textbf{3. Case (III):} We have $y=x$, $z=x+1$ and hence $x'=x$, $y'=x+1$ and $z'=x+1$.

\smallskip \noindent
\textbf{Case (3A):} $x>M$. Then $R^x$ is obtained by appending $x$ to the end of $R$ and $R(x)=0$. Also $R^{xx}=R^x$ 
with output $R^x(x)$. Note $R^{x,x+1}(x+1)=R^x(x)$. 
Both $R^{xx,x+1}$ and $R^{x,x+1,x+1}$ are obtained by appending 
$x+1$ to the end of $R^x$, thus they are the same. The outputs are $R(x)R^x(x)R^{xx}(x+1)=0 R^x(x) 0
\sim 0 0 R^x(x)= R(x)R^x(x+1)R^{x,x+1}(x+1)$.

\smallskip \noindent
\textbf{Case (3B):} $x\leqslant M, x+1>M$. Both $R^{xx,x+1}$ and $R^{x,x+1,x+1}$ are obtained by appending $x+1$
to the end of $R^x$, so they are equal. Since $x\in R^x$, we have $R^{x,x+1}(x+1)=R^{x}(x)$. Thus, the outputs are 
$R(x)R^x(x)R^{xx}(x+1)=R(x)R^x(x) 0 \sim R(x) 0 R^{x,x+1}(x+1)$.

\smallskip \noindent
\textbf{Case (3C):} $x+1\leqslant M$. It is clear that $x\in R^x$. If $x$ is the maximal letter in $R^x$, then the rest follows 
as case (3B).

Otherwise, let $x^{\rightarrow}$ be the letter to the right of $x$ in 
$R^x$. Since $x\in R^x$, we must have $x+1\notin R^x$, thus $x^{\rightarrow}>x+1$. Moreover, we have $R^{xx}=R^x$, 
$R^x(x+1)=R^{xx}(x+1)=x^{\rightarrow}$.
Since $R^{x,x+1}$ is obtained from $R^x$ by replacing $x^{\rightarrow}$ with $x+1$ and $x,x+1\in R^{x,x+1}$, we have 
$R^{x,x+1}(x+1)=R^x(x)$. Both $R^{xx,x+1}$ and $R^{x,x+1,x+1}$ are obtained from $R^x$ by replacing $x^{\rightarrow}$
with $x+1$, thus they are the same. Furthermore, since $R^x(x)\leqslant x$ and 
$x^{\rightarrow}>x+1$, we have that $R(x)R^x(x)R^{xx}(x+1)=R(x)R^x(x) x^{\rightarrow} \sim 
R(x) x^{\rightarrow}R^{x}(x) = R(x)R^x(x+1)R^{x,x+1}(x+1)$ by a type (I2) or (II2) move.
\end{proof}

\begin{proposition} \label{prop: insertion}
If two words in $\mathcal{H}_0(n)$ have the property that their reverse words are equivalent according to 
Definition~\ref{def: equiv}, then they have the same insertion tableau under $\star$-insertion (inserted from right to left).
\end{proposition}

\begin{proof}
Let $P$ be a $\star$-insertion tableau. By Lemma~\ref{lem: star.insertion}, $P^{t}$ is a semistandard tableau. Let the 
rows of $P$ be $R_1,\dots, R_\ell$. Then each row is strictly increasing. The row $R_j$ is considered to be empty
for $j>\ell$.
 
Let $x_1,y_1,z_1$ and $x_1',y'_1,z'_1$ be letters such that $x_1y_1z_1 \sim x'_1y'_1z'_1$ and
$\row(P)\cdot x_1 \cdot y_1 \cdot z_1$ is fully-commutative. Let the output of the $\star$-insertion algorithm of 
$P\leftarrow x_1 \leftarrow y_1 \rightarrow z_1$ (resp. $P\leftarrow x'_1 \leftarrow y'_1 \leftarrow z'_1$) from the row 
$i$ be $x_{i+1},y_{i+1},z_{i+1}$ (resp. $x'_{i+1},y'_{i+1},z'_{i+1}$). That is:
\begin{itemize}
\item
$R^{x_iy_iz_i}_i$ is the first row of $[(R_i\leftarrow x_i) \leftarrow y_i] \leftarrow z_i$ and the outputs in order are 
$x_{i+1},y_{i+1},z_{i+1}$. 
\item
$R_i^{x'_iy'_iz'_i}$ is the first row of $[(R_i\leftarrow x'_i) \leftarrow y'_i] \leftarrow z'_i$ and outputs in order are 
$x'_{i+1},y'_{i+1},z'_{i+1}$. 
\end{itemize}
By Lemma \ref{lema: single-row}, we have that $R^{x_iy_iz_i}_i = R_i^{x'_iy'_iz'_i}$ and 
$x_{i+1}y_{i+1}z_{i+1} \sim x'_{i+1}y'_{i+1}z'_{i+1}$ for all $i$ (possibly some extra rows exceeding $\ell$). 
Thus, we have the desired result.
\end{proof}

\begin{example}
The four words in $\mathcal{H}_0(5)$ of Example~\ref{example.Hecke words} all have the same $\star$-insertion tableau:
\[
\raisebox{1cm}{
\ytableausetup{notabloids} 
\begin{ytableau}
3 & \none & \none \\
1 & \none & \none \\
1 & 2 & 4
\end{ytableau}
}\,.
\]
\end{example}

In the next couple of lemmas, we prove that the crystal operators $f^{\star}_k$ act by a composition of micro-moves
as given in Definition \ref{def: equiv}. More precisely, for a fully-commutative decreasing factorization $\mathbf{h}$, we have 
$\mathbf{h}^{\mathsf{rev}} \sim f_k^\star(\mathbf{h})^{\mathsf{rev}}$ as long as $f_k^\star(\mathbf{h})\neq 0$, where 
$\mathbf{h}^{\mathsf{rev}}$ is the reverse of $\mathbf{h}$. 

\begin{remark}
\label{remark.fstar action}
By Definition~\ref{def: star} and Remark~\ref{rm: star}, there are two cases for the $k$-th and $(k+1)$-st factors
under the crystal operator $f_k^\star$, where $x$ is the largest unpaired letter in the $k$-th factor, $w_i,v_i>x$ and 
$u_i,b_i<x$:
\begin{enumerate}
	\item $(w_1\dots w_p u_1 \dots u_q)(v_1\dots v_s x b_1 \dots b_t)\xrightarrow{f^{\star}_k} 
	(w_1\dots w_p x u_1 \dots u_q)(v_1\dots v_s b_1 \dots b_t)$,\\
	 where  $v_s \neq x+1$.
	\item $(w_1\dots w_p u_1 \dots u_q)(v_1\dots v_s x b_1 \dots b_t)\xrightarrow{f^{\star}_k} 
	(w_1\dots w_{p} x u_1 \dots u_q)(v_1\dots v_{s-1} x b_1 \dots b_t)$, \\
	where $v_s=w_p=x+1$.
\end{enumerate}
In both cases, $u_i<x-1$ since if $u_1=x-1$ then $b_1=x-1$ due to the fact that $x$ is unbracketed; but this would 
mean that the word is not fully-commutative. We also notice that since all $u_i$ are paired with some $b_j$, we 
have that $t\geqslant q$ and $b_i\geqslant u_i$. Similarly, all $v_i$ are paired with some $w_j$, so we have that 
$p\geqslant s$ and $v_i\geqslant w_{p-s+i}$.
Let $u$ denote the sequence $u_1\dots u_q$ and let $b$ denote the 
sequence $b_1\dots b_t$.
\end{remark}

\begin{lemma} \label{lem: far}
\mbox{}
\begin{enumerate}
\item
For $2\leqslant i\leqslant q$, $b_{i-1}>u_i+1$. 
\item
For $1\leqslant i <s$, $v_i>w_{p-s+i+1}+1$.
\end{enumerate}
\end{lemma}

\begin{proof}
\noindent
\textbf{(1):} When $b_{i-1}>b_i+1$ or $u_i<b_i$, the result follows directly. 
	
Consider the case that  $u_i=b_i=a$ and $b_{i-1}=b_i+1=a+1$ for some letter $a$. Since 
$a=u_i<u_{i-1}\leqslant b_{i-1}=a+1$, we must have $u_{i-1}=a+1$. Let $c$ be the largest letter such that 
$[a,c]\subseteq b$. Then $c\geqslant a+1$ and $c+1\notin b$. Moreover, since all $u_i$ are paired, $u_i\leqslant b_i$ 
and  $u_{j-1}>u_j$, it is not hard to see that $[a,c]\subseteq u$ and $c,c-1\in u$. Since $c+1\notin b$, we can use 
commutativity to move $c\in b$ to the left and obtain a subword $c(c-1)c$, which contradicts that the original word 
is fully-commutative.

\smallskip \noindent
\textbf{(2):} The proof is almost identical to the first part. When $w_{p-s+i}>w_{p-s+i+1}+1$ or $v_i>w_{p-s+i}$, the 
result follows.

Consider the case $w_{p-s+i}=w_{p-s+i+1}+1=a+1$ and $v_i=w_{p-s+i}=a+1$ for some letter $a$. Since 
$a=w_{p-s+i+1}\leqslant v_{i+1}<v_i=a+1$, we must that $v_{i+1}=a$. Let $c$ be the smallest letter such that 
$[c,a+1]\subseteq w$. Then $c\leqslant a$ and $c-1\notin w$. Moreover, since all $v_i$ are paired, 
$v_j\geqslant w_{p-s+j}$ and $v_{j+1}<v_j$, we can see that $[c,a+1]\subseteq v$ and $c,c+1\in v$. 
Since $c-1\notin w$, 
we can use commutativity to move $c\in w$ to the right and form a subword $c(c+1)c$, which contradicts that the original 
word is fully-commutative.
\end{proof}

We now summarize several observations that will be used later. 
\begin{remark} \label{remark.both}
For both types of actions of $f_k^\star$ as in Remark~\ref{remark.fstar action}, we have the following equivalence relations:
	\begin{enumerate}
		\item For $1\leqslant i \leqslant q, 1\leqslant j \leqslant s-1, \, v_{j+1}v_ju_i \sim v_{j+1}u_iv_j$, \, 
		since $u_i<v_{j+1}<v_j$.
		\item For $1\leqslant i \leqslant q, \, xv_su_i \sim xu_iv_s$, \, since $u_i<x<v_s$.
		\item For $1\leqslant i \leqslant q, \,b_1xu_i \sim b_1u_ix$, \, since $u_i\leqslant u_1\leqslant b_1<x$, 
		and $u_i<x-1$.
		\item For $1\leqslant j< i-1, \,1\leqslant i \leqslant q, \, b_{j+1}b_ju_i \sim b_{j+1}u_ib_j$, \, 
		since $u_i \leqslant b_i< b_{j+1}<b_j$.
		\item For $2\leqslant i \leqslant q, \, b_ib_{i-1}u_i \sim b_iu_ib_{i-1}$, \, since $u_i\leqslant b_i<b_{i-1}$  
		and $b_{i-1}>u_i+1$ by Lemma \ref{lem: far}.
		\item For $1\leqslant i\leqslant s, \, p-s+i-1\leqslant j\leqslant p-1, \, w_{j+1}v_iw_j \sim v_iw_{j+1}w_j$, \, 
		since $w_{j+1}<w_j<w_{p-s+i}\leqslant v_i$.
		\item For $1\leqslant i\leqslant s-1, w_{p-s+i+1}v_iw_{p-s+i} \sim v_iw_{p-s+i+1}w_{p-s+i} $, \, since 
		$w_{p-s+i+1}<w_{p-s+i}\leqslant v_i$ and $v_i>w_{p-s+i+1}+1$ by Lemma \ref{lem: far}.
		\item For all $1\leqslant j \leqslant s-1, \, 1\leqslant i \leqslant q, v_{j+1}u_iv_j \sim v_{j+1}v_ju_i$, \, 
		since $u_i<v_{j+1}<v_j$.
		\item For $1< i\leqslant q, \, b_1u_iv_s \sim b_1v_su_i$, \, since $u_i<u_1\leqslant b_1<v_s$.
		\item For $1\leqslant i\leqslant q,\, 1\leqslant j \leqslant s$, $xu_iv_j\sim xv_ju_i$,\, since $u_i<x<v_j$.
		\item For $1\leqslant j\leqslant s-1, \, xv_jw_p \sim v_jxw_p$, \, since $x<w_p\leqslant v_s <v_j$.
	\end{enumerate}
\end{remark}

\begin{remark} \label{remark.simple}
	When $v_s\neq x+1$, we have the following equivalence relations:
	\begin{enumerate}
		\item $1\leqslant i\leqslant s, \, xv_iw_p \sim v_ixw_p$, \, since $x<w_p\leqslant v_s$ and $v_s>x+1$. 
		\item $b_1u_1v_s \sim b_1v_su_1$, \, since $u_1\leqslant b_1 <v_s$ and $v_s>x+1>u_1+1$.
	\end{enumerate}
\end{remark}

\begin{lemma} \label{lem: mv1}
	We have that $b_q\dots b_1 x v_s\dots v_1 u_q \dots u_1$ is equivalent to 
	$b_qu_q\dots b_2u_2b_1u_1 xv_s\dots v_1$.
\end{lemma}

\begin{proof}
With the equivalence relations from Remark~\ref{remark.both} (1)-(5), we can make the sequences of equivalence moves 
as follows:
\begin{align*}
b_q\dots b_1 x v_s\dots v_2v_1 u_qu_{q-1} \dots u_1 &\sim b_q\dots b_1 x v_s\dots  v_2u_qv_1u_{q-1} \dots u_1 \sim \\
b_q\dots b_1 x v_s u_q\dots  v_2v_1u_{q-1} \dots u_1 &\sim b_q\dots b_1 x u_q v_s\dots  v_2v_1u_{q-1} \dots u_1 \sim\\
b_q\dots b_1 u_q x v_s\dots  v_2v_1u_{q-1} \dots u_1 &\sim b_qu_q\dots b_1 x v_s\dots  v_2v_1u_{q-1} \dots u_1 \sim \\
b_qu_qb_{q-1}u_{q-1}\dots b_1u_1 x v_s\dots  v_2v_1. &
\end{align*}
\end{proof}

\begin{lemma} \label{lem: mv2}
We have that $v_s\dots v_1w_p\dots w_{p-s+1}$ is equivalent to $v_sw_pv_{s-1}w_{p-1}\dots v_1w_{p-s+1}$.
\end{lemma}

\begin{proof}
With the equivalence relations from Remark~\ref{remark.both} (6)-(7), we can make the following equivalence moves:
\begin{align*}
v_s\dots v_{2}v_1w_pw_{p-1}\dots w_{p-s+1} &\sim v_s\dots v_{2} w_pw_{p-1}\dots w_{p-s+2}v_1w_{p-s+1} \sim \\
v_s\dots v_3w_pw_{p-1}\dots v_2w_{p-s+2}v_1w_{p-s+1} &\sim v_sw_p\dots v_1w_{p-s+1}.
\end{align*}
\end{proof}

\begin{lemma} \label{lem: mv5}
	 We have 
	 \[
	 	xw_pv_{s-1}w_{p-1}\dots v_2w_{p-s+2}v_1w_{p-s+1} \sim 
	 	v_{s-1}\dots v_1 x w_p \dots w_{p-s+1}.
	\]
\end{lemma}

\begin{proof}
With the equivalence relations from Remark~\ref{remark.both} (6),(7) and (11),
we can make the following equivalent moves:
\begin{align*}
xw_pv_{s-1}w_{p-1}\dots v_2w_{p-s+2}v_1w_{p-s+1} &\sim xv_{s-1}w_pw_{p-1}\dots v_2w_{p-s+2}v_1w_{p-s+1} \sim \\
v_{s-1}xw_pw_{p-1}\dots v_2w_{p-s+2}v_1w_{p-s+1} &\sim v_{s-1}\dots v_1xw_pw_{p-1}\dots w_{p-s+2}w_{p-s+1}.
\end{align*}
\end{proof}

\begin{lemma} \label{lem: mv3}
When $v_s\neq x+1$, we have 
\[
	xv_sw_pv_{s-1}w_{p-1}\dots v_1w_{p-s+1} \sim v_s\dots v_1 x w_p \dots w_{p-s+1}.
\]
\end{lemma}

\begin{proof}
With the equivalence relations from Remark~\ref{remark.both} (6)-(7) and Remark~\ref{remark.simple} (1), we can 
make the following equivalence moves:
\begin{align*}
xv_sw_pv_{s-1}w_{p-1}v_{s-2}\dots v_1w_{p-s+1} &\sim v_sxw_pv_{s-1}w_{p-1}v_{s-2}\dots v_1w_{p-s+1} \sim \\
v_sxv_{s-1}w_pw_{p-1}v_{s-2}\dots v_1w_{p-s+1} & \sim v_sv_{s-1}xw_pw_{p-1}v_{s-2}\dots v_1w_{p-s+1} \sim \\
v_sv_{s-1}v_{s-2}\dots v_1 x w_pw_{p-1} \dots w_{p-s+1}. &
\end{align*}
\end{proof}

\begin{lemma} \label{lem: mv4}
When $v_s\neq x+1$, we have $b_qu_q \dots b_1u_1 v_s\dots v_1$ is equivalent to 
$b_q\dots b_1 v_s \dots v_1 u_q \dots u_1$.
\end{lemma}

\begin{proof}
With the equivalence relations from Remark~\ref{remark.both} (4), (5), (8)-(9) and Remark~\ref{remark.simple} (2), we can make the following equivalence moves:
\begin{align*}
b_qu_q \dots b_1u_1 v_s\dots v_1 &\sim b_qu_q \dots b_1 v_s u_1 \dots v_1 \sim\\
 b_qu_q \dots b_1 v_s  \dots v_1 u_1 &\sim b_q\dots b_1 v_s \dots v_1 u_q \dots u_1.
\end{align*}
\end{proof}

\begin{lemma} \label{lem: mv6}
We have $b_qu_q \dots b_1u_1xv_{s-1}\dots v_1$ is equivalent to $b_q\dots b_1 x v_{s-1}\dots v_1u_q\dots u_1$.
\end{lemma}

\begin{proof}
With the equivalence relations from Remark~\ref{remark.both} (1), (3), (5) and (10)
we have the following equivalence moves:
\begin{align*}
b_qu_q \dots b_1u_1xv_{s-1}\dots v_1 &\sim b_qu_q \dots b_1xu_1v_{s-1}\dots v_1 \sim \\
b_qu_q \dots b_1xv_{s-1}\dots v_1u_1 &\sim b_q \dots b_1xv_{s-1}\dots v_1u_q\dots u_1.
\end{align*}
\end{proof}

\begin{proposition} \label{prop: equiv}
Suppose $\mathbf{h}$ is a fully-commutative decreasing factorization such that $f_k^\star(\mathbf{h}) \neq 0$
(resp. $e_k^\star(\mathbf{h})\neq 0$). Then $f_k^\star(\mathbf{h})^{\mathsf{rev}} \sim \mathbf{h}^{\mathsf{rev}}$ 
(resp. $e_k^\star(\mathbf{h})^{\mathsf{rev}} \sim \mathbf{h}^{\mathsf{rev}}$) for the equivalence relation 
$\sim$ of Definition~\ref{def: equiv}.
\end{proposition}

\begin{proof}
We prove the statement for $f^{\star}_k$. Since $e^{\star}_k$ is a partial inverse of $f_k^{\star}$, the result follows.

Let $\mathbf{h}=h^m\dots h^1\in \mathcal{H}^{m,\star}$ and define $\widetilde{\mathbf{h}} = f^{\star}_k(\mathbf{h})= 
h^m\dots \tilde{h}^{k+1}\tilde{h}^k h^{k-1}\dots h^1$. Specifically, $h^{k+1}=(w_1\dots w_p u_1\dots u_q)$ and 
$h^k = (v_1\dots v_sxb_1\dots b_t)$, where $x$ is the largest unpaired letter in $h^k$. Then by Lemmas~\ref{lem: mv1} 
and~\ref{lem: mv2}, we have the following sequence of equivalence moves:
\begin{align*}
(b_q\dots b_1 xv_s\dots v_1 u_q \dots u_1) w_p\dots w_{p-s+1} & \sim 
(b_qu_q\dots b_1u_1 xv_s\dots v_1)w_p\dots w_{p-s+1} \\
b_qu_q\dots b_1u_1 x (v_s\dots v_1w_p\dots w_{p-s+1}) &\sim b_qu_q\dots b_1u_1 x (v_sw_p\dots v_1w_{p-s+1}).
\end{align*}

\smallskip \noindent
\textbf{Case (1):}
When $v_s\neq x+1$, $\tilde{h}^{k+1}=(w_1\dots w_p x u_1 \dots u_q),\, \tilde{h}^k=(v_1\dots v_s b_1 \dots b_t)$. 
By Lemmas~\ref{lem: mv3} and~\ref{lem: mv4}, we have
\begin{align*}
	b_qu_q\dots b_1u_1 (x v_sw_p\dots v_1w_{p-s+1}) 
	& \sim b_qu_q\dots b_1u_1 (v_s\dots v_1 x w_p \dots w_{p-s+1}) \\
	(b_qu_q\dots b_1u_1 v_s\dots v_1) x w_p \dots w_{p-s+1} 
	&\sim (b_q\dots b_1 v_s\dots v_1 u_q \dots u_1) x w_p \dots w_{p-s+1}.
\end{align*}
Thus, we have that 
\[
	b_t\dots b_1xv_s\dots v_1 u_q\dots u_1 w_p\dots w_1 
	\sim b_t\dots b_1xv_s\dots v_1 u_q\dots u_1 xw_p\dots w_1.
\]

\smallskip \noindent
\textbf{Case (2):}
When $v_s=w_p=x+1$, $\tilde{h}^{k+1}=(w_1\dots w_{p} x u_1 \dots u_q),\, \tilde{h}^k=(v_1\dots v_{s-1} x b_1 \dots b_t)$. 
Then by Lemmas~\ref{lem: mv5} and~\ref{lem: mv6}, we have
\begin{align*}
	b_qu_q\dots b_1u_1 (x v_sw_p) v_{s-1}w_{p-1} \dots v_1w_{p-s+1} 
	&\sim b_qu_q\dots b_1u_1 (x x w_p) v_{s-1}w_{p-1} \dots v_1w_{p-s+1} \\
	b_qu_q\dots b_1u_1 x (x w_p v_{s-1}w_{p-1} \dots v_1w_{p-s+1}) 
	&\sim b_qu_q\dots b_1u_1 x (v_{s-1}\dots v_1 x w_p\dots w_{p-s+1}) \\
	(b_qu_q\dots b_1u_1 x v_{s-1}\dots v_1) x w_p\dots w_{p-s+1} 
	&\sim  (b_q\dots b_1 x v_{s-1}\dots v_1 u_q\dots u_1) x w_p\dots w_{p-s+1}. 
\end{align*}
Thus, we have that 
\[
	b_t\dots b_1xv_s\dots v_1 u_q\dots u_1 w_p\dots w_1 
	\sim b_t\dots b_1 x v_{s-1}\dots v_1 u_q \dots u_1 x w_p\dots w_1.
\]
Therefore, we have shown that in both cases, $f^{\star}_k(\mathbf{h})^{\mathsf{rev}}\sim \mathbf{h}^{\mathsf{rev}}$.
\end{proof}

\begin{proposition}
\label{proposition.P tableaux}
For $\mathbf{h}\in \mathcal{H}^{m,\star}$ such that $f^{\star}_k(\mathbf{h})\neq 0$ for some $1\leqslant k<m$, the 
$\star$-insertion tableau for $\mathbf{h}$ equals the $\star$-insertion tableau for $f^{\star}_k(\mathbf{h})$.
\end{proposition}

\begin{proof}
By Proposition~\ref{prop: equiv}, the reverse words for $\mathbf{h}$ and $f^{\star}_k(\mathbf{h})$ are 
$\sim$-equivalent. By Proposition~\ref{prop: insertion}, the corresponding insertion tableaux are equal.
\end{proof}

\begin{proposition} \label{prop: lowest-weight.star-insertion}
Let $\mathbf{h}\in \mathcal{H}^{m,\star}$ be a lowest weight element under Definition~\ref{def: star} of weight $\lambda$. 
Then there exists $r\geqslant 1$ where $\lambda_i=0$ for $i<r$ and $\lambda_{i+1}\geqslant \lambda_i$ for 
$1\leqslant i\leqslant m$. Suppose $\mathbf{h} = h^{m}\cdots h^r= (h^m_{\lambda_m}\dots h^m_1)
(h^{m-1}_{\lambda_{m-1}}\cdots h^{m-1}_1)\dots (h^r_{\lambda_r}\dots h^r_1)$, then the $i$-th row of the $\star$-insertion 
tableau equals $h^{m+1-i}_1, h^{m+1-i}_2, \dots, h^{m+1-i}_{\lambda_{m+1-i}}$, that is,
\begin{equation}
\label{equation.Pstar}
\ytableausetup{mathmode,boxsize=2.5em}
P^\star(\mathbf{h}) = 
\raisebox{2cm}{
	\begin{ytableau}
	h^r_1 & \dots & h^r_{\lambda_r}\\
	\dots & \dots & \dots & \dots \\
	h^{m-1}_1 & h^{m-1}_2 & \dots & \dots & h^{m-1}_{\lambda_{m-1}}\\
	h^m_1 & h^m_2 & \dots & \dots & \dots  & h^m_{\lambda_m}  
	\end{ytableau}}\;.
\end{equation}
\end{proposition}

\begin{proof}
Without loss of generality, we may assume that $r=1$. We prove the statement by induction on $m$.
The case $m=1$ is trivial.

Let $m\geqslant 1$ be arbitrary and suppose that the statement holds for this $m$. We prove the statement for $m+1$. 
We need to insert $P^\star(\mathbf{h}) \leftarrow h^{m+1}_1\leftarrow h^{m+1}_2\leftarrow \cdots \leftarrow 
h^{m+1}_{\lambda_{m+1}}$, where $P^\star(\mathbf{h})$ is as in~\eqref{equation.Pstar} with $r=1$. Note that 
$h^{m+1}_i\leqslant h^m_i$ for $1\leqslant i \leqslant \lambda_m$. Specifically, $h^{m+1}_1\leqslant h^{m}_1$, so its 
insertion path is vertical along the first column and we obtain
\[
\ytableausetup{mathmode,boxsize=2.5em}
P^\star(\mathbf{h}) \leftarrow h^{m+1}_1= 
\raisebox{2.7cm}{
	\begin{ytableau}
	h^1_1\\
	h^2_1 & \dots & h^r_{\lambda_r}\\
	\dots & \dots & \dots & \dots \\
	h^{m}_1 & h^{m-1}_2 & \dots & \dots & h^{m-1}_{\lambda_{m-1}}\\
	h^{m+1}_1 & h^m_2 & \dots & \dots & \dots  & h^m_{\lambda_m}  
	\end{ytableau}}\;.
\]
Since $h^{m+1}_1<h^{m+1}_2\leqslant h^m_2$, the insertion path of $h^{m+1}_2$ is strictly to the right of the insertion
path of $h^{m+1}_1$ and weakly left of the second column by Lemma \ref{lem: row.bumping}, so it is vertical along the 
second column. Similar arguments show that the insertion path for $h^{m+1}_i$ is just vertical along the $i$-th column. 
Thus, the result holds for $m+1$.
\end{proof}

\begin{remark}\label{remark.shape}
For a lowest weight element $\mathbf{h}\in \mathcal{H}^{m,\star}$ of weight $\mathbf{a}$, the corresponding insertion 
tableau must have shape $\mu=\sort(\mathbf{a})$, which is the partition obtained by reordering $\mathbf{a}$.
\end{remark}

\begin{proposition}\label{prop: res.star-insertion}
Let $T\in \SSYT(\lambda)$ and $(P,Q) = \star \circ \res(T)$. Then $Q = T$.
\end{proposition}

\begin{proof}
The proof is done by induction on subtableaux of $T$ similarly to the proof of Theorem~\ref{thm: res}.

For a given step in the insertion process, suppose that the entries of $T$ that are involved so far form a nonempty 
subtableau $T'$ of $T$ with shape $\mu$ containing cell $(1,1)$. Furthermore, assume that the insertion and 
recording tableau at the corresponding step are $P(T')$ and $Q(T')$. Then they both have shape $\mu$, and the entry 
of cell $(i,j)$ of $P(T')$ is $\ell+j-\mu'_j+i-1$. In addition, $Q(T') = T'$, where $\mu'$ is the conjugate of the partition 
$\mu$ and $\ell:= \lambda'_1=\ell(\lambda)$.

Note that we do not encounter Case (1) in the proof of Theorem~\ref{thm: res}. All other arguments still hold since for 
every insertion the letter is not contained in the row it is inserted into, that is, the insertion always bumps the smallest letter 
that is greater than itself. Thus, we omit the detail of the proof.
\end{proof}

\subsection{The $\star$-insertion and crystal operators}
\label{section.star insertion and crystal}

In this section, we prove that the $\star$-insertion and the crystal operators
on fully-commutative decreasing factorizations and semistandard Young tableaux intertwine.

\begin{theorem}\label{thm: star-crystal.star-insertion}
Let $\mathbf{h}\in \mathcal{H}^{m,\star}$.
Let $(P^\star(\mathbf{h}),Q^\star(\mathbf{h})) = \star(\mathbf{h})$ be the insertion and recording tableaux under the
$\star$-insertion of Definition~\ref{def: new_insertion}. Then 
\begin{enumerate}
\item 
$f_i^\star(\mathbf{h})$ is defined if and only if $f_i(Q^\star(\mathbf{h}))$ is defined. 
\item
If $f_i^\star(\mathbf{h})$ is defined, then $Q^\star(f_i^\star(\mathbf{h})) = f_i Q^\star(\mathbf{h})$.
\end{enumerate}
In other words, the following diagram commutes:
\[
	\begin{tikzcd}
	\mathcal{H}^{m,\star} \arrow[r,"Q^\star"] \arrow[d,"f^{\star}_i"] & \SSYT^m \arrow[d,"f_i"] \\
	\mathcal{H}^{m,\star}\arrow[r,"Q^\star"] & \SSYT^m.
	\end{tikzcd}
\]
\end{theorem}

\begin{proof}
The crystal operator $f_i^\star$ acts only on factors $h^{i+1}$ and $h^i$. Hence it suffices to prove the statement
for $\mathbf{h}=h^{i+1} h^i \dots h^1$ with $i+1$ factors.

Suppose $f_i^\star(\mathbf{h})\neq 0$. By Proposition~\ref{proposition.P tableaux}, $P^\star(\mathbf{h})=
P^\star(f^\star_i(\mathbf{h}))$. Furthermore, by Lemma~\ref{lem: star.insertion} $P^\star(\mathbf{h})$ and 
$Q^\star(\mathbf{h})$ have the same shape. Hence in particular, $Q^\star(\mathbf{h})$ and 
$Q^\star(f_i^\star(\mathbf{h}))$ have the same shape and therefore the letters $i$ and $i+1$ in $Q^\star(\mathbf{h})$ 
and $Q^\star(f_i^\star(\mathbf{h}))$ occupy the same skew shape.

Recall from Definition~\ref{def: star} that $f_i^\star$ removes precisely one letter from factor 
$h^i = (h^i_\ell h^i_{\ell-1} \dots h^i_1)$, say $h^i_k$. By Lemma~\ref{lem: row.bumping},
the insertion paths of $h_1^i,\dots,h_\ell^i$ into $P^\star(h^{i-1}\cdots h^1)$ move strictly to the right and the newly
added cells form a horizontal strip. In addition, the letters $h_1^i,\dots,h_\ell^i$ appear in the first row of
$P^\star(h^i \cdots h^1)$. Now compare this to the insertion paths for $h^i_1,\dots,\widehat{h^i_k},\dots,h^i_\ell$
into $P^\star(h^{i-1}\dots h^1)$, where $h^i_k$ is missing. Up to the insertion of $h_{k-1}^i$, everything agrees.
Suppose that $h^i_k$ bumps the letter $x$ in the first row and $h^i_{k+1}$ bumps the letter $y>x$ in the first row
by Lemma~\ref{lem: row.bumping}.
Then when $h^i_{k+1}$ gets inserted without prior insertion of $h^i_k$, the letter $h_{k+1}^i$ either still bumps $y$ 
or $h_{k+1}^i$ bumps $x$ (in which case $x$ and $y$ are adjacent in the first row in $P^\star(h^{i-1} \cdots h^1)$).
There are no other choices, since if there are letters between $x$ and $y$ in the first row and $h^i_{k+1}$ bumps
one of these, it would have already bumped a letter to the left of $y$ in $P^\star(h^i\cdots h^1)$. If $h^i_{k+1}$ bumps
$x$ without prior insertion of $h^i_k$, then its insertion path is the same as the insertion path of $h^i_k$ previously. If 
$h^i_{k+1}$ bumps $y$, then the letter inserted into the second row by similar arguments either bumps the same letter 
as in the previous insertion path of $h^i_{k+1}$ or $h^i_k$ and so on. The last cell added is hence the same cell
added in the previous insertion path of either $h^i_k$ or $h^i_{k+1}$. Repeating these arguments, exactly one
cell containing $i$ in $Q^\star(h^i\cdots h^1)$ is missing in $Q^\star((h^i_\ell \dots\widehat{h^i_k}\dots h^i_1)
h^{i-1}\cdots h^1)$ and all other cells containing $i$ are the same. Hence, $Q^\star(f_i^\star(\mathbf{h}))$ is obtained
from $Q^\star(\mathbf{h})$ by changing exactly one letter $i$ to $i+1$.

It remains to prove that $f_i^\star(\mathbf{h})\neq 0$ if and only if $f_i(Q^\star(\mathbf{h}))\neq 0$ and, if 
$f_i^\star(\mathbf{h})\neq 0$, then the letter $i$ that is changed to $i+1$ from $Q^\star(\mathbf{h})$ to 
$Q^\star(f_i^\star(\mathbf{h}))$ is the rightmost unbracketed $i$ in $Q^\star(\mathbf{h})$. First assume that under the 
bracketing rule for $f_i^\star$, all letters in the factor $h^i$ are bracketed, so that $f_i^\star(\mathbf{h})=0$. This means 
that each letter in $h^i$ is paired with a weakly smaller letter in $h^{i+1}$.
Then by similar arguments as in Lemma~\ref{lem: row.bumping} (2), for each insertion path for the letters in $h^i$,
there is an insertion path for the letters in $h^{i+1}$ that is weakly to the left and the resulting new cell is weakly to
the left and strictly above of the corresponding new cell for the letter in $h^i$. This means that each $i$
in $Q^\star(\mathbf{h})$ is paired with an $i+1$ and hence $f_i(Q^\star(\mathbf{h}))=0$.

Now assume that $f_i^\star(\mathbf{h})\neq 0$. Let us use the same notation as in Remark~\ref{remark.fstar action}
(with $k$ replaced by $i$). Since all letters $u_q,\dots,u_1<x$ are paired with some letters $b_j<x$, their
insertion paths (again by similar arguments as in Lemma~\ref{lem: row.bumping}) lie strictly to the left of the 
insertion path for $x$. First assume that $v_s \neq x+1$. Recall that by Proposition~\ref{proposition.P tableaux}, 
$P^\star(\mathbf{h})= P^\star(f^\star_i(\mathbf{h}))$. Also, by the above arguments, moving letter $x$ to factor 
$h^{i+1}$ under $f_i^\star$, changes one $i$ to $i+1$ (precisely the $i$ that is missing when removing $x$ from $h^i$).
Now the letters $w_p,\dots,w_1>x$ are inserted after the letter $x$ in the $(i+1)$-th factor in $f_i^\star(\mathbf{h})$ and by
Lemma~\ref{lem: row.bumping} their insertion paths are strictly to the right of the insertion path of $x$ in
$f_i^\star(\mathbf{h})$. But this means that the corresponding $i+1$ in $Q^\star(\mathbf{h})$ cannot bracket
with the $i$ that changes to $i+1$ under $f_i^\star$. This proves that $f_i(Q^\star(\mathbf{h}))\neq 0$.
Furthermore, each $v_s,\dots,v_1$ is paired with some $w_j$ and hence the insertion path of this $w_j$ is weakly to the
left of the insertion path of the corresponding $v_h$. Hence all $i$ to the right of the $i$ that changes to an $i+1$
under $f_i^\star$ are bracketed. This proves that this $i$ is the rightmost unbracketed $i$, proving the claim. 
The case $v_s=x+1$ is similar.
\end{proof}

\begin{remark}
\label{remark.schur expansion}
Proposition~\ref{prop: lowest-weight.star-insertion} and Theorem \ref{thm: star-crystal.star-insertion} provide another 
proof via $\star$-insertion, in the case where $w$ is fully-commutative, of the Schur positivity of $\mathfrak{G}_w$ 
of Fomin and Greene~\cite{FominGreene.1998}
\[
 	\mathfrak{G}_w = \sum_\mu \beta^{|\mu|-\ell(w)} g_w^{\mu} s_\mu,
\]
where $g_w^{\mu} = |\{ T \in \mathsf{SSYT}^n(\mu') \mid w_C(T) \equiv w\} |$.
\end{remark}

\subsection{Uncrowding set-valued skew tableaux}
\label{section.uncrowding}

Buch~\cite{Buch.2002} introduced a bijection from a set-valued tableau of straight shape to a pair 
$(P,Q)$, where $P$ is a semistandard tableau and $Q$ is a flagged increasing tableau.
The map involves the use of a dilation operation~\cite{BandlowMorse.2012,RTY.2018} which can be defined equally to 
act on set-valued skew tableaux. Chan and Pflueger~\cite{ChanPflueger.2019} recently studied the operation in this 
more general context. We review here the results needed for our purposes.

Let $\lambda$, $\mu$ be partitions such that $\lambda\subseteq\mu$ and $\lambda_1=\mu_1$. 
A \defn{flagged increasing tableau} (introduced in~\cite{Lenart.2000} and called \defn{elegant fillings} by various
authors~\cite{Lenart.2000,LamPylyavskyy.2007,BandlowMorse.2012,Patrias.2016})
is a row and column strict filling of the skew shape $\mu/\lambda$ such that 
the positive integers entries in the $i$-th row of the tableau are at most $i-1$ for all $1\leqslant i\leqslant \ell(\mu)$.
In particular, the bottom row is empty.
Denote the set of all flagged increasing tableaux of shape $\mu/\lambda$ by $\mathcal{F}_{\mu/\lambda}$.

We use \defn{multicell} to refer to a cell in a set-valued tableau with more than one letter.

\begin{definition} \label{def: uncrowd}
For a skew shape $\lambda/\mu$, the \defn{uncrowding operation} is defined on  $T\in\SVT(\lambda/\mu)$
as follows: identify the topmost row $r$ in $T$ containing a multicell. Let $x$ be the largest letter
in row $r$ which lies in a multicell; delete this $x$ and perform RSK row bumping with $x$ 
into the rows above. The resulting tableau is the output of this operation. Note that its
shape differs from $\lambda/\mu$ by the addition of one cell.

The \defn{uncrowding map}, denoted $\uncrowd$, is defined as follows. Let $T\in\SVT(\lambda/\mu)$ with 
$\ext(T)=\ell$. 
\begin{itemize}
\item Start with $\widetilde{P}_0=T$ and $\widetilde{Q}_0=F$, where $F$ is the unique flagged increasing tableau of 
shape $\lambda/\lambda$. 
\item For each $1\leqslant i\leqslant \ell$, $\widetilde{P}_i$ is obtained 
from $\widetilde{P}_{i-1}$ by successively applying the uncrowding operation until no multicells remain.
Each operation involves the addition of cell $C$ to form $\widetilde{P}_i$ by first deleting an entry in cell 
$B$ of $\widetilde{P}_{i-1}$; this is recorded by adding a cell with entry $k$ to $\widetilde{Q}_{i-1}$ 
at the same position as $C$, where $k$ is the difference in the row indices of cells $B$ and $C$.
\item Terminate and return $(\widetilde{P},\widetilde{Q})=(\widetilde{P}_\ell,\widetilde{Q}_\ell)$.
\end{itemize}
\end{definition}

\begin{example}
Let $T$ be the semistandard set-valued tableau
\[
\ytableausetup{mathmode,boxsize=1.8em}
T = 
\raisebox{1.6cm}{
\begin{ytableau}
5\\
4 & 4 & 5\\
2 & 23 & 3\\
1 & 1 & 1 & 12 & 234 & 5 
\end{ytableau}}\;.
\]	
Perform an uncrowding operation to obtain
\[
T' = 
\raisebox{2.2cm}{
\begin{ytableau}
5\\
4\\
3 & 4 & 5\\
2 & 2 & 3\\
1 & 1 & 1 & 12 & 234 & 5 
\end{ytableau}}\;.
\]
Proceeding with uncrowding the remaining multicells and recording the changes, we have
$\uncrowd(T)=(\widetilde{P},\widetilde{Q})$, where
\[
\widetilde{P} = 
\raisebox{2.2cm}{
\begin{ytableau}
5 & 5\\
4 & 4\\
3 & 3 & 4\\
2 & 2 & 2 & 3\\
1 & 1 & 1 & 1 & 2 & 5 
\end{ytableau}}
\qquad \text{and} \qquad
\widetilde{Q} =
\raisebox{2.2cm}{
\begin{ytableau}
3 & 4\\
*(gray) & 3\\
*(gray) & *(gray) & *(gray)\\
*(gray) & *(gray) & *(gray) & 1\\
*(gray) & *(gray) & *(gray) & *(gray) & *(gray) & *(gray) 
\end{ytableau}}\;.
\]
\end{example}

\begin{lemma}
\label{lem:uncrowdintertwine}
For skew shape $\lambda/\mu$, the crystal operators on $\SVT^m(\lambda/\mu)$ intertwine with those on 
$\SSYT^m(\nu/\mu)$, for $\lambda\subseteq \nu$,  under $\uncrowd$.  
\end{lemma}

\begin{proof}
Chan and Pflueger~\cite{ChanPflueger.2019} proved that the image of $T\in\SVT(\lambda/\mu)$ under the uncrowding map is 
a pair $(P,Q)$, where $P$ is a semistandard tableau of shape $\nu/\mu$ and $Q$ is a flagged increasing tableau of shape 
$\nu/\lambda$. Monical, Pechenik and Scrimshaw in~\cite[Theorem 3.12]{MPS.2018} proved that the crystal operators on 
$\SVT^m(\lambda)$ intertwine with those on $\SSYT^m(\nu)$ under $\uncrowd$.  
Since $\uncrowd$ is defined equally on skew shapes, the result follows.
\end{proof}

\subsection{Compatibility of $\star$-insertion with uncrowding}

For a partition $\mu$, let $T_\mu$ be the unique tableau of shape $\mu$ with $\mu_i$ letters $i$ in each row $i$.
Note that $\uncrowd(T_\mu)=(T_\mu,\emptyset)$ since $\ext(T_\mu)=0$.

\begin{lemma}
\label{lem:sss}
For $T\in\SVT^m(\lambda/\mu)$, if $(P,Q)=\star(\mathbf{h}\mathbf{h}')$ where $\mathbf{h}=\res(T)$ and 
$\mathbf{h}'=\res(T_\mu)$, then $T_\mu$ is contained in $Q$.
\end{lemma}
\begin{proof}
For $T\in\SVT^m(\lambda/\mu)$, let $T*$ be the set-valued tableau of shape $\lambda$ obtained from $T$ 
by adding $\ell(\mu)$ to each entry and filling in the cells of $\mu$ with $T_\mu$.
By Proposition~\ref{prop: res.star-insertion}, we have
\begin{equation}
	\star \circ\res (T_\mu) = (P_{\mu},T_\mu)\,,
\end{equation}
where $P_\mu$ is the semistandard tableau specified in the proof of Proposition~\ref{prop: res.star-insertion}.
The claim follows by noting that $\res(T*)=\res(T)\res(T_\mu)$. 
\end{proof}

\begin{definition}
A modification of $\star$-insertion is defined on $\HH^{*,m}$ as follows:
for $\mathbf{h}\in\HH^{*,m}$, let $\lambda/\mu$ be the shape of $\res^{-1}(\mathbf{h})$ (which is well-defined up to
a shift by Proposition~\ref{prop.res.image}).
For $\mathbf{h}'=\res(T_\mu)$, let $(P*,Q*)=\star(\mathbf{hh}')$.  Define $\tilde\star(\mathbf{h})=(P,Q)$ where
$P$ is obtained from $P*$ by deleting all entries in cells of $\mu$ and $Q$ is defined
from $Q*$ by deleting $T_\mu$ from it and decreasing all other letters by $\ell(\mu)$.
\end{definition}

Note that this is well-defined by Lemma~\ref{lem:sss} and the fact that
each $\mathbf{h}\in\HH^{*,m}$ can be associated to a skew shape $\lambda/\mu$ 
which is the shape of $\res^{-1}(\mathbf{h})$ by Proposition~\ref{prop.res.image}.
Also note that $\tilde\star(\mathbf{h})=\star(\mathbf{h})$ if $\mu=\emptyset$.

\begin{theorem}
\label{theorem.uncrowding}
Let $T\in\SVT^m(\lambda/\mu)$, $(\tilde{P},\tilde{Q})=\uncrowd(T)$, and $(P,Q)= \tilde \star \circ \res(T)$. Then $Q = \tilde{P}$.
\end{theorem}

\begin{proof}
We start by addressing the straight-shape case; for $T*\in\SVT^m(\lambda)$,
consider the following compositions of maps:
\[
\begin{tikzcd}
	{(\tilde{P},\tilde{Q})} \arrow[d,"f_k"] & T* \arrow[r, "\res"] \arrow[d,"f_k"] \arrow[l, "\uncrowd"] 
	& \mathbf{h} \arrow[r, "\star"] \arrow[d,"f_k^\star"] & {(P,Q)} \arrow[d,"f_k"] \\
	{(f_k(\tilde{P}),\tilde{Q})}      & f_k(T*) \arrow[l, "\uncrowd"] \arrow[r, "\res"]      
	& f^{\star}_k(\mathbf{h}) \arrow[r, "\star"] & {(P,f_k(Q))}.   
\end{tikzcd}
\]
By Lemma~\ref{lem:uncrowdintertwine}, the left square commutes.
By Theorem \ref{thm: star} the center square 
commutes. By Proposition~\ref{proposition.P tableaux} and Theorem~\ref{thm: star-crystal.star-insertion} the right
square commutes. Hence it suffices to prove that $Q=\tilde{P}$ when $T*$ is a lowest weight element in the crystal.

Suppose $T*\in\SVT^m(\lambda)$ is of lowest weight with $\wt(T*)=\mathbf{a}$ and $\ext(T*)=\ell$. Then the decreasing 
factorization $\mathbf{h} \in \mathcal{H}^{m,\star}$ is lowest weight by Theorem \ref{thm: star}. 
By Remark~\ref{remark.shape}, $P$ and hence $Q$ has to be of shape $\nu = \sort(\mathbf{a})$.
By Theorem~\ref{thm: star-crystal.star-insertion}, $Q$ is the unique lowest weight element in $\SSYT^m$ of
shape $\nu$.
	
Consider the uncrowding operator on $T*$ and record each tableau during the process of uncrowding as in 
Definition~\ref{def: uncrowd} by a sequence of set-valued tableaux $T*=\tilde{P}_0 \rightarrow \tilde{P}_1 \rightarrow 
\dots \rightarrow \tilde{P}_{\ell} =\tilde{P}$. Since $T*$ is of lowest weight, so are all the $\tilde{P}_i$. Furthermore, all
$\tilde{P}_i$ have the same weight $\mathbf{a}$. Let $(P_i,Q_i)=\star \circ \res(\tilde{P}_i)$. For all 
$0\leqslant i\leqslant \ell$, $Q_i$ is the unique lowest weight element in $\SSYT^m$ of shape $\nu$. Hence in 
particular $Q_i=Q$ for all $0\leqslant i\leqslant \ell$. By Proposition~\ref{prop: res.star-insertion}, $Q=Q_\ell=\tilde{P}$,
proving the claim for straight shapes.

Now take $T\in\SVT^m(\lambda/\mu)$ and construct $T*$ from $T$ by adding $\ell(\mu)$ 
to each entry and filling in the cells of $\mu$ with $T_\mu$.  Note that $T*$ is a set-valued tableaux of shape $\lambda$.
Let $(P,Q)=\star\circ\res(T*)$ and $(P*,Q*)=\tilde\star\circ\res(T)$.
Since $\res(T*) =\res(T)\res(T_\mu)$, Lemma~\ref{lem:sss} implies that $Q*=Q/T_\mu$.
On the other hand, since $T*$ has straight shape, the preceding paragraph gives that $\uncrowd(T*)=(Q,\tilde Q)$ for 
some $\tilde Q$.
We then note that $\uncrowd(T)$ and $\uncrowd(T*)$ are identical on cells of $\lambda/\mu$ up to a shift of the entries
by $\ell(\mu)$; in particular, applying $\uncrowd$ to $T*$ does not involve any cell of $\mu$ since none of these are 
multicells and their entries are the smallest $\ell(\mu)$ letters.
\end{proof}

\section{Results on the non-fully-commutative case}
\label{section.outlook}

In this section, we discuss some aspects when we generalize to the non-fully-commutative case. In 
Section~\ref{section.n=3}, we describe a local crystal on $\mathcal{H}^m(3)$.
In Section~\ref{section.nonlocal}, we show that under very mild assumptions it is not possible to expect
a local crystal for $n>3$.

\subsection{The case $n=3$}
\label{section.n=3}

We provide a description of a type $A_{m-1}$ crystal structure on $\mathcal{H}^m(3)$.

\begin{definition}\label{def: pairing.process}
Let $\mathbf{h}=h^m h^{m-1}\dots h^2 h^1\in\mathcal{H}^m(3)$.
Fix $1\leqslant k<m$. Define the \defn{pairing process} of $\mathbf{h}$ and the number of pairs in
$h^{k-1}\dots h^{j+1} h^j$, denoted $p([j,k-1])$, recursively as follows:
\begin{enumerate}
\item The empty factorization, denoted $\emptyset$, has no pairs and $p(\emptyset)=0$.
\item If $p([1,j-1])$ is defined for all $1\leqslant j\leqslant k$,
then we have $p([j,k-1])=p([1,k-1])-p([1,j-1])$. 
\item If $h^k=(\;)$, then set $p([1,k])=p([1,k-1])$.
\item Otherwise, if $h^k=(21)$, pair the 2 with the 1 in $h^k$ and set $p([1,k])=p([1,k-1])+1$.
\item Otherwise, if $h^k=(2)$ and $p([1,k-1])$ is even, ignoring all previously paired letters, locate the leftmost 
unpaired letter in $h^{k-1}\dots h^2 h^1$.
	\begin{enumerate}
	\item If this letter is in $h^j=(1)$ and $p([j+1,k-1])$ is even, then pair the 2 in $h^k$ with the 1 in $h^j$ 
	and set $p([1,k])=p([1,k-1])+1$.
	\item If this letter is in $h^j=(2)$ and $p([j+1,k-1])$ is odd, then pair the 2 in $h^k$ with the 2 in $h^j$ 
	and set $p([1,k])=p([1,k-1])+1$.
	\item Else, set $p([1,k])=p([1,k-1])$.
	\end{enumerate}
\item Otherwise, if $h^k=(1)$ and $p([1,k-1])$ 
is odd, ignoring all previously paired letters, locate the leftmost unpaired letter in $h^{k-1}\dots h^2 h^1$.
	\begin{enumerate}
	\item If this letter is in $h^j=(2)$ and $p([j+1,k-1])$ is even, then pair the 1 in $h^k$ with the 2 in $h^j$ 
	and set $p([1,k])=p([1,k-1])+1$.
	\item If this letter is in $h^j=(1)$ and $p([j+1,k-1])$ is odd, then pair the 1 in $h^k$ with the 1 in $h^j$ 
	and set $p([1,k])=p([1,k-1])+1$.
	\item Else, set $p([1,k])=p([1,k-1])$.
	\end{enumerate}
\item Else, set $p([1,k])=p([1,k-1])$.
\end{enumerate}
\end{definition}

\begin{example}
Let $m=8$ and consider $\mathbf{h}=(\;)(2)(\;)(21)(1)(1)(2)(21)\in \mathcal{H}^8(3)$. The pairing process results 
in $(\;)(2)(\;)(\underbracket{21})(1)(\underbracket{1)(2})(\underbracket{21})$, where the paired letters are indicated 
with braces. Hence, we have the following values of $p([1,k])$ for $1\leqslant k\leqslant 8$: 0, 1, 1, 2, 2, 3, 3, 3. 
Note that the letters in the fourth and seventh factors are left unpaired.

Similarly, if we take $\mathbf{h}=(\;)(2)(2)(21)(2)(1)(21)(21)\in \mathcal{H}^8(3)$, we obtain
$(\;)(\underbracket{2)(\overbracket{2)(\underbracket{21})(2})(1})(\underbracket{21})(\underbracket{21})$. 
Thus, we have the following values of $p([1,k])$ for $1\leqslant k\leqslant 8$: 0, 1, 2, 2, 2, 3 ,4, 5. In this case
all the letters in $\mathbf{h}$ are paired.
\end{example}

\begin{definition}
Let $\mathbf{h}=h^m\dots h^2 h^1 \in \mathcal{H}^m(3)$. The crystal operator $f_i$ for $1\leqslant i<m$
on $\mathbf{h}$ is defined as follows. The operator $f_i$ only depends on $h^{i+1}h^i$ and the parity of $p([1,i-1])$ 
of Definition~\ref{def: pairing.process}. In the following cases, we indicate only the changes in $h^{i+1}h^i$ 
under $f_i$ as the remainder of $\mathbf{h}$ remains invariant:
\begin{enumerate}
\item $(21)(x)  \xrightarrow{i} 0$, where $(x)\in\{(\;),(1),(2),(21)\}$,
\item $(x)(\;)  \xrightarrow{i} 0$, where $(x)\in\{(\;),(1),(2),(21)\}$,
\item $(x)(x)  \xrightarrow{i} 0$, where $(x)\in\{(\;),(1),(2)\}$,
\item $(1)(21) \xrightarrow{i} (21)(2)$,
\item $(2)(21)  \xrightarrow{i} (21)(1)$,
\item $(\;)(x) \xrightarrow{i} (x)(\;)$, where $(x)\in\{(1),(2)\}$,
\item $(\;)(21) \xrightarrow{i} (2)(1) \xrightarrow{i}  (21)(\;)$, if $p([1,i-1])$ is even,
\item $(\;)(21) \xrightarrow{i} (1)(2) \xrightarrow{i}  (21)(\;)$, if $p([1,i-1])$ is odd.
\end{enumerate}
The operator $e_i$ is defined similarly. One reverses the changes introduced in cases (4) to (8) 
and annihilates $\mathbf{h}$ when the following occurs at $h^{i+1}h^i$:
\begin{enumerate}
\item[(1)'] $(x)(21)  \xrightarrow{i} 0$, where $(x)\in\{(\;),(1),(2),(21)\}$,
\item[(2)'] $(\;)(x)  \xrightarrow{i} 0$, where $(x)\in\{(\;),(1),(2),(21)\}$,
\item[(3)'] $(x)(x)  \xrightarrow{i} 0$, where $(x)\in\{(\;),(1),(2)\}$.
\end{enumerate}

\begin{figure}\begin{tikzpicture}[>=latex,line join=bevel,xscale=0.7,yscale=0.7,every node/.style={scale=0.7}]
\node (node_9) at (265.0bp,221.5bp) [draw,draw=none] {$\left(2\right)\left(21\right)\left(2\right)$};
  \node (node_8) at (109.0bp,150.5bp) [draw,draw=none] {$\left(21\right)\left(\;\right)\left(21\right)$};
  \node (node_7) at (265.0bp,150.5bp) [draw,draw=none] {$\left(21\right)\left(1\right)\left(2\right)$};
  \node (node_6) at (265.0bp,292.5bp) [draw,draw=none] {$\left(2\right)\left(1\right)\left(21\right)$};
  \node (node_5) at (69.0bp,8.5bp) [draw,draw=none] {$\left(2, 1\right)\left(21\right)\left(\;\right)$};
  \node (node_4) at (189.0bp,292.5bp) [draw,draw=none] {$\left(1\right)\left(1\right)\left(21\right)$};
  \node (node_3) at (59.0bp,221.5bp) [draw,draw=none] {$\left(1\right)\left(2\right)\left(21\right)$};
  \node (node_2) at (189.0bp,221.5bp) [draw,draw=none] {$\left(1\right)\left(21\right)\left(2\right)$};
  \node (node_1) at (189.0bp,150.5bp) [draw,draw=none] {$\left(21\right)\left(2\right)\left(2\right)$};
  \node (node_10) at (69.0bp,79.5bp) [draw,draw=none] {$\left(21\right)\left(2\right)\left(1\right)$};
  \node (node_11) at (59.0bp,292.5bp) [draw,draw=none] {$\left(\;\right)\left(21\right)\left(21\right)$};
  \node (node_0) at (29.0bp,150.5bp) [draw,draw=none] {$\left(1\right)\left(21\right)\left(1\right)$};
  \draw [blue,->] (node_3) ..controls (50.688bp,201.83bp) and (42.61bp,182.71bp)  .. (node_0);
  \definecolor{strokecol}{rgb}{0.0,0.0,0.0};
  \pgfsetstrokecolor{strokecol}
  \draw (56.5bp,186.0bp) node {$1$};
  \draw [blue,->] (node_8) ..controls (97.858bp,130.72bp) and (86.937bp,111.34bp)  .. (node_10);
  \draw (102.5bp,115.0bp) node {$1$};
  \draw [red,->] (node_2) ..controls (189.0bp,202.04bp) and (189.0bp,183.45bp)  .. (node_1);
  \draw (197.5bp,186.0bp) node {$2$};
  \draw [red,->] (node_3) ..controls (73.003bp,201.62bp) and (86.842bp,181.96bp)  .. (node_8);
  \draw (98.5bp,186.0bp) node {$2$};
  \draw [red,->] (node_9) ..controls (265.0bp,202.04bp) and (265.0bp,183.45bp)  .. (node_7);
  \draw (273.5bp,186.0bp) node {$2$};
  \draw [blue,->] (node_10) ..controls (69.0bp,60.042bp) and (69.0bp,41.449bp)  .. (node_5);
  \draw (77.5bp,44.0bp) node {$1$};
  \draw [red,->] (node_11) ..controls (59.0bp,273.04bp) and (59.0bp,254.45bp)  .. (node_3);
  \draw (67.5bp,257.0bp) node {$2$};
  \draw [blue,->] (node_6) ..controls (265.0bp,273.04bp) and (265.0bp,254.45bp)  .. (node_9);
  \draw (273.5bp,257.0bp) node {$1$};
  \draw [red,->] (node_0) ..controls (40.142bp,130.72bp) and (51.063bp,111.34bp)  .. (node_10);
  \draw (62.5bp,115.0bp) node {$2$};
  \draw [blue,->] (node_4) ..controls (189.0bp,273.04bp) and (189.0bp,254.45bp)  .. (node_2);
  \draw (197.5bp,257.0bp) node {$1$};
\end{tikzpicture}

\caption{The crystal graph for $\mathcal{H}^3(3)$ restricted to decreasing factorizations with four letters.}
\label{fig: 121.crystal}
\end{figure}
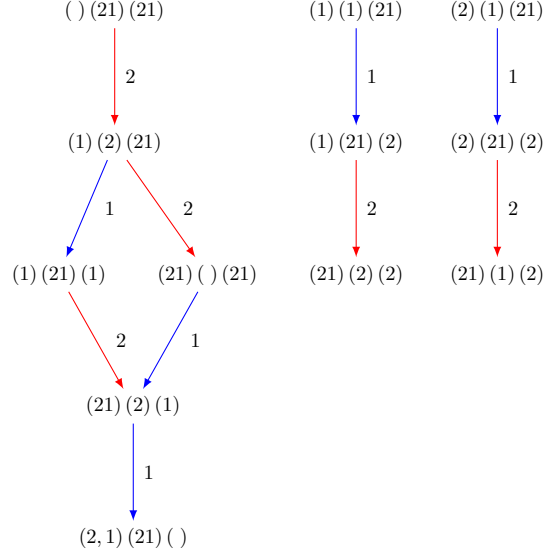

Similar to Definition \ref{def: star}, the weight map is defined as $\wt(\mathbf{h})=(\len(h^1),\len(h^2),\dots,\len(h^m))$.
Meanwhile, $\varphi_i(\mathbf{h})$ (resp. $\varepsilon_i(\mathbf{h})$) is defined to be the largest nonnegative integer
$k$ such that $f_i^{k}(\mathbf{h})\neq 0$ (resp. $e_i^{k}(\mathbf{h})\neq 0$).
\end{definition}

It is not difficult to check that the operators $f_i$ and $e_i$ defined above preserve the relation $\equiv_{\mathcal{H}_0}$ on $\mathcal{H}^m(3)$ whenever they do not annihilate the decreasing factorizations.
Furthermore, the structure above defines an abstract, seminormal 
$A_{m-1}$ crystal on $\mathcal{H}^m(3)$.

We note that one may also verify that the crystal is a Stembridge crystal by checking that the axioms formulated 
in~\cite{Stembridge.2003} are satisfied. Figure~\ref{fig: 121.crystal} displays the crystal graph on $\mathcal{H}^3(3)$
restricted to decreasing factorizations that use exactly 4 letters.

\subsection{Nonlocality}
\label{section.nonlocal}

In this subsection, we show that it is impossible to construct a crystal on $\mathcal{H}^m$ with the following
properties for $f_i$:
\begin{enumerate}
	\item $f_i$ only changes the $i$-th and $(i+1)$-th decreasing factors;
	\item $f_i$ is determined by the first $(i+1)$ factors;
	\item $f_i(\mathbf{h})\equiv_{\mathcal{H}_0}\mathbf{h}$ and $\ext[f_i(\mathbf{h})]=\ext(\mathbf{h})$, for all $\mathbf{h}\in\mathcal{H}^m$ with $f_i(\mathbf{h})\neq 0$.
\end{enumerate}
Let $\mathbf{h}_1=h_1^m\dots h_1^2 h_1^1\in\mathcal{H}^m$ and suppose that $f_i(\mathbf{h}_1)\neq 0$. 
If we write $f_i(\mathbf{h}_1)=h_2^m\dots h_2^2 h_2^1$, then the above assumptions imply that $h_1^{i+1}h_1^i\dots h^1_1\equiv_{\mathcal{H}_0} h_2^{i+1}h_2^i\dots h^1_2$. Obviously the crystal on $\mathcal{H}^m(3)$ defined in Section~\ref{section.n=3}
satisfies these assumptions.

Suppose that a crystal structure with the above assumptions exists on $\mathcal{H}^4(4)$. Consider the Schur expansion
of the stable Grothendieck polynomial in 4 variables for $w=12132$:
\[
	\mathfrak{G}_{12132}(x_1,x_2,x_3,x_4;\beta) = s_{221}+ \beta(2s_{222}+3s_{2211})
	+ \beta^2(6s_{2221}+6s_{22111}) + \cdots.
\]
(Note that $s_{22111}$ is zero in four variables and hence could be omitted).
The linear term in $\beta$ implies that there are two connected components with highest weight $(2,2,2,0)$ 
(lowest weight $(0,2,2,2)$) for the crystal $\mathcal{H}^4(4)$ with excess 1. All decreasing factorizations mentioned 
below are those of $w=12132$ with 4 factors and excess 1.

There are two decreasing factorizations of weight $(2,2,2,0)$:  $(\;)(21)(21)(32)$ and $(\;)(21)(32)(32)$. Focus on the 
connected component with highest weight $(\;)(21)(32)(32)$ and try to complete the crystal graph from top to bottom. Since 
the only decreasing factorization of weight $(2,2,1,1)$ with the first and second factors both being $(32)$ is 
$(2)(1)(32)(32)$, we can compute the action of $f_3$ on this highest weight element. By some similar arguments 
we can fill in part of the crystal graph as indicated in Figure \ref{fig: 12132.crystal.graph} with the above assumptions. The dashed spaces are undetermined.
\begin{figure}
	\begin{tikzpicture}[>=latex,line join=bevel,xscale=1.15,yscale=0.52,every node/.style={scale=1}]
	\node (node_9) at (187.5bp,283.5bp) [draw,draw=none] {$(21)(\;)(32)(32)$};
	\node (node_8) at (187.5bp,145.5bp) [draw,draw=none] {$(21)(\_\;\_)(\;)(32)$};
	\node (node_7) at (55.5bp,283.5bp) [draw,draw=none] {$(2)(\_\;\_)(2)(32)$};
	\node (node_6) at (55.5bp,145.5bp) [draw,draw=none] {$(21)(\_)(32)(3)$};
	\node (node_5) at (55.5bp,214.5bp) [draw,draw=none] {$(2)(\_\;\_)(32)(3)$};
	\node (node_4) at (121.5bp,352.5bp) [draw,draw=none] {$(2)(1)(32)(32)$};
	\node (node_3) at (121.5bp,7.5bp) [draw,draw=none] {$(21)(\_\;\_)(32)(\;)$};
	\node (node_2) at (121.5bp,76.5bp) [draw,draw=none] {$(21)(\_\;\_)(3)(2)$};
	\node (node_1) at (121.5bp,421.5bp) [draw,draw=none] {$(\;)(21)(32)(32)$};
	\node (node_0) at (185.5bp,214.5bp) [draw,draw=none] {$(21)(\_)(\_)(32)$};
	\draw [green,->] (node_1) ..controls (121.5bp,403.44bp) and (121.5bp,384.5bp)  .. (node_4);
	\definecolor{strokecol}{rgb}{0.0,0.0,0.0};
	\pgfsetstrokecolor{strokecol}
	\draw (130.0bp,387.0bp) node {$3$};
	\draw [red,->] (node_6) ..controls (73.364bp,126.82bp) and (93.078bp,106.21bp)  .. (node_2);
	\draw (105.0bp,111.0bp) node {$\textcolor{red}{\mathbf{2}}$};
	\draw [blue,->] (node_2) ..controls (121.5bp,58.442bp) and (121.5bp,39.496bp)  .. (node_3);
	\draw (130.0bp,42.0bp) node {$1$};
	\draw [green,->] (node_5) ..controls (55.5bp,196.44bp) and (55.5bp,177.5bp)  .. (node_6);
	\draw (64.0bp,180.0bp) node {$3$};
	\draw [green,->] (node_7) ..controls (91.931bp,264.16bp) and (135.37bp,241.11bp)  .. (node_0);
	\draw (145.0bp,249.0bp) node {$3$};
	\draw [red,->] (node_4) ..controls (103.64bp,333.82bp) and (83.922bp,313.21bp)  .. (node_7);
	\draw (105.0bp,318.0bp) node {$2$};
	\draw [blue,->] (node_8) ..controls (169.64bp,126.82bp) and (149.92bp,106.21bp)  .. (node_2);
	\draw (171.0bp,111.0bp) node {$1$};
	\draw [blue,->] (node_7) ..controls (55.5bp,265.44bp) and (55.5bp,246.5bp)  .. (node_5);
	\draw (64.0bp,249.0bp) node {$1$};
	\draw [red,->] (node_0) ..controls (186.02bp,196.44bp) and (186.57bp,177.5bp)  .. (node_8);
	\draw (196.0bp,180.0bp) node {$2$};
	\draw [blue,->] (node_0) ..controls (149.07bp,195.16bp) and (105.63bp,172.11bp)  .. (node_6);
	\draw (145.0bp,180.0bp) node {$1$};
	\draw [red,->] (node_9) ..controls (186.98bp,265.44bp) and (186.43bp,246.5bp)  .. (node_0);
	\draw (196.0bp,249.0bp) node {$2$};
	\draw [green,->] (node_4) ..controls (139.36bp,333.82bp) and (159.08bp,313.21bp)  .. (node_9);
	\draw (171.0bp,318.0bp) node {$3$};
	\end{tikzpicture}
	\caption{Partial filling of the connected component of $\mathcal{H}^4(3)$ containing highest weight element $(\;)(21)(32)(32)$.}
	\label{fig: 12132.crystal.graph}
\end{figure}
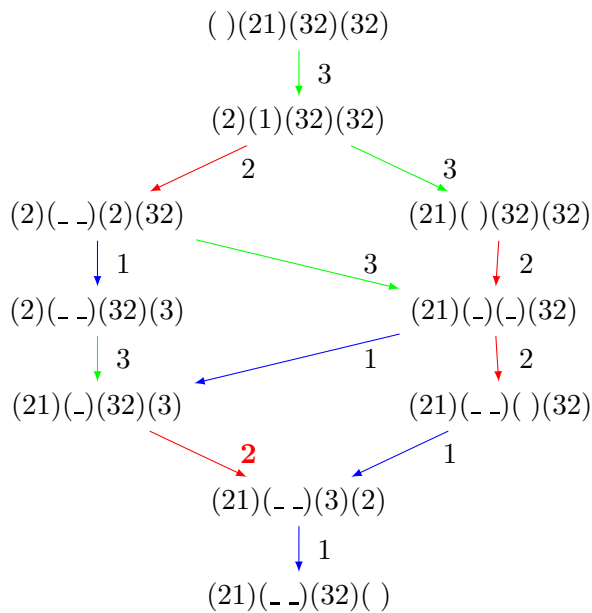

Yet note that the red $f_2$ highlighted in the graph changed the first factor from $(3)$ to $(2)$.
Hence, Condition (1) is violated, providing a counterexample that crystals with the above conditions
always exist on $\mathcal{H}^m(n)$ for $n>3$. 

\bibliographystyle{alpha}
\bibliography{main}{}

\newcommand{\etalchar}[1]{$^{#1}$}
\begin{thebibliography}{BKS{\etalchar{+}}08}

\bibitem[BKS{\etalchar{+}}08]{BKSRY.2008}
Anders~Skovsted Buch, Andrew Kresch, Mark Shimozono, Harry Tamvakis, and
  Alexander Yong.
\newblock Stable {G}rothendieck polynomials and {$K$}-theoretic factor
  sequences.
\newblock {\em Math. Ann.}, 340(2):359--382, 2008.

\bibitem[BM12]{BandlowMorse.2012}
Jason Bandlow and Jennifer Morse.
\newblock Combinatorial expansions in {$K$}-theoretic bases.
\newblock {\em Electron. J. Combin.}, 19(4):Paper 39, 27, 2012.

\bibitem[BS17]{BumpSchilling.2017}
Daniel Bump and Anne Schilling.
\newblock {\em Crystal bases}.
\newblock World Scientific Publishing Co. Pte. Ltd., Hackensack, NJ, 2017.
\newblock Representations and combinatorics.

\bibitem[Buc02]{Buch.2002}
Anders~Skovsted Buch.
\newblock A {L}ittlewood-{R}ichardson rule for the {$K$}-theory of
  {G}rassmannians.
\newblock {\em Acta Math.}, 189(1):37--78, 2002.

\bibitem[CP19]{ChanPflueger.2019}
Melody Chan and Nathan Pflueger.
\newblock Combinatorial relations on skew {S}chur and skew stable
  {G}rothendieck polynomials.
\newblock {\em arXiv preprint arXiv:1909.12833}, 2019.

\bibitem[FG98]{FominGreene.1998}
Sergey Fomin and Curtis Greene.
\newblock Noncommutative {S}chur functions and their applications.
\newblock volume 193, pages 179--200. 1998.
\newblock Selected papers in honor of Adriano Garsia (Taormina, 1994).

\bibitem[FK94]{FominKirillov.1994}
Sergey Fomin and Anatol~N. Kirillov.
\newblock Grothendieck polynomials and the {Y}ang-{B}axter equation.
\newblock In {\em Formal power series and algebraic combinatorics/{S}\'{e}ries
  formelles et combinatoire alg\'{e}brique}, pages 183--189. DIMACS,
  Piscataway, NJ, 1994.

\bibitem[Ful96]{Fulton.1996}
William Fulton.
\newblock {\em Young Tableaux:}.
\newblock London Mathematical Society Student Texts. Cambridge University
  Press, 1996.
\newblock With Applications to Representation Theory and Geometry.

\bibitem[Len00]{Lenart.2000}
Cristian Lenart.
\newblock Combinatorial aspects of the {$K$}-theory of {G}rassmannians.
\newblock {\em Ann. Comb.}, 4(1):67--82, 2000.

\bibitem[Len04]{Lenart.2004}
Cristian Lenart.
\newblock A unified approach to combinatorial formulas for {S}chubert
  polynomials.
\newblock {\em J. Algebraic Combin.}, 20(3):263--299, 2004.

\bibitem[LP07]{LamPylyavskyy.2007}
Thomas Lam and Pavlo Pylyavskyy.
\newblock Combinatorial {H}opf algebras and {$K$}-homology of {G}rassmannians.
\newblock {\em Int. Math. Res. Not. IMRN}, (24):Art. ID rnm125, 48, 2007.

\bibitem[LS82]{LascouxSchutzenberger.1982}
Alain Lascoux and Marcel-Paul Sch\"{u}tzenberger.
\newblock Structure de {H}opf de l'anneau de cohomologie et de l'anneau de
  {G}rothendieck d'une vari\'{e}t\'{e} de drapeaux.
\newblock {\em C. R. Acad. Sci. Paris S\'{e}r. I Math.}, 295(11):629--633,
  1982.

\bibitem[LS83]{LascouxSchutzenberger.1983}
Alain Lascoux and Marcel-Paul Sch\"{u}tzenberger.
\newblock Symmetry and flag manifolds.
\newblock In {\em Invariant theory ({M}ontecatini, 1982)}, volume 996 of {\em
  Lecture Notes in Math.}, pages 118--144. Springer, Berlin, 1983.

\bibitem[MPS18]{MPS.2018}
Cara Monical, Oliver Pechenik, and Travis Scrimshaw.
\newblock Crystal structures for symmetric {G}rothendieck polynomials.
\newblock {\em arXiv preprint arXiv:1807.03294}, 2018.

\bibitem[MS16]{MorseSchilling.2016}
Jennifer Morse and Anne Schilling.
\newblock Crystal approach to affine {S}chubert calculus.
\newblock {\em Int. Math. Res. Not. IMRN}, (8):2239--2294, 2016.

\bibitem[Pat16]{Patrias.2016}
Rebecca Patrias.
\newblock Antipode formulas for some combinatorial {H}opf algebras.
\newblock {\em Electron. J. Combin.}, 23(4):Paper 4, 30, 2016.

\bibitem[PP16]{PP.2016}
Rebecca Patrias and Pavlo Pylyavskyy.
\newblock Combinatorics of {$K$}-theory via a {$K$}-theoretic
  {P}oirier-{R}eutenauer bialgebra.
\newblock {\em Discrete Math.}, 339(3):1095--1115, 2016.

\bibitem[RTY18]{RTY.2018}
Vic Reiner, Bridget~E. Tenner, and Alexander Yong.
\newblock Poset edge densities, nearly reduced words, and barely set-valued
  tableaux.
\newblock {\em J. Combin. Theory, Ser. A}, 158:66--125, 2018.

\bibitem[Sta84]{Stanley.1984}
Richard~P. Stanley.
\newblock On the number of reduced decompositions of elements of {C}oxeter
  groups.
\newblock {\em European J. Combin.}, 5(4):359--372, 1984.

\bibitem[Ste96]{Stembridge.1996}
John~R. Stembridge.
\newblock On the fully commutative elements of {C}oxeter groups.
\newblock {\em J. Algebraic Combin.}, 5(4):353--385, 1996.

\bibitem[Ste03]{Stembridge.2003}
John~R. Stembridge.
\newblock A local characterization of simply-laced crystals.
\newblock {\em Trans. Amer. Math. Soc.}, 355(12):4807--4823, 2003.

\end{thebibliography}

\end{document}